\documentclass[11pt,leqno,a4paper,psamsfonts]{amsart}
\usepackage[english]{babel}

\usepackage{amssymb,amsmath,amscd,enumerate,amsthm}
\usepackage[psamsfonts]{amsfonts}
\usepackage{latexsym}
\usepackage{amsbsy}
\usepackage[mathscr]{eucal}
\usepackage[latin1]{inputenc}
\usepackage{pdfsync}
\usepackage{color}

\usepackage{graphicx}
\usepackage{enumerate}
\usepackage{verbatim}
\usepackage[all,color ]{xy}

\newcommand{\C}{{\mathbb{C}}}
\newcommand{\F}{{\mathcal{F}}}
\newcommand{\W}{{\mathcal{W}}}
\newcommand{\V}{{\mathcal{V}}}
\newcommand{\G}{\mathcal{G}}

\newcommand{\I}{\mathcal{I_{F}}}
\newcommand{\It}{\mc I^{\rm tr}_{\F}}
\newcommand{\Ii}{\mc I^{\rm inv}_{\F}}

\newcommand{\p}{{\mathbb{P}}}
\newcommand{\pd}{\check{\p}}
\newcommand{\pid}{\check{\pi}}

\newcommand{\N}{\mathbb{N}}

\newcommand{\Z}{{\mathbb{Z}}}

\newcommand{\leg}{{\rm Leg}\,}

\newcommand{\mf}[1]{{\mathfrak{#1}}}
\newcommand{\mr}[1]{{\mathrm{#1}}}
\newcommand{\mc}[1]{{\mathcal{#1}}}
\newcommand{\mb}[1]{{\mathbb{#1}}}
\newcommand{\wt}[1]{{\widetilde{#1}}}
\newcommand{\wh}[1]{{\widehat{#1}}}
\newcommand{\ov}[1]{{\overline{#1}}}

\newtheorem{teo}{Theorem}[section]
\newtheorem{thmm}{Theorem}

\newtheorem{corm}[thmm]{Corollary}

\newtheorem{lema}[teo]{Lemma}

\newtheorem{prop}[teo]{Proposition}
\newtheorem{defin}[teo]{Definition}
\newtheorem{cor}[teo]{Corollary}
\newtheorem{obs2}[teo]{Remark}
\newtheorem{q2}[teo]{Question}
\newtheorem{p2}[teo]{Problem}
\newtheorem{ex2}[teo]{Example}
\newenvironment{obs}{\begin{obs2}\rm}{
\end{obs2}}
\newenvironment{ex}{\begin{ex2}\rm}{\qed\end{ex2}}
\newenvironment{question}{\begin{q2}}{\end{q2}}

\newenvironment{dem}{\begin{proof}[Proof]}{\end{proof}}

\title{Foliations and webs inducing Galois coverings}
\author{A. Beltr\'an, M. Falla Luza, D. Mar\'{\i}n and M. Nicolau}
\date{\today}
\thanks{This work was partially supported by
Pontificia Universidad Católica del Perú through its Dirección de Gestión de la Investigación, the grants 476877/2013-0 from CNPQ-Brazil, MTM2011-26674-C02-01  from the Ministry of Economy and Competitiveness of Spain and by the grant UNAB10-4E-378 co-funded by ERDF ``A way to build Europe''}


\address{Andrés Beltrán\\ Departamento de Matemática \\
 Pontificia Universidad Ca\-tó\-lica del Perú\\
Av. Universitaria 1801, 
Lima, Perú\\
abeltra@pucp.edu.pe }
\address{Maycol Falla Luza\\ Departamento de An\'alise -- IM \\
 Universidade Federal Fluminense  \\
 M\'ario Santos Braga s/n -- Niter\'oi, 24.020-140 RJ Brasil\\ 
 maycolfl@gmail.com}
\address{David Marín\\ Departament de Matem\`{a}tiques \\ Universitat Aut\`{o}\-noma de Bar\-ce\-lo\-na \\ E-08193  Bellaterra (Barcelona) Spain \\
 davidmp@mat.uab.es}
\address{Marcel Nicolau\\ Departament de Matem\`{a}tiques \\ Universitat Aut\`{o}noma de Barcelona \\ E-08193  Bellaterra (Barcelona) Spain \\
 nicolau@mat.uab.es}

\begin{document}
\maketitle

\begin{abstract}
We introduce the notion of Galois holomorphic foliation on  the complex projective space as that of foliations whose Gauss map is a Galois covering when restricted to an appropriate Zariski open subset.  First, we establish general criteria assuring that a rational map between projective manifolds of the same dimension defines a Galois covering.
Then, these criteria are used to give a geometric characterization of Galois foliations in terms of their inflection divisor and their singularities. We also characterize Galois foliations on $\mathbb P^2$ admitting continuous symmetries, obtaining a complete classification of Galois homogeneous foliations. 
\end{abstract}


\section{Introduction}

In this article we introduce the notion of Galois holomorphic foliation on the complex projective space. 
Our main objective is to establish general criteria characterizing those foliations that are Galois.

Let $\mathcal F$ be a holomorphic foliation in the complex projective plane $\mathbb P^2$. 
The degree $\deg \mathcal F$
of the foliation  is the number of tangencies  of $\mathcal F$ with a generic 
line of $\mathbb P^2$ and the Gauss map $\mc G_{\F}:\p^{2}\dasharrow\pd^{2}$ of the foliation, 
sending $x\in\p^{2}$ into the  tangent line of $\mathcal F$ at $x$,
is a well defined rational map whose indeterminacy points are just the singularities of the foliation. 
If the foliation is non degenerated then the restriction of $\mc G_{\F}$ 
to a suitable Zariski open subset $W$ of $\p^{2}$ is a covering map of order 
$\deg\mathcal F>0$.

The determination of finite subgroups of the Cremona group $\mr{Bir}(\p^{2})$
of birational transformations of $\p^2$ is a classical topic, nevertheless it continues to be an active field of current research
(cf. \cite{Beauville,Dol}). In \cite{CD}, Cerveau and Deserti addressed the problem of finding non-trivial 
birational deck transformations of the covering associated to a foliation $\F$, that is, birational maps 
$\tau:\p^{2}\dasharrow\p^{2}$ 
fulfilling $\mc G_{\F}\circ\tau=\mc G_{\F}$. Their aim was to construct  periodic 
elements of $\operatorname{Bir}(\mathbb P^2)$ in an effective way. 
In particular they associated a birational involution 
to each quadratic foliation and trivolutions to certain classes of  cubic foliations. In all these cases the
restriction of $\mc G_{\F}$ to the Zariski open set $W$ is necessarily a Galois covering.
It is therefore a natural question to determine the Galois foliations of $\p^2$, that is those foliations in~$\mathbb P^2$ whose Gauss map defines a Galois covering. And this is  the original purpose of this article.
We will see that for every Galois foliation the deck transformations of its Gauss map are birational.
This fact provides non-trivial explicit realizations of the symmetry groups of regular polyhedra into  the Cremona group (cf. Example~\ref{cremona}).

We are specially concerned with the problem of characterizing Galois foliations on $\p^{2}$ in terms of its geometric elements. In this direction, our main results are Theorems~\ref{D} and~\ref{E} and Corollary~\ref{B'} stated below. In order to prove them, we first consider the more general setting of arbitrary dominant rational maps $f\colon X \dasharrow Y$ between complex connected projective manifolds of the same dimension.

Such a rational map is called Galois if the field extension $f^{*}\colon\C(Y)\hookrightarrow\C(X)$ is Galois or equivalently if the group $\mr{Deck}(f):=\{\phi\in\mr{Bir}(X)\,|\,f\circ\phi=f\}$ acts transitively on the fibers of $f$. It is also equivalent to say that $f$ induces a topological Galois  covering by restriction to suitable Zariski open subsets.
On the other hand such a rational map $f\colon X\dasharrow Y$ admits a canonical birational model $\rho\colon N\to Y$, which is a finite branched covering, obtained by applying Stein factorization to a desingularization of $f$.
Thus, $f$ is Galois if and only if $\rho$ is a Galois branched covering and in this case the deck transformation group of $\rho$, which is birationally conjugated to $\mr{Deck}(f)$, consists in automorphisms of $N$.

There is a natural notion of branched covering of regular type by asking that the ramification indices are constant along the fibers. This notion, which translates naturally to rational maps, is of semi-local nature and it is implied by the global property of being Galois. One of our main results states that these two notions are equivalent when the source  is the projective space.

\begin{thmm}\label{A}
A dominant rational map $f:\p^{n}\dasharrow Y$ is Galois if and only if it is of regular type.
\end{thmm}

The proof of this theorem is based on the dimensional reduction provided by Theorem~\ref{dim-red} which implies that the character Galois can be tested by restriction to appropriated hyperplane curves.

\medskip

We address the  natural question of describing the space of Galois maps in a given family of dominant rational maps. In this direction we have the following result that is a particular case of Theorem~\ref{TC}.

\begin{thmm}\label{B}
Consider a family $f:X\times T\dasharrow \p^{n}\times T$ of dominant rational maps of constant topological degree parametrized by $T$. Then 
$$\mr{Gal}(T):=\{t\in T\,|\, f_{t} \text{ is Galois}\}$$
is a Zariski closed subset of $T$ and the  Galois group is constant along each connected component of $\mr{Gal}(T)$.
\end{thmm}

We introduce two new combinatorial invariants of such  dominant rational maps: the weighted branching type (Definitions~\ref{bt} and~\ref{wbt}) and the genus (Definition~\ref{gf}). In Proposition~\ref{TB} and Theorem~\ref{TC} it is
shown that they are generically constant along the irreducible components of $\mr{Gal}(T)$. These invariants and the corresponding Galois groups are used to distinguish the different components of $\mr{Gal}(T)$.

\medskip

In Section~\ref{SFW} we turn back to our original motivation of studying the Gauss map $\G_{\F}:\p^{n}\dasharrow\pd^{n}$ of a foliation $\F$ on $\p^{n}$. We say that $\F$ is Galois if its Gauss map $\G_{\F}$ is a Galois rational map.
By duality, a foliation $\F$ on $\p^{n}$ induces a $d$-web $\leg\F$ (called Legendre transform of $\F$) on $\pd^{n}$ where $d=\deg\G_{\F}$. This web can be thought as the direct image of the foliation $\F$ by its Gauss map. This motivates to consider the direct image of a foliation $\F$ by a rational map $f\colon X\dasharrow Y$  which is a well defined web $f_{*}\F$ on $Y$ whenever $\F$ is in general position with respect to $f$, see Definition~\ref{gp} and Proposition~\ref{direct}. It turns out that the monodromy of the web $f_{*}\F$ is naturally identified to the monodromy of the map $f$. This allows us to formulate the following characterization of Galois rational maps in terms of decomposability of webs.

\begin{thmm}\label{C}
Let $f:X\dasharrow Y$ be a dominant rational map between complex projective manifolds of the same dimension and let $\F$ be a foliation on $X$ in general position with respect to $f$. Then $f$ is Galois if and only if the web $f^{*}f_{*}\F$ is totally decomposable. 
\end{thmm}

In the case of a foliation $\F$ on $\p^{n}$ this result states that $\F$ is Galois if and only if the web $\G^{*}_{\F}\leg\F$ is totally decomposable. 
By means of this criterion and a dimensional reduction we are able to decide if certain families of foliations are Galois or not. In particular, we exhibit Galois foliations in every dimension and with degree arbitrarily large (cf. Corollary~\ref{n-dim-ex}).

\medskip

In section~\ref{S6} we focus on the study of Galois foliations in $\p^{2}$.
One of our main goals is to give a  characterization of Galois foliations $\F$ in terms of geometric elements, more concretely in terms of the inflection divisor $\I$ (whose definition is recalled in subsection~\ref{7.3}) and the singular locus $\Sigma_{\F}$. We decompose $\I=\mc I_{\F}^{\mr{inv}}+\mc I_{\F}^{\mr{tr}}$, where $\mc I_{\F}^{\mr{inv}}$ is given by the invariant lines of $\F$, and we set
$$\mc I_{\F}^{\mr{tr}}=\sum_{\varrho>1}(\varrho-1)\mc I_{\F}^{\varrho},$$
where $\mc I_{\F}^{\varrho}$ are reduced curves. We have that $p\in\mc I_{\F}^{\varrho}\setminus\Sigma_{\F}$ if and only if the tangency order at $p$ between the leaf of $\F$ passing through $p$ and its tangent line is $\varrho$. It turns out that $\G_{\F}$ has ramification index $\varrho$ along $\mc I_{\F}^{\varrho}$ and this is a local property. 
We consider the set $\Sigma_{\F}^{\mr{ram}}$ of singularities giving rise to  ramification components of the exceptional divisor after desingularizing $\G_{\F}$ and we denote by $\Sigma_{\F}^{\varrho}$ the subset of $\Sigma_{\F}^{\mr{ram}}$ consisting in those singularities whose corresponding ramification divisors are all of order $\varrho$.
The sets $\Sigma_{\F}^{\mr{ram}}$ and $\Sigma_{\F}^{\varrho}$ are characterized geometrically in Lemma~\ref{sigmad} by means of local arithmetic invariants. With these notations and using Theorem~\ref{A}  we can state the following semi-local characterization of Galois foliations:
\begin{thmm}\label{D}
A degree $d$ 
foliation $\F$ on $\p^{2}$ is Galois if and only if,  
for each  $\check{\ell}\in\pd^{2}$ such that the tangency locus $\mr{Tang}(\F,\ell)$ between $\F$ and $\ell$ has less than $d$ points, there is  $\varrho|d$, $\varrho>1$, 
such that $\mr{Tang}(\F,\ell)\subset(\mc I_{\F}^{\varrho}\setminus\Sigma_{\F})\cup\Sigma_{\F}^{\varrho}$. 
\end{thmm}

Since being Galois is a global property, in general one can not expect to obtain a fully characterization of Galois foliations only in purely local terms.  However, we are able to state two natural conditions, one sufficient and the other necessary, using only local data of $\F$:

\begin{thmm}\label{E} 
Let $\F$ be a degree $d$ foliation on $\p^{2}$
and consider the following assertions:
\begin{enumerate}[(1)]
\item $\mc I_{\F}^{\varrho}=\emptyset$ unless for $\varrho=d$ and $\Sigma_{\F}^{\mr{ram}}=\Sigma_{\F}^{d}$;
\item $\F$ is Galois;
\item $\mc I_{\F}^{\varrho}=\emptyset$ unless for $\varrho|d$ and $\Sigma_{\F}^{\mr{ram}}=\bigcup\limits_{1<\varrho|d}\Sigma_{\F}^{\varrho}$.
\end{enumerate}
Then the implications $(1)\Rightarrow(2)\Rightarrow(3)$ hold. 
\end{thmm}
Condition (1) characterizes foliations such that its Gauss map is \emph{extremal} in the sense that all its ramification indices are maximal, i.e. equal to $d$. The above theorem has the following corollary which implies that the Galois character of a foliation of prime degree can be checked by means of purely local data.

\begin{corm}\label{B'}
Let $\F$ be a foliation on $\p^{2}$ of prime degree. Then $\G_{\F}$ is Galois
if and only if $\G_{\F}$ is extremal.
\end{corm}

The set of degree $d$ foliations on $\p^{2}$ is a Zariski open subset of a projective space and we can consider the family of their Gauss maps. By applying Theorem~\ref{B} we deduce that the space $\mb G_{d}$ of degree $d$ Galois foliations is a quasi-projective variety. This raises the question  of describing its irreducible components in geometric terms. This problem is of similar nature to the study of the irreducible components of the space of codimension one foliations on $\p^{n}$ for $n\ge 3$ (cf. \cite{CLN}) and the study of the irreducible components of the space of flat webs (cf. \cite{MP}).

For each degree $d$ we present a continuous family  of Galois foliations, that include all the examples considered in \cite{CD}, and that we expect to be components of $\mb G_{d}$. We also exhibit  a degree $3$ Galois foliation that does not belong to the previous family. Looking at its genus we  show that $\mb G_{3}$ has at least two irreducible components (cf. Proposition~\ref{G3red}).

In subsection~\ref{homog} we treat the case of homogeneous foliations in $\mathbb P^2$, i.e. invariant by the flow associated to a radial vector field. Using Proposition~\ref{AutF} we can perform a dimensional reduction $\wh\G_{\F}:\p^{1}\to\p^{1}$ of $\G_{\F}:\p^{2}\dasharrow\pd^{2}$.
Then we can use the Klein classification of the 
Galois ramified coverings of $\mathbb P^1$ by itself (cf. Theorem~\ref{klein}) in terms of their Galois groups.
The left-right equivalence between rational functions on $\p^{1}$ preserve Galois property and translates into a natural action of  $\mr{PSL}_{2}(\C)\times\mr{PSL}_{2}(\C)$ on the space of homogeneous foliations. We obtain the following result.

\begin{thmm}\label{F}
The homogeneous Galois foliations of degree $d$ consists of the orbits by the left-right action of $\mr{PSL}_{2}(\C)\times\mr{PSL}_{2}(\C)$ of the following ones:
\begin{enumerate}[(1)]
\item $x^{d}\partial_{x}+y^{d}\partial_{y}$ for every $d$,
\item $(x^{n}+y^{n})^{2}\partial_{x}+(x^{n}-y^{n})^{2}\partial_{y}$ if $d=2n$ is even,
\item $(x^{4}+2i\sqrt{3}x^{2}y^{2}+y^{4})^{3}\partial_{x}+(x^{4}-2i\sqrt{3}x^{2}y^{2}+y^{4})^{3}\partial_{y}$ if $d=12$, 
\item $(x^{8}+14x^{4}y^{4}+y^{8})^{3}\partial_{x}+(xy(x^{4}-y^{4}))^{4}\partial_{y}$ if $d=24$,
\item $\scriptstyle{(x^{20}-228x^{15}y^{5}+494x^{10}y^{10}+228x^{5}y^{15}+y^{20})^{3}\partial_{x}+(xy(x^{10}+11x^{5}y^{5}-y^{10}))^{5}\partial_{y}}$ if $d=60$.
\end{enumerate}
\end{thmm}

As a consequence of this classification and Theorem~\ref{B} we deduce that~$\mb G_{d}$ has at least $2$ irreducible components if $d$ is even and it has at least $3$ irreducible components for $d=12,24,60$, as they are distinguished by their Galois groups: cyclic, dihedral, tetrahedral, octahedral and icosahedral.

We finish the article by considering foliations admitting other continuous groups of symmetries. We show that in all cases there is a dimensional reduction, analogous to the homogeneous one, which gives a characterization of Galois foliations in this setting (cf. Proposition~\ref{ghat}).

\medskip

{\it Acknowledgements.} The authors wish to thank J.V. Pereira  and T. Fassarella
for fruitful conversations. The first and second authors thank the Departament de Matemàtiques de la Universitat Autònoma de Barcelona for their hospitality and support.


\section{Branched coverings}

Along this article we will deal with morphisms between projective manifolds of the same dimension. 
Such maps turn out to be a branched coverings when restricted to 
appropriate Zariski open subsets. In this section we collect the results about branched coverings that 
will be used in the article. We begin by recalling some well-known facts about unbranched topological coverings.

\subsection{Topological coverings}

Let $\varpi:E\to B$ be a $d$-sheeted covering over a connected and locally path connected topological space $B$. Fix a base point $p_{0}\in B$ and its fibre $F=\varpi^{-1}(p_{0})=\{p_{1},\ldots,p_{d}\}$.
We consider the deck transformation group of the covering
$$D=\mr{Deck}(\varpi)=\{\tau:E\stackrel{\sim}{\longrightarrow} E\mid\varpi\circ\tau=\varpi\}$$ 
acting on $F$, on the left, by restriction. In fact, the restriction map is a monomorphism $D\hookrightarrow\mf S(F)$, where $\mf S(F)$ is the permutation group of $F$.
We also consider the monodromy anti-representation $\bar\mu:\pi_{1}(B,p_{0})\to\mf{S}(F)$ of $\varpi$ defined by $\bar\mu([\gamma])(p)=\wt \gamma_{p}(1)$ for each $p\in F$, where $\wt\gamma_{p}(t)$ is a path in $E$ starting at $\wt\gamma_{p}(0)=p$ and projecting onto $\gamma=\varpi(\wt\gamma_{p})$. The anti-morphism $\bar\mu$ defines a \emph{right} action of the fundamental group of $B$  on $F$. We define the monodromy representation $\mu:\pi_{1}(B,p_{0})\to\mf S(F)$ as the morphism $\gamma\mapsto\mu(\gamma)=\bar\mu(\gamma^{-1})$. Its image subgroup, denoted by $M=\mr{Mon}(\varpi)$,  is called the monodromy group of $\varpi$.
It is clear that if $E$ is connected, then the action of $D$ is free and the action of $M$ is transitive. Consequently,  if $E$ is connected then $|D|\le d$ and $|M|\ge d$.
Identifying $F\simeq\{1,\ldots,d\}$ we can consider both $D$ and $M$ as subgroups of the symmetric group $\mf S_{d}\simeq \mf S(F)$.
The image of $D$ inside $\mf S(F)$ can be characterized as the group of permutations of $F$ commuting with all the elements of the 
monodromy group (cf.   \cite{Cuk} or \cite{Douady}), that is
\begin{equation}\label{D=ZM}
D=Z(M)\quad\text{in}\quad\mf S(F).
\end{equation}

\begin{obs}\label{1.5}
In general, there is no inclusion between the subgroups $D$ and $M$. In fact, it follows from \eqref{D=ZM}  that their intersection $D\cap M=Z(M)\cap M=C(M)$ is the centre of $M$. Thus, $M\subset D$ if and only if $M$ is abelian.
In addition, if the covering $\varpi$ is finite and $E$ is connected  then $M$ is abelian if and only if $M=D$ because $|D|\le \deg\varpi\le |M|$. 
\end{obs}

Given a connected and locally path connected space $B'$ and a continuous map $f:B'\to B$, the pull-back covering of
$\varpi:E\to B$ by the map $f$ is defined as 
$\varpi'\colon E'=E\times_{B}B'\to B'$, where 
$$
E\times_{B}B'=  \{(p,b')\in E\times B'\mid \varpi(p) = f(b')\},
$$  
and $\varpi'$ is the restriction of the natural projection $E\times B'\to B'$. 
Notice that $F$ is also the fibre of $E'$ and that $E'$ in not necessarily connected.
We shall denote by $E\times_{B}E\to E$ the pull-back covering obtained from $\varpi$ when $f= \varpi$.

\begin{defin}\label{def-gal}
A connected covering $\varpi:E\to B$ with fibre $F$, deck transformation group $D$ and 
monodromy group $M$ is said to be Galois if one of the following equivalent conditions hold:
\begin{enumerate}
\item $D$ acts transitively on $F$,
\item $M$ acts freely on $F$,
\item\label{4} the covering $E\times_{B}E\to E$ is trivial.
\end{enumerate}
In that case $M\simeq D$ and $E/D\simeq B$.
\end{defin}

From Remark~\ref{1.5} we get the following:

\begin{obs}\label{cyclic} 
If $\varpi:E\to B$ is a connected $d$-sheeted covering with abelian monodromy group $M\subset\mf S_{d}$, then $\varpi$ is Galois. In particular, if $M$ is cyclic then $M\simeq\Z_{d}$. Moreover, when $\deg\varpi$ is prime, $\varpi$ is Galois if and only if $M$ is cyclic.
\end{obs}

The following statement describes the relation between the monodromy groups $M$ and $M'$, as well as the deck transformation groups
$D$ and $D'$, of a given 
covering $E\to B$ and of its pull-back $E'\to B'$ by a continuous map.

\begin{prop}\label{1} 
Let $\varpi:E\to B$ be a covering with $E$ connected and let $\varpi'\colon E'=E\times_{B}B'\to B'$ be the pull-back covering of $\varpi$
by a continuous map  $f\colon B'\to B$. Let $M$, $M'$ and $D$, $D'$ denote the monodromy groups and the 
deck transformation groups of $\varpi$ and $\varpi'$ respectively. Then we have

\begin{enumerate}[(a)]
\item There are natural inclusions $M'\hookrightarrow M$ and $D\hookrightarrow D'$.
\item If $\varpi$ is Galois and $E'_0$ is a connected component of $E'$, then the restricted covering 
$\varpi_{0}':E_{0}'\to B'$ is also Galois. Moreover, the deck transformation group $D'_0$ of $\varpi'_0$
is naturally included in $D$. 
\item Assume that $f_\ast\colon \pi_1(B')\to \pi_1(B)$ is surjective. Then $E'$ is connected, 
$M\simeq M'$, $D\simeq D'$ and, consequently, $\varpi$ is Galois if and only if $\varpi'$ is Galois.
\end{enumerate}
\end{prop}

\begin{dem}
The map $E'=E\times_{B}B'\to E$ induced by the natural projection $E\times B'\to E$ identifies the fibre $\varpi'^{-1}(p_{0}')$ with $\varpi^{-1}(p_{0})=F$.
Then the monodromy representation of $\varpi':E'\to B'$ is the composition 
$$\mu\circ f_{*}:\pi_{1}(B',p_{0}')\to\pi_{1}(B,p_{0})\to\mf S(F).$$ 
This implies assertions (a) and (c) using the characterization of $D$ given by identity~\eqref{D=ZM}.

Let $E'_0$ be a connected component of $E'$ and denote by $F_0$ the intersection of $E'_0$ with the fibre $F=\varpi'^{-1}(p_{0}')\equiv\varpi^{-1}(p_{0})$;
that is, $F_0$ is the fibre of $\varpi'_0$. We notice that the components of $E'$ induce a partition of the fibre $F$ and $F_0$ is one of these components.
Because of the inclusion $D\hookrightarrow D'$, the action of $D$ on the fibre $F$ preserves that partition. In particular, for a given $\tau\in D$
one has $\tau(F_0) = F_0$ or $\tau(F_0) \cap F_0 = \emptyset$. Assume now that $\varpi$ is Galois and therefore that $D$
acts transitively on $F_0$. Given two points $p_i, p_j \in F_0$ there is a unique $\tau\in D$ such that $\tau(p_i)= p_j$.  It follows that $\tau(F_0) = F_0$, and therefore that $\tau$ 
is an element of $D_0$. We deduce that $D'_0$ acts transitively on $F_0$ and that $D'_0$ is naturally identified to $D_0\subset D$, ending the proof.
\end{dem}

\subsection{Analytic branched coverings}\label{2.2}
We consider now surjective morphisms $f\colon X\to Y$ between complex analytic spaces of the same dimension. 
Under some conditions, the restriction of such a map $f$ to appropriate dense open subsets of $X$ and $Y$ is a topological covering.

\begin{defin} Let $f\colon X\to Y$ be a morphism between complex analytic spaces of the same dimension. Then $\nabla_{f}$ will stand for 
the analytic subset of $X$ defined by 
$$\nabla_{f}:=\{x\in X\,|\, f \textrm{ is not a local biholomorphism at }x\}.$$
\end{defin}

Along the article we make use of the following conventions. If $f:X\to Y$ is a morphism between complex analytic spaces and $K$ is an arbitrary subset of $Y$, then we denote
\begin{enumerate}[$\bullet$]
\item $X_{K}:=f^{-1}(K)$ and $f_{K}$ the restriction of $f$ to $X_{K}$; in the case $K=\{p\}$ then we will denote $X_{\{p\}}$ and $f_{\{p\}}$ simply by $X_{p}$ and $f_{p}$;
\item $f^{\nu}:X^{\nu}\to Y$ the composition of the normalization $X^{\nu}\to X$ of $X$ and $f$.
\end{enumerate}

We recall the following definition.
\begin{defin}
A \emph{finite branched covering} $f:X\to Y$ is a proper finite holomorphic map from a complex normal space $X$ onto a connected complex 
manifold $Y$ whose restriction to each connected component of $X$ is surjective.
\end{defin}

\begin{obs} (a) Since the analytic space $X$ in the above definition is assumed to be normal, its connected components are irreducible.

(b) A more general definition of branched covering, not requiring the map $f$ to be finite, can also be considered (cf. \cite{Namba}).
Nevertheless, in this article we will only deal with branched coverings whose fibres are finite, even without mention.
\end{obs}

The \emph{ramification locus} of a finite branched covering $f:X\to Y$ is the analytic subset $\nabla_{f}$ of $X$ and the \emph{branching locus} (also called discriminant) of $f$ is the analytic subset  of $Y$
given by
$$\Delta_{f}:=f(\nabla_{f}).$$ 
Notice that $\nabla_{f}$ contains $\mr{Sing}(X)$ because $Y$ is smooth.
The ramification and branching loci $\nabla_{f}$ and $\Delta_{f}$ are hypersurfaces of $X$ and $Y$ respectively.
This follows from the purity of branch theorem (cf. \cite{Fischer}) and the finiteness of the map $f$.

Given a non-singular point  $q$ of $\Delta_{f}$, each $p\in f^{-1}(q)$ is a non-singular point of $X$ (cf. \cite[Corollary~1.1.10]{Namba}) and  there are 
local coordinates $(x_{1},\ldots,x_{n})$ in a neighborhood $V$ of $p$ and $(y_{1},\ldots,y_{n})$ in a neighborhood
$W$ of $q$ fulfilling $W\cap\Delta_{f}=\{y_{n}=0\}$ and $f(V)\subset W$, and such that in these coordinates
\begin{equation}\label{NF}
f(x_{1},\ldots,x_{n})=(x_{1},\ldots,x_{n-1},x_{n}^{\varrho}),
\end{equation} 
for some positive integer $\varrho=\varrho_{D}\geq 1$ which is constant along the irreducible component~$D$ of $f^{-1}(\Delta_{f})$ containing $p$ and which is called the \emph{ramification index} of $f$ along $D\subset f^{-1}(\Delta_{f})$, cf. \cite[Theorem~1.1.8]{Namba}. Notice that $\varrho_{D}=1$ if and only if $p\in f^{-1}(\Delta_{f})\setminus\nabla_{f}$.

The set $U:=Y\setminus\Delta_{f}$ is the maximal open subset of $Y$ such that the restriction
\begin{equation}\label{top_cov}
f_{U}:X_{U}=X\setminus f^{-1}(\Delta_f)\longrightarrow U=Y\setminus\Delta_{f}
\end{equation}
is an unbranched covering. The monodromy of that covering will be denoted by
$$\mu_{f}:\pi_{1}(Y\setminus\Delta_{f})\longrightarrow \mf S_{d},$$
where $d=\deg f_U$. We say that $d$ is the degree of the branched covering $f$.

Two finite branched coverings $f:X\to Y$ and $f':X'\to Y$ are said to be isomorphic if there is a biholomorphism $\phi:X\to X'$ such that $f'\circ\phi=f$.
The group $\mr{Deck}(f)=\{\phi\in\mr{Aut}(X)\,|\,f\circ\phi=f\}$ of all automorphisms of the branched covering $f:X\to Y$ is 
called the deck transformation group of $f$. The restrictions to $X_{U}$ of the elements of $\mr{Deck}(f)$ are deck transformations 
of the topological covering $f_U$
defined in~(\ref{top_cov}).

\begin{defin}
A finite branched covering $f:X\to Y$ is said to be Galois if $\mr{Deck}(f)$ acts transitively on each fiber of $f$. In that case, 
the quotient complex space $X/\mr{Deck}(f)$ is biholomorphic to $Y$.
\end{defin}

The following result states that $\mr{Deck}(f)$ and $\mr{Deck}(f_{U})$ are naturally isomorphic. Its proof, which is based on Riemann's extension theorem, 
can be found in \cite[Theorem~1.1.7]{Namba}.

\begin{teo}\label{iso1}
Let $f:X\to Y$ be a finite branched covering. The restriction map $\mr{Deck}(f)\to\mr{Deck}(f_{U})$ is an isomorphism. In particular, $f$ is Galois if and only if $f_{U}$ is Galois.
\end{teo}

We also recall the following theorem
due to Grauert and Remmert \cite{GR2} (cf. \cite[Theorem~1]{NT}). 
\begin{teo}\label{GR}
Let $\Delta$ be a hypersurface of a connected complex manifold $Y$ and let $f':X'\to Y\setminus\Delta$ be a finite unbranched covering. Then there are a unique (up to isomorphism) finite branched covering $f:X\to Y$ and an inclusion $X'\subset X$ with the the property that $f$ branches at most at $\Delta$, i.e. $\Delta_{f}\subset\Delta$, and that $f$ is an extension of $f'$.
\end{teo}

Two finite branched coverings $f:X\to Y$ and $f':X'\to Y'$ of the same degree are said to be topologically (resp. analytically) equivalent if there are homeomorphisms (resp. biholomorphisms) $\phi:X\to X'$ and $\psi:Y\to Y'$ such that the following diagram is commutative
$$\xymatrix{X\ar[r]^{\phi}\ar[d]_{f} &X'\ar[d]^{f'}\\ Y\ar[r]^{\psi}& Y'}$$

From Theorem~\ref{GR} one deduces the following criterion for deciding when two finite branched coverings are equivalent in 
terms of the base spaces and the corresponding monodromies (cf. \cite[Theorem~2]{NT}).

\begin{teo}\label{clas}
Two finite branched coverings $f:X\to Y$ and $f':X'\to Y'$ of degree $d$ are topologically  (resp. analytically) equivalent if and only if there is a homeomorphism (resp. biholomorphism) $\psi:(Y,\Delta_{f})\to (Y',\Delta_{f'})$ such that the representations $\mu_{f},\mu_{f'}\circ\psi_{*}:\pi_{1}(Y\setminus\Delta_{f})\to\mf S_{d}$ are conjugated.
\end{teo}

It is worth to recall also the following two results of M. Namba, proved in \cite{Namba2}.

\begin{teo}\label{Namba-existencia} 
For every finite group $G$ and every connected complex projective manifold $Y$ there 
exists a Galois branched covering $\rho:X\to Y$ whose  deck transformation group is isomorphic to~$G$.
\end{teo}

\begin{teo}\label{Namba-pull-back}
For every Galois branched covering $f:X\to Y$ over a projective manifold $Y$ there is an isomorphism 
$\mr{Deck}(f)\stackrel{\sim}{\to}G\subset\mr{Aut}(\p^{n})$ for some $n\in\N$ and a rational map 
$g:Y\dasharrow\p^{n}/G$ such that $f:X\to Y$ is birationally equivalent to the fibred product
$Y_{0}\times_{\p^{n}/G}\p^{n}\to Y_{0}$, where $g_{0}:Y_{0}\to \p^{n}/G$ is a resolution of the indeterminacy of $g$.
\end{teo}

This last theorem states that, in the setting of birational mappings that we will consider in Section~3, each Galois finite branched covering is the pull-back of a branched covering whose source space is a projective space. 
This motivates our interest in characterizing Galois branched coverings of the type $\p^n\to Y$. In this direction our main result is Theorem~\ref{GG}.


\section{Dominant rational maps}\label{Sec_3}

In this section we describe some properties of dominant rational maps $f:X\dasharrow Y$ between projective or quasi-projective manifolds.  We see that in the case that $X$ and $Y$ have the same dimension
 there is a finite branched covering $\rho\colon N\to Y$, naturally associated to  $f$, which is unique up to isomorphism and birationally equivalent to $f$. 
We define Galois rational maps as those whose associated branched covering $\rho$ is Galois. This definition coincides 
with the classical one that requires the field extension $\C(Y)\hookrightarrow \C(X)$ induced by $f$ to be Galois.

All the analytic or algebraic objects considered in this section and all along the article are defined over the field $\C$ of the complex numbers. 

\subsection{Equisingularity theorem}

We begin by recalling a general and powerful theorem due to A.N.~Var\v cenko, which implies the topological local triviality of rational maps on appropriate Zariski open subsets. It plays a key role in the article. To state it properly we give first the following definition:

\begin{defin}
Let $f:E\to B$ be a continuous map and let $E^{1},\ldots,E^{q}$ be subsets of the topological space $E$. 
The family $(f,E,E^{1},\ldots,E^{q})$ is called 
\emph{equisingular over $V\subset B$} if for every $p\in V$ there is a neighborhood $W$ of $p$ in $V$ and a homeomorphism $h:E_{W}\to E_{p}\times W$ such that $h(E_{W}^{i})=E_{p}^{i}\times W$, where $E_{W}^{i}=E_{W}\cap E^{i}$ and $E_{p}^{i}=E_{p}\cap E^{i}$.
\end{defin}

With this notation we can state Var\v cenko's theorem as follows, cf. \cite[Theorems~5.2 and 5.3]{V}.

\begin{teo}\label{V}
Let $f:E\to B$ be a morphism from a constructible set 
$E$ onto an irreducible constructible set 
$B$, and let $E^{1},\ldots,E^{q}$ be constructible subsets of $E$. Then there is a 
non empty Zariski open subset $V$ of $B$ such that the family $(f,E,E^{1},\ldots,E^{q})$ is equisingular over $V$.
\end{teo}

We recall that a constructible set is a finite union of  quasi-projective varieties. Over the complex numbers, a constructible set is just a  semi-algebraic set; that is, a set given locally by a finite number of algebraic equations $f_{i}=0$ and a finite number of algebraic inequalities $g_{i}\neq 0$. In particular, an irreducible constructible set is a quasi-projective variety.

Using Theorem~\ref{V} we prove the following proposition that describes the properties of the composition of dominant morphisms. It will be used all along the article.

\begin{prop}\label{Vb}
Let $f\colon X \to Y$ and $g\colon Y\to Z$ be dominant morphisms between quasi-projective varieties. Assume that $Y$ and $Z$ are irreducible. Then 
there exist Zariski open subsets $X'\subset X$, $Y'\subset Y$ and $Z'\subset Z$ fulfilling the following properties:
\begin{enumerate}[(a)]
\item 
The restrictions $f':X'\to Y'$, $g':Y'\to Z'$ and $g'\circ f':X'\to Z'$ are topological (locally trivial) fibre bundles.
\item If $\dim X=\dim Y$ then $f'$ and $f'_{z}:X'_{z}\to Y'_{z}$ are finite coverings of the same topological degree as $f$, for all $z\in Z'$.
In addition, if the generic fibre of $g$ is irreducible then $f'_{z}$ and~$f'_{z'}$ are topologically equivalent coverings for all $z,z'\in Z'$ by homeomorphisms $Y'_{z}\to Y'_{z'}$ that extend to $Y_{z}\to Y_{z'}$.
\end{enumerate}
\end{prop}

\begin{dem}
(a) By applying Theorem~\ref{V} to $f$, we see that there is a Zariski open subset $U$ of $Y$ contained in $ f(X)$  such that $f_{U}:X_{U}\to U$ is a topological fiber bundle. By applying again Theorem~\ref{V} to $g:U\to g(U)$, it follows that there is a Zariski open subset $V$ of $Z$ contained in $g(U)$ such that $g_{V}\colon U_{V}\to V$ and $f_{U_{V}}\colon f^{-1}(U_{V})\to U_{V}$ are also topological fiber bundles. Applying one more time Theorem~\ref{V} to the composition $g_{V}\circ f_{U_{V}}$ we obtain an open Zariski subset $Z'$ of $Z$ contained in $V$ such that the restrictions of $g$ to $Y':=g^{-1}(Z')$ and of $f$ to $X':=f^{-1}(Y')$ satisfy the desired properties.\\
(b) That the map $f'$ is a finite covering follows from~(a) because the generic fibre of $f$ has dimension $\dim X-\dim Y=0$. Moreover, the restriction $f'_{z}$ is a pull-back of the covering $f'$ by the inclusion $Y'_{z}\hookrightarrow Y'$. It remains to prove that the coverings $f'_{z}$ and $f'_{z'}$ are topologically equivalent for every $z,z'\in Z'$ by homeomorphisms extending to $Y_{z}\to Y_{z'}$. 
By Theorem~\ref{V}, we can assume that the family $(Y,Y')\to Z$ is equisingular over $V$.
Each point $z\in Z'$ has a contractible neighborhood $W\subset Z'$ such that  $(Y_{W},Y'_{W})\simeq (Y_{z}\times W, Y'_{z}\times W)$.
Then the monodromy representations   $\mu_{W}:\pi_{1}(Y'_{W})\simeq\pi_{1}(Y'_{z})\stackrel{\mu_{z}}{\to}\mf S_{d}$ of the restricted coverings $f'_{W}\colon X'_{W}\to Y'_{W}$ and $f'_{z}\colon X'_{z}\to Y'_{z}$ can be canonically identified. Consequently, $f'_{W}\colon X'_{W}\to Y'_{W}$ is topologically equivalent to $f'_{z}\times\mr{id}_{W}:X'_{z}\times W\to Y'_{z}\times W$. This implies that $f_{z}'$ and $f_{z'}'$ are topologically equivalent by homeomorphisms that extend to $Y_{z}\to Y_{z'}$  if $z'\in W$.
If $z'\in Z'$ is arbitrary then we can join it with $z$ by a path $\gamma$ in $Z'$ and choose a finite set of open sets $W$ covering $\gamma$ in order to conclude.
\end{dem}

\subsection{Galois rational maps}\label{3.2}
Let $X, Y$ be connected complex projective manifolds of the same dimension and let $f:X\dasharrow Y$ be a dominant rational map, i.e. a rational map 
with dense image. Let  $\Sigma_{f}\subset X$ be the indeterminacy locus of $f$. 
We consider  the closed graph of $f$
$$\Gamma_f:=\overline{\{(x,f(x))\,|\, x\in X\setminus\Sigma_{f}\}}\subset X\times Y$$ 
and we denote by $p_{X}$ and $p_{Y}$ the restrictions to $\Gamma_f$ of the natural projections from $X\times Y$ onto $X$ and $Y$ respectively.

Let $\delta:\wt X\to\Gamma_{f}$ be a desingularization of $\Gamma_{f}$, i.e a proper surjective birational morphism from a smooth projective manifold $\wt X$. Without loss of generality we can assume that the \emph{exceptional divisor} $\nabla_{\beta}$ of the birational map $\beta:=p_{Y}\circ\delta:\wt X\to X$ satisfies 
\begin{equation}\label{min}
\nabla_{\beta}=\beta^{-1}(\Sigma_{f}).
\end{equation}

The map $\wt f:=p_{Y}\circ\delta:\wt X\to Y$ is a  proper surjective morphism because $f$ is dominant and $\wt X$ is projective. 
We will say that $\wt f$ is a \emph{desingularization} of the rational map $f$. Thus we can apply the Stein factorization theorem to $\wt f$ in order to  write it as the composition $\wt X\stackrel{\gamma}\to N\stackrel{\rho}{\to}Y$ with $\gamma$ having connected fibres and $\rho$ being finite. In fact, $\gamma$ is birational because $\dim X=\dim Y$, $N$ is normal because $\wt X$ is smooth (cf. \cite[p. 213]{GR84}) and $\rho:N\to Y$ is a finite branched covering. Then the following diagram is commutative:
\begin{equation}\label{stein}
\xymatrix{\wt X\ar[r]^{\gamma}\ar[d]_{\beta}\ar[rd]^{\wt f} & N\ar[d]^{\rho}\\ X\ar@{.>}[r]^{f} & Y}
\end{equation}

The following proposition follows from Theorem~\ref{GR}. It 
states that the finite branched covering $\rho\colon N\to Y$ does not depend 
on the chosen desingularization $\delta$ of $\Gamma_{f}$.

\begin{prop}\label{rho}
Let $f:X\dasharrow Y$ be a dominant rational map between projective manifolds of the same dimension. 
The morphism $\rho:N\to Y$ constructed above is unique up to isomorphism. 
We say that $\rho$ is the finite branched covering associated to $f$.
\end{prop}

Although the finite branched covering $\rho\colon N\to Y$ is a 
morphism univocally associated to $f$, it has the disadvantage that $N$ can be singular. By that reason, 
we will look sometimes at the rational morphism $\wt f\colon \wt X\to Y$ rather than $\rho$ itself. The hypersurface $\nabla_{\wt f}$ decomposes as 
\begin{equation}\label{delta}
\nabla_{\wt f} =\mc R_{\wt f}\cup\mc C_{\wt f}
\end{equation}
where $\mc R_{\wt f}$ is the union of all irreducible components $C\subset\nabla_{\wt f}$ such that $\wt f(C)$ is a hypersurface of 
$Y$ and $\mc C_{\wt f}=\overline{\nabla_{\wt f}\setminus\mc R_{\wt f}}$ 
is the union of all irreducible components of $\nabla_{\wt f}$ that are contracted by $\wt f$. Notice that  $\gamma(\mc R_{\wt f})$ coincides with the subset $\nabla_{\rho}$ of $N$ and that $\gamma_{|\mc R_{\wt f}} \colon \mc R_{\wt f} \to \nabla_\rho$ is a birational map. This means that the components of the ramification locus of $\rho$ and their ramification indices can be seen in $\wt X$.
We also deduce that the Zariski closed subset $\Delta_{f}:=\wt f(\mc R_{\wt f})$ fulfills 
\begin{equation}\label{delta2}
\Delta_{f}=\Delta_{\rho}
\end{equation}
and that it does not depend on the desingularization. We also consider the Zariski closed subset 
$$\Lambda_{f}:=\wt f(\nabla_{\wt f}\cup\nabla_{\beta})\subset Y$$
which, under the asumption~(\ref{min}), is independent on the chosen desingularization  because it coincides with $p_{Y}(p_{X}^{-1}(\overline{\nabla_{f_{|X\setminus\Sigma_{f}}}}\cup\Sigma_{f}))$.

\begin{defin}
A dominant rational map $f:X\dasharrow Y$ between 
projective manifolds of the same dimension is said to be Galois if
the associated finite branched covering $\rho\colon N\to Y$ is Galois.
\end{defin}

Next theorem collects some known facts with the assertion that $f\colon X\to Y$ is Galois if its restriction $f_{V}:X_{V}=f^{-1}(V)\to V$ to any Zariski open subset $V$ of $Y\setminus\Lambda_{f}$ is Galois.

Since the manifolds $X$ and $Y$ are assumed to be connected the rational map $f$ induces a finite field extension 
$f^\ast \colon\C(Y)\hookrightarrow\C(X)$ whose degree is the topological degree $\deg  f$ of $f$. Hence, one could also say that 
$f$ is Galois if the the field extension $\C(X)|\C(Y)$ is Galois.

\begin{teo}\label{car-gal}
Let $f:X\dasharrow Y$ be a given dominant rational map between projective manifolds of the same dimension and let $\rho\colon N\to Y$  be the finite branched covering associated to $f$. If $V$ is a Zariski open subset of $Y$ contained in $Y\setminus \Lambda_{f}$ then $f_{V}:X_{V}\to V$ is a covering whose monodromy group does not depend on $V$. Moreover,  the groups
\begin{itemize}
\item $\mr{Deck}(\rho)=\{\phi\in\mr{Aut}(N)\,|\,\rho\circ\phi=\rho\}$,
\item $\mr{Deck}(f_{V})=\{\phi\in\mr{Homeo}(X_{V})\,|\,f_{V}\circ\phi=f_{V}\}$,
\item $\mr{Deck}(f)=\{\phi\in\mr{Bir}(X)\,|\,f\circ\phi=f\}$,
\item $\mr{Aut}(\C(X)|\C(Y))=\{\varphi\in\mr{Aut}(\C(X))\,|\,\varphi_{|\C(Y)}=\mr{id}_{\C(Y)}\}$.
\end{itemize}
are naturally isomorphic. 
\end{teo}

\begin{dem}
If $V$ is a Zariski open subset of $Y$ contained in $Y\setminus\Lambda_{f}$ then the restrictions $\beta_{V}$ and $\gamma_{V}$ of $\beta$ and $\gamma$  to $\wt f^{-1}(V)$ are biholomorphisms onto $X_{V}=f^{-1}(V)$ and $N_{V}=\rho^{-1}(V)$ respectively. Hence we can identify the restriction $f_{V}:X_{V}\to V$ and the covering $\rho_{V}:N_{V}\to V$ via the biholomorphism $\gamma_{V}\circ\beta_{V}^{-1}:X_{V}\to N_{V}$. On the other hand, since $V\subset Y\setminus\Lambda_{f}\subset Y\setminus\Delta_{f}=U$, the covering $\rho_{V}$ is a restriction of the maximal unbranched covering $\rho_{U}$ considered in \eqref{top_cov}. The monodromy representation of $\rho_{V}$ is the composition of the monodromy representation of $\rho_{U}$ with the morphism $\imath_{*}:\pi_{1}(V)\to\pi_{1}(U)$ induced by the inclusion $\imath:U\hookrightarrow V$. Using the Lefschetz type theorem proved by Hamm and L\^{e} in \cite{HL}, we deduce that $\imath_{*}$ is an epimorphism, so that the monodromy groups $\mr{Mon}(\rho_{U})$ and $\mr{Mon}(\rho_{V})\simeq\mr{Mon}(f_{V})$ of $\rho_{U}$ and $\rho_{V}$ coincide.

It is easy to see that the following diagram is commutative:
$$
\xymatrix{
\mr{Deck}(f)\ar@{^{(}->}[r] &\mr{Deck}(f_{V})\\
\mr{Deck}(\rho)\ar@{^{(}->}[r]^{r}\ar@{^{(}->}[u]^{\gamma^{*}} & \mr{Deck}(\rho_{U})\ar@{<->}[u]
}
$$
where $\gamma^\ast$ is defined by $\gamma^{*}(\phi)=\gamma\circ\phi\circ\gamma^{-1}$ if $\phi\in\mr{Aut}(N)$,
the horizontal arrows are injective because they are given by restriction, and the right vertical arrow is the composition 
of the isomorphisms
$$
\mr{Deck}(\rho_{U})\simeq Z(\mr{Mon}(\rho_{U}))
=Z(\mr{Mon}(\rho_{V}))
=Z(\mr{Mon}(f_{V}))\simeq\mr{Deck}(f_{V}),
$$
where we are using~(\ref{D=ZM}). Moreover, $r$ is surjective thanks to Theorem~\ref{iso1}. Hence all the arrows considered  are isomorphisms.
Finally, the groups $\mr{Deck}(f)$ and $\mr{Aut}(\C(X)|\C(Y))$ are naturally identified.
\end{dem}

\begin{obs}
Since the above natural isomorphisms are given by restrictions the previous proof shows that every $\phi\in\mr{Bir}(N)$ such that $\rho\circ\phi=\rho$ is actually in $\mr{Aut}(N)$ and that every $\phi\in\mr{Homeo}(X_{V})$  such that $f\circ\phi=f$ extends to a birational map $X\dasharrow X$.
\end{obs}

The characterization of Galois rational maps $f:X\dasharrow Y$ via the induced field extension $\C(X)|\C(Y)$ show that being Galois is a birational property.
More precisely, two rational maps $f\colon X\dasharrow Y$ and $f'\colon X'\dasharrow Y'$ are called \emph{birationally left-right-equivalent} (\emph{birationally equivalent} for short) if there are
birational maps $\beta_{X}\colon X'\dasharrow X$ and $\beta_{Y}\colon Y\dasharrow Y'$ such that $f'=\beta_{Y}\circ f\circ\beta_{X}$. It follows that if $f$ and $f'$ are birationally equivalent then $f$ is Galois if and only if $f'$ is Galois.

From the above discussion, we conclude that every dominant rational map $f':X'\dasharrow Y'$ between irreducible projective varieties of the same dimension is birationally equivalent to a branched covering $f:X\to Y$. 
In that case 
$\mr{Deck}(f)\subset\mr{Aut}(X)$, and if $f'$ is Galois then $Y=X/G$ with $G=\mr{Deck}(f)$.

\section{Rational maps of regular type}

Let $f:X\to Y$ be a finite branched covering of degree $d$.
According to Theorem~\ref{clas}, a complete systems of topological invariants of $f$ is given by the embedded topological type of $\Delta_{f}\subset Y$ jointly with the conjugacy class of the monodromy representation $\mu_{f}:\pi_{1}(Y\setminus\Delta_{f})\to\mf S_{d}$. Now, we introduce a weaker topological invariant of combinatorial nature. 

\subsection{Branching type}\label{4.1} Let $C$ be an irreducible component of $\Delta_{f}$. Each irreducible component $D$ of $f^{-1}(C)$
has a ramification index $\varrho_{D}$ defined by the normal form~(\ref{NF}).
The monodromy of a local generator of
$\pi_{1}(W\setminus C)\simeq\pi_{1}(\C^{n-1}\times\C^{*})\simeq\Z$ is a product of disjoint $\varrho_{D}$-cycles.
Therefore the sum of the ramification indices of the irreducible components of $f^{-1}(C)$ is equal to the degree $d$ of $f$. We consider the set of degree $d$ branching types
$$\mf B_{d}=\bigcup\limits_{k=1}^{d-1}\Big\{(\varrho_{1},\ldots,\varrho_{k})\in\mb N^{k}\,|\,\varrho_{1}\ge\varrho_{2}\ge\cdots\ge\varrho_{k},\ \sum_{j=1}^{k}\varrho_{j}=d\Big\}$$ and the subsets
$\mf B_{d}^{\mr{ext}}=\{(d)_{1}\}\subset\mf B_{d}^{\mr{reg}}=\mathop{\bigcup}\limits_{k}\{(d/k)_{k}\}\subset\mf B_{d}$, where $k$ varies in the set of divisors of $d$ with $k<d$ and
 $(\varrho)_{k}=(\varrho,\ldots,\varrho)\in\mb N^{k}$.

\begin{defin}\label{bt}
The \emph{branching type} of a branching covering $f:X\to Y$ of degree $d$ is the map $\mf b_{f}:\mf C_{f}\to \mf B_{d}$ obtained by taking in increasing order the ramification indices along the irreducible components of $f^{-1}(C)$, where $C$ varies in the set $\mf C_{f}$ of  irreducible components of $\Delta_{f}$. The branched covering $f$ is called of regular type (resp. extremal type) if the image of $\mf b_{f}$ is contained in $\mf B_{d}^{\mr{reg}}$ (resp. $\mf B_{d}^{\mr{ext}}$).
\end{defin}

Notice that a branched covering is of regular type  if the ramification indices are constant on the fibres over $\Delta_{f}\setminus\mr{Sing}(\Delta_{f})$. A finite branched covering is of extremal type if all its ramification indices are equal to the degree of $f$.

\begin{obs}\label{global-local}
For a finite branched covering, being Galois is a global property, being of regular type is a semi-local one and being of extremal type is purely local. 
For arithmetical reasons, if the degree $d$ is prime, then regular type is equivalent to extremal type. Galois implies regular type and extremal type implies regular type but the converses do not hold in general as the following example shows.
\end{obs}

\begin{ex}\label{contra}
Let  $Y_{0}=\p^{1}\setminus\bigcup\limits_{i=0}^{3}D_{i}$ be the complement of four disjoint open disks  in $\p^{1}$ with boundaries $\gamma_{i}$   and let $\mu:\pi_{1}(Y_{0})\simeq\mb Z[\gamma_{1}]
\ast\mb Z[\gamma_{2}]\ast\mb Z[\gamma_{3}]\to\mf S_{4}$ be the morphism given by 
$\mu(\gamma_{1})=(1234)$,
$\mu(\gamma_{2})=(12)(34)$ and $\mu(\gamma_{3})=(14)(23)$. Define $f_{0}:X_{0}=\wt{Y}_{0}\times_{\mu}\{1,2,3,4\}\to Y_{0}$ to be the suspension covering associated to $\mu$ which is not Galois because the monodromy group $M=\mr{Im}\,\mu$ has order $8>4$. Notice that $[\gamma_{0}]^{-1}=[\gamma_{1}\gamma_{2}\gamma_{3}]\stackrel{\mu}{\mapsto}(1432)$.
It is clear that $f_{0}^{-1}(\gamma_{i})=\delta_{i}$ is a circle and $f_{0|\delta_{i}}$ is a $4:1$ map for $i=0,1$; on the other hand, if $i=2,3$ then
 $f_{0}^{-1}(\gamma_{i})=\delta_{i}^{+}\sqcup\delta_{i}^{-}$ are two disjoint circles and $f_{0|\delta_{i}^{\pm}}$ is a $2:1$ map. Consequently, we can glue disks $\Delta_{i}$, $i=0,1$, and $\Delta_{i}^{\pm}$, $i=2,3$, to $X_{0}$ in order to obtain a compact Riemann surface $X$ and a branched covering $f:X\to \p^{1}$ extending $f_{0}$ with four branched points $q_{0},q_{1},q_{2},q_{3}\in\p^{1}$ and six ramification points $p_{0},p_{1},p_{2}^{\pm},p_{3}^{\pm}\in X$ with ramification indices $4,4,2,2,2,2$ respectively. Hence $\mr{Im}\,\mf b_{f}=\{(2)_{2},(4)_{1}\}\subset\mf B_{4}^{\mr{reg}}$ and consequently $f$ is a degree~$4$ branched covering of regular type.  Riemann-Hurwitz formula implies that $X$ has genus $2$.
\end{ex}

In \cite[Lemma~1]{G} L. Greenberg shows that if the source space $X$ is a connected and simply connected Riemann surface (for instance if $X=\p^{1}$), then a regular type branched covering is Galois. This property follows, as a particular case, from the following result in which the difference between Galois coverings and regular type branched coverings is enlightened.

Let $f:X\to Y$ be a finite branched covering a let $f':X'\to Y'$ be the restriction of $f$ to $X'=f^{-1}(Y')$ with $Y'=Y\setminus\mr{Sing}(\Delta_{f})$.
Since $X\setminus X'$ and $Y\setminus Y'$ are Zariski closed subsets of codimension $\ge 2$ the branching type of $f$ and $f'$ coincide and $f$ is of regular type if and only if $f'$ is of regular type. By Riemann's extension theorem, $\mr{Deck}(f)\simeq\mr{Deck}(f')$ and $f$ is Galois if and only $f'$ is Galois. We have that  $\Delta_{f'}=\Delta_{f}\cap X'$ and  $\mr{Sing}(X)\subset f^{-1}(\mr{Sing}(\Delta_{f}))$, cf. \cite[Corollary~1.1.10]{Namba}. Thus, in order to characterize the Galois or regular type property, we can assume without loss of generality that $\Delta_{f}$ and $X$ are smooth.

\begin{prop}\label{Galreg}
Let $f:X\to Y$ be a finite branched covering. Assume that $\Delta_{f}$ is smooth and consider the branched covering 
$$F:(X\times_{Y}X)^{\nu}\to X\times_{Y}X\to X$$
given by the composition of the normalization of the fibered product $X\times_{Y}X$ and the projection onto the first factor. Then
\begin{enumerate}[(a)]
\item $f$ is of regular type if and only if $F$ is unbranched;
\item $f$ is Galois if and only if $F$ is trivial.
\end{enumerate}
\end{prop}
\begin{dem}
Since the normal form (\ref{NF}) holds in every point of $X$ we can proceed as follows.
Let $X_{0}:=\bigsqcup_{i=1}^{r}\mb D_{i}^{n}\hookrightarrow X$ be the preimage by $f$ of a polydisk $Y_{0}:=\mb D^{n}\hookrightarrow Y$ such that $\{0\}\times\mb D^{n-1}=\Delta_{f}\cap Y_{0}$ and the  restriction $f_{i}$ of $f$ to the polydisk $\mb D_{i}^{n}$ writes as $f_{i}(x_{i},u)=(x_{i}^{n_{i}},u)$. Then 
$$X_{0}\times_{Y}X_{0}=\bigsqcup_{i,j=1}^{r}\{(x_{i},u,y_{j},v)\in\mb D^{2n}\,|\, x_{i}^{n_{i}}=y_{j}^{n_{j}},\, u=v\}\hookrightarrow X\times_{Y}X$$ is the preimage in $X\times_{Y}X$ of $X_{0}\hookrightarrow X$ by the  projection $ X\times_{Y}X\to X$. 
The preimage $Z_{0}$ of $X_{0}$ by $F$ is nothing more than the normalization of $X_{0}\times_{Y}X_{0}$. If $n_{i}=n_{j}$ for all $i,j=1,\ldots,r$ then $Z_{0}$ is a disjoint union of polydisks $\{x_{i}=\zeta^{k} y_{j}\}\times \mb D^{n-1}$, over which $F(x_{i},y_{j},u)=(x_{i},u)$ is an isomorphism, where $\zeta$ is a primitive $n_{i}$-root of the unity. This shows that if $f$ is of regular type then $F$ is unbranched.
To prove the converse, assume that $n_{i}\neq n_{j}$. Then $\{(x_{i},y_{j})\in\mb D^{2}\,|\, x_{i}^{n_{i}}=y_{j}^{n_{j}}\}\times\mb D^{n-1}$ decomposes as $k$ branches of type $x_{i}^{n_{i}'}=\zeta'y_{j}^{n_{j}'}$, and where $n_{i}=n_{i}'k$, $n_{j}=n_{j}'k$, $\gcd(n_{i}',n_{j}')=1$ and $\zeta'$ is a primitive $k$-root of the unity.
The normalization morphism of each branch takes the form $\mb D^{n}\ni (z,w)\mapsto(z^{n_{j}'},z^{n_{i}'},w)$. Hence the restriction of $F$ to the normalization of this branch writes as $F(z,w)=(z^{n_{i}'},w)$ which ramifies if $n_{i}'>1$. Finally, if $n_{i}\neq n_{j}$ there is always a connected component of the preimage of $\{0\}\times\mb D^{n-1}$ with $n_{i}'>1$.

Assertion~(b) follows easily from Theorem~\ref{car-gal} using characterization~(\ref{4}) in Definition~\ref{def-gal}.
\end{dem}

\subsection{Dimensional reduction}

In this subsection we translate the problem of deciding if a given rational map is Galois to a lower dimensional situation. This is done in two different ways: the first one by considering the restriction to appropriate curves and the second one by taking suitable quotients of the manifolds.

\begin{defin}
Let $f:X\to Y$ be a finite branched covering and let $Z\subset Y$ be a connected submanifold. We will denote by $f_{Z}^{\nu}:X_{Z}^{\nu}\to Z$
the branched covering given by the composition of the normalization map of $f^{-1}(Z)$ and the restriction of $f$ to $f^{-1}(Z)$.\end{defin}

Clearly, $\deg f_{Z}=\deg f$ if and only if $Z\not\subset\Delta_{f}$.

\begin{prop}\label{Z}
Let $f:X\to Y$ be a finite branched covering and let $Z\subset Y$ be a connected submanifold.
If $Z$ meets transversely $\Delta_{f}$ then
\begin{enumerate}[(a)]
\item $X_{Z}=f^{-1}(Z)$ is smooth and $f_{Z}^{\nu}=f_{Z}$,
\item $\Delta_{f_{Z}}=\Delta_{f}\cap Z$,
\item there is a map $\imath_{Z}:\mf C_{f_{Z}}\to\mf C_{f}$ such that $\mf b_{f_{Z}}=\mf b_{f}\circ\imath_{Z}$.
\end{enumerate}
\end{prop}

\begin{dem}
(a) Transversality implies that $Z\cap \Delta_{f}\subset\Delta_{f}\setminus\mr{Sing}\,\Delta_{f}$. By \cite[Theorem~1.1.8]{Namba} every point $q$ of $f^{-1}(Z\cap\Delta_{f})$ is non-singular for both $X$ and $f^{-1}(\Delta_{f})$ and there are local coordinates $(x_{1},\ldots,x_{n})$ and $(y_{1},\ldots,y_{n})$ around $q\in X$ and $f(q)\in Y$ in which $f$ writes in the normal form (\ref{NF}).
Moreover we can assume that $Z=\{y_{1}=\cdots=y_{k}=0\}$ with $k=n-\dim Z<n$. Hence $f^{-1}(Z)=\{x_{1}=\cdots=x_{k}=0\}$ is smooth at $q$. Assertion (b) is obvious from the local writting $f_{Z}(x_{k+1},\ldots,x_{n})=(x_{k+1},\ldots,x_{n-1},x_{n}^{\varrho})$.
(c) Since $Z$ is disjoint from the pairwise intersections of the irreducible components of $\Delta_{f}$, there is a well defined map $\imath_{Z}:\mf C_{f_{Z}}\to\mf C_{f}$ sending an irreducible component $C\subset\Lambda_{f_{Z}}\subset Z$ to the unique irreducible component of $\Delta_{f}\subset Y$ containing $C$. Then $\mf b_{f_{Z}}=\mf b_{f}\circ\imath_{Z}$.
\end{dem}

\begin{obs}
If $Z$ meets transversely all the irreducible components of $\Delta_{f}$ (e.g. if $Y$ and $Z$ are projective) then $f$ is of regular type if and only if $f_{Z}$ is of regular type.
On the other hand, if the inclusion $Z\setminus\Delta_{f}\subset Y\setminus\Delta_{f}$ induces an epimorphism $\pi_{1}(Z\setminus\Delta_{f})\twoheadrightarrow\pi_{1}(Y\setminus\Delta_{f})$ then $f$ is Galois  if and only if $f_{Z}$ is Galois  after Proposition~\ref{1}.
\end{obs}

These considerations motivate the following definition.

\begin{defin}
Let $\Delta$ be a proper Zariski closed subset of $Y$.
A curve $\ell\subset Y$ is called $\Delta$-admisible if $\ell$ meets transversely $\Delta$ and the inclusion $\ell\setminus\Delta\subset Y\setminus\Delta$ induces an epimorphism $\pi_{1}(\ell\setminus\Delta)\twoheadrightarrow\pi_{1}(Y\setminus\Delta)$.
\end{defin}

The main result of this section is Theorem~\ref{dim-red} below which provides  a reduction to dimension one to the problem of deciding if a given branched covering is Galois of or regular type. In order to state it properly we introduce the following notion.

\begin{defin}\label{wbt}
The \emph{weighted branching type} of a finite branched covering $f:X\to Y$ of degree $d$ with respect to a fixed embedding $Y\subset\p^{N}$ is the element of the $\Z$-module $\Z[\mf B_{d}]$ given by
$$\mf b^{w}_{f}:=\sum_{C\in\mf C_{f}}(\deg C)\mf b_{f}(C).$$
We also define the integer
$|\mf b_{f}^{w}|:=\sum\limits_{C\in\mf C_{f}}(\deg C)(d-k_{C})\in\Z^{+}$,
where $\mf b_{f}(C)=(\varrho_{1},\ldots,\varrho_{k_{C}})\in\N^{k_{C}}$.
\end{defin}

\begin{teo}\label{dim-red}
Let $f\colon X\to Y$ be a finite branched covering of a projective manifold $Y$ of dimension $n$, let $\Lambda\subset Y$ be a Zariski closed subset containing $\Delta_{f}$ and fix an embedding $Y\subset\p^{N}$. Then there is a Zariski open subset $V$ of the Grassmannian $\mb G$ of $(N-n+1)$-planes in $\p^{N}$ such that 
\begin{enumerate}[(a)]
\item
for all $L\in V$ the \emph{hyperplane curve} $\ell=L\cap Y$ is connected, smooth and $\Lambda$-admisible;
\item  all the one-dimensional branched coverings $f_{\ell}:X_{\ell}\to\ell=L\cap Y$, varying $L\in V$, are topologically equivalent and $\mf b_{f_{\ell}}^{w}=\mf b_{f}^{w}$.
\end{enumerate}
In particular, $f$ is Galois (resp. of regular type) if and only if $f_{\ell}$ is Galois (resp. of regular type) and
the genera of $X_{\ell}$ and $\ell$ and the degree $d$ of $f$ are related by the Riemann-Hurwitz formula
\begin{equation}\label{RHF}
2-2g_{X_{\ell}}=d(2-2g_{\ell})-|\mf b_{f}^{w}|.
\end{equation}
\end{teo}

\begin{dem}
Let us consider the projective varieties $\mc L=\{(y,L)\in Y\times\mb G\,|\, y\in L\}$ and $\mc M=\{(x,L)\in X\times\mb G\,|\, f(x)\in L\}$ jointly with the natural morphisms $\mc L\to\mb G$, $\mc M\to\mb G$ and $\lambda=f\times\mr{id}_{\mb G}\colon\mc M\to\mc L$.
By Theorem~\ref{V} there is a Zariski open subset $V\subset\mb G$ such that the family $(Y\times\mb G,\Lambda\times\mb G,\mc L)\to\mb G$ is equisingular over $V$. 
For every $L\in V$ there is a neighborhood $W$ of $L$ in $V$ and a homeomorphism $h:Y\times W\to Y\times W$ such that
$h(\Lambda\times W)=\Lambda\times W$ and $h(\mc L_{W})=\ell\times W$, where $\ell=L\cap Y$ is (by definition) a hyperplane curve of~$Y$, which is smooth (Bertini) and connected (Lefschetz). By restricting the Zariski open subset $V$ if necessary we can assume that $L\cap Y$ is transverse to $\Lambda$ for every $L\in V$.
Moreover, by successive application of \cite[Theorem~1.1.3]{HL} there is a dense real semi-algebraic open subset~$U$ of~$\mb G$ such that if $L\in U$ then $\ell=L\cap Y$  is $\Lambda$-admisible.
If $L_{0}\in U\cap V$ and $L_{1}\in V$ then the following diagram is commutative
$$\xymatrix{\pi_{1}(Y\setminus\Lambda)\ar@{<->}[r]^{\sim} & \pi_{1}(Y\setminus\Lambda)\\
\pi_{1}(\ell_{0}\setminus\Lambda)\ar@{<->}[r]^{\sim}\ar@{->>}[u]& \pi_{1}(\ell_{1}\setminus\Lambda)\ar[u]}$$
and consequently the hyperplane curve $\ell_{1}=L_{1}\cap Y\subset Y$ is also $\Lambda$-admissible. 
By Proposition~\ref{Z}, for each $L\in V$ we have a finite branched covering $f_{\ell}:X_{\ell}\to\ell=L\cap V$.

By Proposition~\ref{Vb}, there exist Zariski open subsets $\mc L'\subset \mc L$ and $\mc M'\subset \mc M$ such that
$\lambda'\colon\mc M'\to\mc L'$ and $\lambda'_{L}\colon\mc M'_{L}\to\mc L'_{L}$, $L\in V$,  are finite coverings and the projections $\mc M'\to V$ and $\mc L'\to V$ are topological fibre bundles. Moreover, for every $L_{0},L_{1}\in V$ there are homeomorphisms $\psi:\mc L_{L_{0}}\to\mc L_{L_{1}}$ and $\phi':\mc M_{L_0}'\to\mc M_{L_{1}}'$ making commutative the following diagram:
$$\xymatrix{\mc M'_{L_{0}}\ar[r]^{\phi'}\ar[d]_{\lambda'_{L_{0}}} & \mc M'_{L_{1}}\ar[d]^{\lambda'_{L_{1}}}\\
\mc L_{L_{0}}'\ar[r]^{\psi'} &\mc L_{L_{1}}'
}$$
where $\psi'$ is the restriction of $\psi$. Notice that  for $i=0,1$ we have that $\mc L_{L_{i}}'$  and $\mc M'_{L_{i}}$ 
are naturally included in $\ell_{i}\setminus\Delta_{f}$ and $f^{-1}(\ell_{i}\setminus\Delta_{f})$ respectively. Moreover we can identify $\lambda_{L_{i}}'$ with the restriction of $f_{\ell_{i}}:X_{\ell_{i}}\to\ell_{i}$, $i=0,1$. 
Since $\psi$ maps $\ell_{0}\setminus\Delta_{f}$ isomorphically onto $\ell_{1}\setminus\Delta_{f}$ and the monodromy groups of $f_{\ell_{i}}$ and $\lambda_{L_{i}}'$ coincide by Theorem~\ref{car-gal} we deduce that the monodromy representations of the maximal unbranched coverings of $f_{\ell_{i}}$ are conjugated. We deduce that $f_{\ell_{0}}$ and $f_{\ell_{1}}$ are topologically equivalent by applying Theorem~\ref{clas}.

Let $C$ be an irreducible component of $\Delta_{f}$.  Then $C\cap\ell=\{q_{1}^{C},\ldots,q_{\delta_{C}}^{C}\}$ where $\delta_{C}=\deg C$.
By Proposition~\ref{Z} we have $\mf b_{f_{\ell}}(q_{i}^{C})=\mf b_{f}(C)$. Therefore
$$\mf b_{f_{\ell}}^{w}=\sum_{C\in\mf C_{f}}\sum_{i=1}^{\delta_{C}}\mf b_{f_{\ell}}(q_{i}^{C})=\sum_{C\in\mf C_{f}}\delta_{C}\mf b_{f}(C)=\mf b_{f}^{w}.$$
Finally, 
$$|\mf b_{f}^{w}|=\sum_{C\in\mf C_{f}}(\deg C) (d-k_{C})=\sum_{C\in\mf C_{f}}\sum_{q\in\ell\cap C}\sum_{p\in f_{\ell}^{-1}(q)}(\varrho_{p}-1)=\sum_{p\in \nabla_{f_{\ell}}}(\varrho_{p}-1)$$ is the ramification summand in the Riemann-Hurwitz formula,
proving the last assertion of the theorem.
\end{dem}

We finish this section with another useful dimensional reduction in the context of rational maps.

\begin{prop}\label{quo-red}
Consider a commutative square of dominant rational maps 
$$\xymatrix{X\ar@{.>}[r]^{f}\ar@{.>}[d]_{u} & Y\ar@{.>}[d]^{v}\\ \wh X\ar@{.>}[r]^{\wh f} & \wh Y}$$ 
with $\dim\wh X=\dim\wh Y$ and $\dim X=\dim Y$. Assume that $\wh X$, $\wh Y$, $Y$ and the generic fibre of $v:Y\dasharrow\wh Y$ are irreducible and $\deg f=\deg\wh f=d$. Then the monodromy and the deck transformation groups of $f$ and $\wh f$ are conjugated in $\mf S_{d}$.
\end{prop}

\begin{dem}
By Proposition~\ref{Vb} there are Zariski open subsets 
$X'\subset X\setminus(\Sigma_{f}\cup\Sigma_{u})$,
$Y'\subset Y\setminus\Sigma_{v}$, $\wh X'\subset \wh X\setminus\Sigma_{\wh f}$ and $\wh Y'\subset\wh Y$
such that the restricted diagram
\begin{equation}\label{sq}
\xymatrix{X'\ar@{->}[r]^{f'}\ar@{->}[d]_{u'} & Y'\ar@{->}[d]^{v'}\\ \wh X'\ar@{->}[r]^{\wh f'} & \wh Y'}
\end{equation}
is commutative, the horizontal arrows are coverings and the vertical arrows are fibre bundles. The exact sequence associated to the $F'$-fiber bundle $v'$ ends as
\begin{equation}\label{epiL}
\pi_{1}(Y')\to\pi_{1}(\wh Y')\to 0=\pi_{0}(F').
\end{equation}
By the universal property of the fibered product there is a map $w':X'\to Z'=\wh X'\times_{\wh Y'}Y'$ making commutative the following diagram:
$$\xymatrix{X'\ar@/^/[rrd]^{f'}\ar@/_/[rdd]_{u'} \ar@{->}[dr]^{w'}& &\\
&Z'\ar@{->}[r]_{q'}\ar@{->}[d]^{p'} & Y'\ar@{->}[d]^{v'}\\ &\wh X'\ar@{->}[r]^{\wh f'} & \wh Y'}$$
The long exact sequence of the pull-back bundle $p':Z'\to\wh X'$ ends as
$$\pi_{0}(F')=0\to\pi_{0}(Z')\to 0=\pi_{0}(\wh X')$$
and consequently $Z'$ is connected. Hence 
$$\deg f=\deg f'=(\deg w')(\deg q')=(\deg w')(\deg \wh f')=(\deg w')(\deg\wh f).$$
Since $\deg f=\deg\wh f$, we deduce that $w'$ is birational and restricting the Zariski open set if necessary we can assume that $w'$ is a biholomorphism, i.e. the square \eqref{sq} is cartesian, or equivalently, the covering $f'$ is the pull-back by $v'$ of the covering $\wh f'$. We conclude by applying Proposition~\ref{1} to the epimorphism~\eqref{epiL}.
\end{dem}

\subsection{Rational maps in the projective space.}\label{RT}

By definition, the branching type of a dominant rational map $f:X\dasharrow Y$ between projective manifolds of the same dimension is the branching type of its associated branched covering, see Proposition~\ref{rho}. The notions of regular and extremal type rational map are the obvious ones. In particular, if $f\colon X\dasharrow Y$ and $f'\colon X'\to Y$ are birationally left-equivalent, i.e. there is a birational map $\beta\colon X'\dasharrow X$ such that $f'=f\circ\beta$ then $f$ is of regular type if and only if $f'$ is of regular type because they have the same associated branched covering $\rho:N\to Y$.

Theorem~\ref{Namba-pull-back}  states that every Galois branched covering over a projective manifold with Galois group $G$ is birationally equivalent to a
certain pull-back of $\p^{n}\to\p^{n}/G$, for some monomorphism $G\hookrightarrow\mr{Aut}(\p^{n})$. Therefore, it is of particular interest to study rational maps with source space the projective space $\p^{n}$. 
Combining all the previous results we obtain the following semi-local characterization of Galois rational maps $f:\p^{n}\dasharrow Y$
that generalizes the one-dimensional Greenberg criterion \cite[Lemma~1]{G} to arbitrary dimension:

\begin{teo}\label{GG}
Let $Y$ be a connected complex projective manifold and let $f:\p^{n}\dasharrow Y$ be a dominant rational map.
Then, $f$ is of regular type if and only if $f$ is Galois.
\end{teo}

\begin{dem}
It suffices to prove that if $f$ is of regular type then $f$ is Galois. 
Put $X=\p^{n}$, let $\beta:\wt X\to X$ be a birational map such that $\wt f=f\circ\beta\colon\wt X\to Y$ is a desingularization of $f$ and let $\rho:N\to Y$ be the associated finite branched covering obtained by the Stein factorization of $\wt f:\wt X\stackrel{\gamma}{\to}N\stackrel{\rho}{\to}Y$, see Proposition~\ref{rho}. 
Recall that for a rational map $f$, being of regular type means that $\rho$ is a branched covering of regular type. 
Set $Y'=Y\setminus\mr{Sing}(\Delta_{\rho})$ and $\rho'\colon N'\to Y'$ the restriction of $\rho$ to $N'=\rho^{-1}(Y')$. 
Let $\sigma$ be the composition of the normalization $Z$ of $N'\times_{Y'}N'$ and the projection on the first factor. 
Since $Y\setminus Y'$ and $N\setminus N'$ have codimension $\ge 2$ we have that $\rho'$ is also of regular type and by Proposition~\ref{Galreg} the map $\sigma:Z\to N'$ is an unbranched covering.
Consider the Zariski open subset $V=Y\setminus\Lambda_{f}$ which is contained in $Y'$ because $\Delta_{\rho}=\Delta_{f}\subset\Lambda_{f}$. 
Let $\ell\subset X=\p^{n}$ be a straight line  avoiding the codimension $\ge 2$ subsets $\Sigma_{f}$ and $(f_{|X\setminus\Sigma_{f}})^{-1}(\mr{Sing}(\Delta_{\rho}))$ which is also $C$-admisible, where $C$ is the Zariski closed set $X\setminus X_{V}$.
Notice that the birational map $\gamma\circ\beta^{-1}:X\dasharrow N$ restricts to a well-defined morphism $\ell\to N'$ and 
also restricts to a biholomorphism $\phi:X_{V}\to N_{V}$. Consider the covering $W\to\ell$ pull-back of $\sigma:Z\to N'$ by $(\gamma\circ\beta^{-1})_{|\ell}$, which is trivial because $\ell\simeq\p^{1}$.
Since $X_{V}\times_{V}X_{V}$ is smooth we can identify it with a Zariski open subset of $Z$ using the biholomorphism $\phi:X_{V}\to N_{V}$. The restricted covering $W'\to\ell':=\ell\cap X_{V}$ is the pull-back of the covering $X_{V}\times_{V}X_{V}\to X_{V}$ by the inclusion $\ell'\subset X_{V}$. Since $\ell$ is $C$-admissible, the morphism $\pi_{1}(\ell')\to\pi_{1}(X_{V})$ is surjective. By Proposition~\ref{1}, the covering $X_{V}\times_{V}X_{V}\to V$ is also trivial. By the characterization~\eqref{4} of Definition~\ref{def-gal} we have that $f_{V}$ is Galois. We conclude that $f$ is Galois by applying Theorem~\ref{car-gal}.
\end{dem}

It follows from the description of the local generators of the monodromy group given in subsection~\ref{4.1} that if  the rational map $f$ is extremal of degree $d$, then its monodromy group contains a $d$-cycle and we obtain the following result.

\begin{cor}\label{corext}
Every dominant rational map $f:\p^{n}\dasharrow Y$ of extremal type is Galois with cyclic monodromy group. Moreover, every Galois rational map of prime degree is of extremal type.
\end{cor}

However, there are examples of cyclic Galois rational maps that not are of extremal type.

\begin{ex}
If $f_{i}:X_{i}\dasharrow Y_{i}$ are rational Galois maps of degree $d_{i}>1$ with cyclic monodromy group, $i=1,2$, and $\gcd(d_{1},d_{2})=1$ then $f=f_{1}\times f_{2}:X_{1}\times X_{2}\dasharrow Y_{1}\times Y_{2}$ is Galois and $\mr{Deck}(f)\simeq\mr{Deck}(f_{1})\oplus\mr{Deck}(f_{2})=\Z_{d_{1}}\oplus\Z_{d_{2}}\simeq\Z_{d} $ with $d=d_{1}d_{2}=\deg f$ but the ramification indices of $f$ are  all of them $\le\max(d_{1},d_{2})<d$.
\end{ex}

A natural class of rational maps to be considered is that of dominant rational maps $f\colon\p^{n}\dasharrow\p^{n}$. In that case, the straight lines $\ell\in\mb G^{n}_{1}$ are the hyperplane curves of $\p^{n}$. 

Although a branched covering of regular type is not necessarily Galois, as it is shown in Example~\ref{contra}, and despite that in general $f^{-1}(\ell)$ is not a rational curve, using Theorem~\ref{dim-red}, we  have:

\begin{cor}\label{grtl}
Let $f\colon\p^{n}\dasharrow\p^{n}$ be a dominant rational map and let $\rho:N\to\p^{n}$ be its associated branched covering. If $\ell\subset\p^{n}$ is a generic line then the one-dimensional reduction $\rho_{\ell}\colon N_{\ell}\to\ell$ of $\rho$ can be identified to the map
$f_{\ell}^{\nu}\colon f^{-1}(\ell)^{\nu}\to\ell$. 
It  satisfies the following property:
$$f_{\ell}^{\nu} \textrm{ regular type }  \Leftrightarrow f \textrm{ regular type } \Leftrightarrow f \textrm{ Galois }\Leftrightarrow f_{\ell}^{\nu}  \textrm{ Galois.}$$
\end{cor}
\begin{defin}\label{gf}
The \emph{genus} $\mf g_{f}$ of a dominant rational map $f\colon X\dasharrow\p^{n}$ is the geometric genus of the curve $f^{-1}(\ell)$ for a generic straight line $\ell\subset\p^{n}$. 
\end{defin}

\begin{ex}\label{ebt}
If a dominant rational map $f\colon\p^{n}\dasharrow\p^{n}$ 
has extremal weighted branching type $\mf b_{f}^{w}=c(d)_{1}$ then $\mf g_{f}=\frac{(c-2)(d-1)}{2}$ by \eqref{RHF}.
\end{ex}

In this context, the simplest case is that of rational maps $f\colon\p^{n}\dasharrow\p^{n}$ of genus zero. By Theorem~\ref{dim-red} the study of the Galois property in this case reduces to the one-dimensional situation $f:\p^{1}\to\p^{1}$, which is completely understood. If we regard $\p^{1}$ as the unit sphere $\mb S^{2}$, then the deck transformation group of $f$ is conjugate to a finite subgroup of the group $\mr{SO}_{3}=\mr{PSU}_{2}$, which is the maximal compact subgroup of $\mr{PSL}_{2}(\C)$ and
whose finite subgroups are well-known: cyclic, dihedral, tetrahedral, octahedral and icosahedral. In fact, for each finite subgroup $G$ of $\mr{PSL}_{2}(\C)$ there is a Galois branched covering $f:\p^{1}\to\p^{1}$ whose deck transformation group (also called Galois group) is just $G$. More precisely, the following classification goes back to Klein \cite[Chapter IV]{Klein}, see also \cite[Theorem~3.6.2, pp. 43--44 and 65--66]{Shurman} for a modern exposition:

\begin{teo}\label{klein}
Let $f:\p^{1}\to\p^{1}$ be a degree $d$ Galois rational map with deck transformation group $G$. Then $f$ (resp. $G$) is left-right-equivalent (resp. conjugated) to one of the rational functions (resp. triangular subgroups of $\mr{SO}_{3}\subset\mr{PSL}_{2}(\C)$) appearing in Table~\ref{tabla},
\begin{table}[ht]
\begin{tabular}{|c|c|c|c|c|c|c|}
\hline
  & $d$ & $f$ & $\mf b_{f}^{w}$ & $G$  & $\sigma(z)$ & $\tau(z)$\\
\hline\hline
Cyclic & $n$ & $f_{C_{n}}$& $2(n)_{1}$ & $C_{n}=T(1,n,n)$  & $z$ & $\zeta_{n}z$\\[1mm]
\hline
Dihedral & $2n$ & $f_{D_{n}}$ & $2(2)_{n}+(n)_{2}$  & $D_{n}=T(2,2,n)$ & $\frac{1}{z}$& $\frac{\zeta_{n}}{z}$\\[1mm]
\hline
Tetrahedral & $12$ & $f_{T}$ & $(2)_{6}+2(3)_{4}$ & $A_{4}=T(2,3,3)$ & $-z$ & $\frac{z+i}{z-i}$\\[1mm]
\hline
Octahedral & $24$ & $f_{O}$ & $(2)_{12}+(3)_{8}+(4)_{6}$ & $S_{4}=T(2,3,4)$ & $\frac{iz-1}{z-i}$ & $\frac{z+i}{z-i}$ \\[1mm]
\hline
Icosahedral & $60$ & $f_{I}$ & $(2)_{30}+(3)_{20}+(5)_{12}$ & $A_{5}=T(2,3,5)$ & $\frac{\phi-z}{\phi z+1}$ & $\frac{(\phi-z)\zeta_{5}}{\phi z+1}$\\[1mm]
\hline
\end{tabular}\\[2mm]
\caption{\mbox{Klein's classification of Galois rational functions on $\p^{1}$
.}}\label{tabla}
\end{table}
where\\
 \centerline{$T(p,q,r)=\langle\sigma,\tau\,|\,\sigma^{p}=\tau^{q}=(\sigma\tau)^{r}=1\rangle$,\quad  $\zeta_{n}=e^{\frac{2i\pi}{n}}$,\quad  $\phi=\frac{\sqrt{5}-1}{2}$\quad  and}
$$f_{C_{n}}(z)=z^{n},\quad f_{D_{n}}=\frac{(z^{n}+1)^{2}}{4z^{n}},\quad f_{T}(z)=\left(\frac{z^{4}+2i\sqrt{3}\,z^{2}+1}{z^{4}-2i\sqrt{3}\,z^{2}+1}\right)^{3},$$
$$f_{O}(z)=\frac{(z^{8}+14z^{4}+1)^{3}}{108z^{4}(z^{4}-1)^{4}},\quad f_{I}(z)=\frac{(z^{20}-228z^{15}+494z^{10}+228 z^{5}+1)^{3}}{-1728z^{5}(z^{10}+11z^{5}-1)^{5}}.$$
\end{teo}

\section{Families of rational maps}\label{fam}

In this section we consider families of rational maps with the aim of determine
the structure of the set of Galois maps in the family.
We show that this set is always constructible and Zariski closed when the target family is a $\p^{n}$-bundle over the parameter space. In that case the genus of each element of the family is well defined and we prove that it is Zariski lower semi continuous.

We begin by recalling the notion of family of compact complex manifolds.
\begin{defin}
A family of manifolds over $T$ is a surjective proper submersion $\pi:Y\to T$ with connected fibres between connected complex manifolds.
\end{defin}

Notice that by Ehreshman's lemma, $\pi$ is a locally trivial $C^{\infty}$ fibre bundle. We can thought the family $\pi:Y\to T$ as the collection $\{Y_{t}\}_{t\in T}$ of fibres $Y_{t}=\pi^{-1}(t)$.

From now on $X$, $Y$ and $T$ will be quasi-projective manifolds and the maps occurring between them will be algebraic maps.

\begin{defin}
A family of dominant rational maps of constant topological degree over $T$ is a dominant rational map $f:X\dasharrow Y$ and two families of manifolds $\pi_{X}:X\to T$ and $\pi_{Y}:Y\to T$ such that $\pi_{X}=\pi_{Y}\circ f$, $X_{t}\not\subset\Sigma_{f}$  and the restricted rational map $f_{t}:X_{t}\dasharrow Y_{t}$ has topological degree $\deg f$ for all $t\in T$.
\end{defin}

\begin{obs}\label{pb}
Let $f\colon X\dasharrow Y$ be a family of rational maps over $T$. Given a morphism $\delta\colon\wt T\to T$ the pull-back of the family $f$ by $\delta$ is well defined and it is given by
$\wt f=f\times\mr{id}_{\wt T}:\wt X=X\times_{T}\wt T\to\wt Y=Y\times_{T}\wt T$.
Notice also that there is natural identification between the maps $f_{\delta(s)}$ and $\wt f_{s}$ for  $s\in\wt T$ preserving in particular the Galois character. 
\end{obs}

\begin{prop}\label{TB}
For every family of dominant rational maps of constant topological degree $f:X\dasharrow Y$ over $T$ there is a Zariski open subset $T'\subset T$ such that the finite branched coverings $\rho_{t}:N_{t}\to Y_{t}$ associated to $f_{t}$ are pairwise topologically equivalent for $t\in T'$. In particular, for $t\in T'$ the weighted branching type $b_{f_{t}}^{w}$ is constant and the monodromy groups of $f_{t}$ are all conjugated. Moreover the subset
$$\mr{Gal}(T):=\{t\in T\,|\, f_{t} \text{ is Galois}\}\subset T$$
is constructible.
\end{prop}

\begin{dem}
By applying Proposition~\ref{Vb} to the morphisms $X\setminus\Sigma_{f}\to Y\to T$
we obtain Zariski open subsets $X'\subset X$, $Y'\subset Y$ and $T'\subset T$ such that the restrictions  $f'_{t}\colon X'_{t}\to Y'_{t}$ of $f'\colon X'\to Y'$  for $t\in T'$ are pairwise topologically equivalent coverings by homeomorphisms $Y'_{t}\to Y'_{t'}$ extending to $Y_{t}\to Y_{t'}$.
We deduce the first assertion by applying Theorem~\ref{clas}.
Since the monodromy groups of $\rho_{t}$ for $t\in T'$ are all conjugated, we obtain the following dichotomy: either $T'\subset\mr{Gal}(T)$ or $\mr{Gal}(T)\subset T\setminus T'$.
Considering the pull-back families (cf. Remark~\ref{pb}) of $f:X\dasharrow Y$ by desingularizations $\delta_{i}:\wt T_{i}\to T_{i}\subset T$ of the irreducible components $T_{i}$ of the closed Zariski subset $T\setminus T'$ and reasoning by induction on $\dim T$ (the $0$-dimensional case being trivial) we deduce that $\mr{Gal}(T)\setminus T'=\bigcup_{i}\delta_{i}(\mr{Gal}(\wt T_{i}))$ is constructible by Chevalley theorem. We conclude that $\mr{Gal}(T)=(\mr{Gal}(T)\cap T')\cup(\mr{Gal}(T)\setminus T')$ is also constructible thanks to the above dichothomy.
\end{dem}

\begin{ex}\label{Rd}
Let $R_{d}\simeq \p^{2d+1}$ be the projectivisation of the complex vector space of pairs of homogeneous degree $d>0$ polynomials in $x,y$. Consider the Zariski open subset $T=\{[A,B]\in R_{d},\ \gcd(A,B)=1\}$ whose complementary is given by the vanishing of the resultant of the homogeneous polynomials $A(x,y)$ and $B(x,y)$. 
The rational map $f:\p^{1}\times T\to \p^{1}\times T$ given by
$f([x,y],[A,B])=([A(x,y),B(x,y)],[A,B])$ can be thought as a family of rational maps of constant topological degree $d$ between the trivial families $\p^{1}\times T\to T$. In this situation the subset $\mr{Gal}(T)$ of $T$ consists in the orbits of the elements described in Theorem~\ref{klein} by the action of $\mr{PSL}_{2}(\C)\times\mr{PSL}_{2}(\C)$ on $T$ given by the left-right equivalence.
In that case $\mr{Gal}(T)$ is more than just a constructible subset, it is a  quasi-projective  manifold.
\end{ex}

An interesting situation to be considered is when the target family $Y$ is a $\p^{n}$-bundle over $T$. In that case we have an arithmetic well defined invariant, namely the genus $\mf g_{f_{t}}$ of $f_{t}$, i.e. the geometric genus of the curve $f^{-1}_{t}(\ell_{t})\subset X_{t}$ for a generic line $\ell_{t}\subset Y_{t}\simeq\p^{n}$.

\begin{teo}\label{TC}
Consider a family of dominant rational maps of constant topological degree $f:X\dasharrow Y$ over $T$, where $Y\to T$ is a $\p^{n}$-bundle. Then $\mr{Gal}(T)$ is a Zariski closed subset of $T$, the (abstract) Galois group is constant along each connected component of $\mr{Gal}(T)$ and the genus map $\mf g:T\to\Z^{+}$ sending $t$ to the genus of $f_{t}$ is (Zariski) lower semi-continuous.
\end{teo}

Before proving Theorem~\ref{TC} let us make some previous considerations. 
\begin{obs}\label{closed}
Let $W$ be a Zariski constructible subset of a quasi-projective manifold $T$ which is closed in the euclidian topology. Then $W$ is Zariski closed in $T$. In particular, if $f\colon X\to Y$ is a proper morphism between quasi-projective manifolds then $f$ is Zariski closed thanks to Chevalley theorem.
\end{obs}

\begin{lema}\label{semicont}
Let $T$ be a quasi-projective, irreducible variety and $\varphi: T \to \mathbb{Z}^{+}$ a map with the following property: for every irreducible closed set $S \subset T$, there exist a non empty open set $U_S \subset S$ such that $\varphi|_{U_S} \equiv \sup \varphi|_{S}$. Then $\varphi$ is lower-semicontinous.
\end{lema}

\begin{proof}
By hypothesis there is a non empty open set $U_T$ such that $\varphi|_{U_T} \equiv \sup \varphi = d_0$. We consider $U_T$ maximal open with this property and write $\varphi (T) = \{ d_0 > d_1 > \ldots > d_r \}$. Now, decompose the closed set $T \setminus U_T = S_1 \cup \ldots \cup S_r$ as union of irreducible components; by maximality it is easy to see $T \setminus U_T = \{  \varphi < d_0 \}$. Applying the same argument to each $S_j$ it is straightforward to conclude that $\{ \varphi <k \}$ is closed for every $k \in \mathbb{Z}^{+}$.
\end{proof}

If $Y\to T$ is a $\p^{n}$-bundle we can consider its associated grassmannian bundle $\pi_{T}\colon\mb G^{Y}\to T$ with fibre $\mb G=\mb G^{n}_{1}$, the space of lines in $\p^{n}$.

\begin{lema}\label{aux}
Let $Y\to T$ be a $\p^{n}$-bundle and let $V\subset T$ and $\Lambda\subset Y$ be Zariski closed subsets such that $\Lambda_{t}:=\Lambda\cap Y_{t}$ is a proper subset of $Y_{t}$ for all $t\in T$.
For each point $t_{0}\in V$ there is an analytic curve $\gamma:\mb D\to \mb G^{Y}$, $\gamma(z)=(t_{z},\ell_{z})$, such that the line $\ell_{z}\subset Y_{t_{z}}$ is a $\Lambda_{t_{z}}$-admisible curve for all $z\in\mb D$, transverse to $\Lambda$ if $z\in\mb D^{*}$ and the curve $\Gamma=\mr{Im}(\pi_{T}\circ\gamma)\subset T$ is smooth and $\Gamma\cap V=\{t_{0}\}$. 
\end{lema}

\begin{dem}
Consider the subset $W_{0}$ of  pairs $(t,\ell)\in \mb G^{Y}$ such that $\ell$ is not transverse to $\Lambda$ or $\pi_{1}(\ell\setminus\Lambda_{t})\to\pi_{1}(Y_{t}\setminus\Lambda_{t})$ is not an epimorphism. 
By applying Theorem~\ref{V} to the family 
$$(Y\times_{T}\mb G^{Y},\Lambda\times_{T}\mb G^{Y},\{(y,t,\ell)\,|\,y\in\ell\subset Y_{t}\})\to\mb G^{Y}$$ 
as in the proof of Theorem~\ref{dim-red}, we deduce that the Zariski closure $\ov W_{0}$ of $W_{0}$ is a proper subset of $\mb G^{Y}$.
We decompose $\ov W_{0}=W_{0}^{h}\cup W_{0}^{v}$, where $W_{0}^{v}$ is the maximal $\pi_{T}$-saturated subset of $\ov W_{0}$. Consider the Zariski closed subset $T_{1}=\pi_{T}(W_{0}^{v})\subset T$ and the subset $W_{1}$ of pairs $(t,\ell)\in\pi_{T}^{-1}(T_{1})$ such that $\ell$ is not transverse to $\Lambda\cap\pi^{-1}_{T}(T_{1})$ or $\pi_{1}(\ell\setminus\Lambda_{t})\to\pi_{1}(Y_{t}\setminus\Lambda_{t})$ is not an epimorphism. Again by Theorem~\ref{V} the Zariski closure  $\ov W_{1}$ is a proper subset of $\pi_{T}^{-1}(T_{1})$ that we can decompose as $W_{1}^{h}\cup W_{1}^{v}$ for a maximal $\pi_{T}$-saturated subset $W_{1}^{v}$. We continue to define $W_{i}$ inductively until $W_{\dim T}^{v}=\emptyset$ thanks to Theorem~\ref{dim-red}.
Consider the Zariski closed subset $W:=\bigcup\limits_{i=0}^{\dim T}W_{i}^{h}\subset\mb G^{Y}$. By construction the restriction of $\pi_{T}$ to $\mb G^{Y}\setminus W$ is surjective. Let $U$ be a neighborhood of $t_{0}$ in $T$ such that $Y_{U}\simeq U\times\p^{n}$ and $\mb G^{Y}_{U}\simeq U\times\mb G$.
Take a regular parametrization $z\mapsto t_{z}$ of a smooth analytic curve $\Gamma\subset U\subset T$ such that $\Gamma\cap (V\cup T_{1})=\{t_{0}\}$.
Choose a point $\ell_{0}\in\mb G$ such that $(t_{0},\ell_{0})\in(U\times\mb G)\setminus W$. Then $\gamma(z):=(t_{z},\ell_{0})\notin W$, for $|z|$ small enough, satisfies the desired properties. 
\end{dem}

\begin{proof}[Proof of Theorem~\ref{TC}]
By Proposition~\ref{TB}, $\mr{Gal}(T)=\bigcup_{i=1}^{r}T_{i}$ where $T_{i}\subset T$ is quasi-projective and irreducible. Take a desingularization $\delta_{i}:\wt T_{i}\twoheadrightarrow\ov{T}_{i}\subset T$ of the closure $\ov T_{i}$  of $T_{i}$ in $T$ and consider the corresponding pull-back family over $\wt T_{i}$, whose generic element is Galois (more precisely $\delta_{i}^{-1}(T_{i})\subset\mr{Gal}(\wt T_{i})$).

On the other hand, for every  irreducible Zariski closed subset $S$ of $T$ we 
fix a desingularization $\delta_{S}\colon\wt S\twoheadrightarrow S\subset T$ 
of $S$ and we consider the pull-back family $\wt f$ of $f$ by $\delta_{S}$. Consider the Zariski open subset $\wt S'$ of $\wt S$ given by Proposition~\ref{TB} along which the genus of $\wt f_{\wt s}$ is constant. Then the genus of $f_{s}$ is constant along a Zariski open subset $U_{S}\subset\delta_{S}(\wt S')\subset S$.

In both situations, taking $T=\wt T_{i}$ or $T=\wt S$, we are in the case that $T'\subset\mr{Gal}(T)$ and the genus $\mf g_{t}$ of $f_{t}$ is constant along a Zariski open subset $T'$ of $T$. 
It is sufficient to prove that if $t_{0}\in T\setminus T'$ then $f_{t_{0}}$ is Galois with the same abstract monodromy group than $f_{t}$ and that $\mf g_{t_{0}}\le \mf g_{t}$ thanks to Lemma~\ref{semicont}.

We fix a desingularization $\wt f:\wt X\to Y$ of $f$ where $\beta:\wt X\to X$ is a composition of blow-ups centered in $\Sigma_{f}\subset X$ and $\wt f=f\circ \beta$. Then $\wt f$ is proper (because $\beta$ and $\pi_{X}$ are proper) and
 $\Lambda:=\wt f(\nabla_{\wt f}\cup\nabla_{\beta})$ is a Zariski closed subset of $Y$ thanks to Remark~\ref{closed}. Notice that the restriction of $\wt f$ to $\wt f^{-1}(Y\setminus\Lambda)$ is a covering onto its image.
 Moreover, thanks to the constancy of the topological degree of $f_{t}$ we have that $\Lambda_{t}:=\Lambda\cap Y_{t}$ is a proper subset of $Y_{t}$ containing~$\Lambda_{f_{t}}$. 

By applying Lemma~\ref{aux} to the Zariski  closed  sets $V=T\setminus T'$ and $\Lambda$, for each $t_{0}\in V$ we obtain a curve $\Gamma$ passing through $t_{0}$ and a $\p^{1}$-bundle
$$\mc L:=\{y\in Y\,|\,\exists z\in\mb D,\ y\in \ell_{z}\subset Y_{t_{z}}\}$$ 
over $\mb D\simeq\Gamma\subset T$. Since $\mb D$ is Stein $\mc L\simeq\mb\p^{1}\times\mb D$. 

Let $g:\mc M\to\mc L$ be the composition of the normalization $\nu$ of $\wt f^{-1}(\mc L)$ and the restriction of $\wt f$. Then $g$ is a degree $d=\deg f$ branched covering because it is surjective proper and finite ($\mc L$ can be chosen to avoid the codimension $\ge 2$ subvariety $\wt f(\mc C_{\wt f})$ of $\Lambda$) and $\Delta_{g}\subset\Lambda\cap\mc L$.
Moreover $g$ satisfies the following properties:
\begin{enumerate}[(1)]
\item $\mc L_{z}\not\subset\Delta_{g}$ because $\mc L_{z}$ contains generic points of $Y_{t_{z}}$ and $f_{t_{z}}$ has topological degree $d$;
\item $\mc L_{z}$ meets transversely $\Delta_{g}$ for every $z\in\mb D^{*}$ because $\ell_{z}$ is transverse to $\Lambda\supset\Delta_{g}$;
\item for every $z\in\mb D^{*}$ the fibre $\mc M_{z}$ is a smooth curve and the restriction $g_{z}:\mc M_{z}\to\mc L_{z}\simeq\p^{1}$ is a  degree $d$ branched covering with $\Delta_{g_{z}}=\Delta_{g}\cap\mc L_{z}$ thanks to Proposition~\ref{Z}.
\end{enumerate}
Hence $g:\mc M\to\p^{1}\times\mb D$ is a degenerating family of finite branched coverings of $\p^{1}$ in the sense of \cite[\S 5]{NT}. 
Let us denote $Y'=Y\setminus\Lambda$, $X'=f^{-1}(Y')$, 
$\mc L'=\mc L\cap Y'$ and $\mc M'=g^{-1}(\mc L')$. We can identify the $d$-sheeted coverings $g\colon\mc M'\to\mc L'$ and $f\colon f^{-1}(\mc L')\to \mc L'$ via the isomorphism $\beta\circ\nu_{|\mc M'}$.
Since $X'_{t_{z}}\to Y'_{t_{z}}$ is a connected $d$-sheeted covering and 
$\ell_{z}$ is a $\Lambda_{t_{z}}$-admisible curve we deduce that $\mc M_{z}'$ is connected for every $z\in\mb D$. Furthermore $\mc M_{z}$ is a connected smooth curve for $z\in\mb D^{*}$ and $\mc M_{0}'$ determines an irreducible component of $\mc M_{0}$ and the remaining irreducible components of $\mc M_{0}$, if they exist, must be $0$-dimensional because $\mc M_{0}'\to\mc L_{0}'$ is a $d$-sheeted covering and $\deg g=d$. A  topological argument implies that $\mc M_{0}$ must be connected because $\mc M$ is connected, the map $\mc M\to \mb D$ is proper and $\mc M_{z}$ is connected for $z\in\mb D^{*}$. Hence $\mc M_{0}$ is an irreducible curve.
 
From \cite[Theorem~5]{NT} and Theorem~\ref{dim-red} we deduce that the monodromy group $\mr{Mon}(g_{0}^{\nu})=\mr{Mon}(f_{t_{0}})$ of $g_{0}^{\nu}:\mc M_{0}^{\nu}\to\mc L_{0}$ injects canonically into the monodromy group $\mr{Mon}(g_{z})=\mr{Mon}(f_{t_{z}})$ of $g_{z}$ for $z\in\mb D^{*}$. 
Since $f_{t_{z}}$ is Galois for $z\in \mb D^{*}$ the monodromy group $\mr{Mon}(f_{t_{z}})$ has $d$ elements. Since $\mc M_{0}$ is irreducible, the monodromy group $\mr{Mon}(g_{0}^{\nu})\subset\mf S_{d}$ acts transitively on $\{1,\ldots,d\}$. Hence $\mr{Mon}(f_{t_{0}})=\mr{Mon}(f_{t_{z}})$ and consequently $f_{t_{0}}$ is Galois with the same monodromy group as $f_{t}$.  

On the other hand, from \cite[Theorem~4]{NT} it follows that $\chi(\mc M_{z})\le\chi(\mc M_{0})$ for $z\in\mb D^{*}$. Then
$$2-2\mf g_{t_{z}}=\chi(\mc M_{z})\le\chi(\mc M_{0})=\chi(\mc M_{0}^{\nu})-\sum_{x\in \mc M_{0}}(\beta_{x}-1)\le\chi(\mc M_{0}^{\nu})=2-2\mf g_{t_{0}},$$
where $\beta_{x}$ is the number of branches of the irreducible curve $\mc M_{0}$ at $x$.
\end{proof}

\section{Foliations and webs}\label{SFW}

Given a codimension one holomorphic foliation $\mathcal F$ on the projective space $\p^n$, its associated Gauss map
$\G_{\mathcal F}$ induces a well defined web $\operatorname{Leg}\mathcal F$ on the dual space $\pd^n$ which is called the Legendre
transform of $\mathcal F$, provided that $\G_{\F}$ is dominant. In this section we study the direct image of foliations and webs by more general rational maps. In particular we deduce that the foliation $\mathcal F$ is Galois, which means that the Gauss map $\G_{\mathcal F}$ is Galois, if and only the web
${\G}^\ast_{\mathcal F}\operatorname{Leg}\mathcal F$ is totally decomposable. This criterion will be the starting point of the discussion of Galois foliations on the projective plane carried out in Section~\ref{S6}.

We begin by recalling the notion of web given for instance in \cite[\S1.3.1 and \S1.3.3]{PP}.

\begin{defin}\label{Dweb}
For a positive integer $k$, a codimension one  $k$-web $\W$ on a complex manifold $Y$ is  given by an open cover $\{V_{i}\}$ of $Y$ and $k$-symmetric forms $\omega_{i}\in\mr{Sym}^{k}\Omega^{1}_{Y}(V_{i})$ subject to the conditions 
\begin{enumerate}[(a)]
\item for each non-empty intersection $V_{i}\cap V_{j}$ there exists a non-vanishing function $g_{ij}\in\mc O_{Y}^{*}(V_{i}\cap V_{j})$ such that $\omega_{i}=g_{ij}\omega_{j}$;
\item  the zero set $\mr{Sing}(\omega_{i})$ of $\omega_{i}$ has codimension at least two;
\item  the germ of $\omega_{i}$ at every generic point of $V_{i}$ is a product of $k$  germs of  integrable $1$-forms $\omega_{i\alpha}$, $\alpha=1,\ldots,k$
that are not collinear two by two. 
\end{enumerate} 
The subset of $Y$ where the non-collinearity condition fails is called the \emph{discriminant} of $\W$ and it is denoted by $\Delta(\W)$. The \emph{singular set} $\Sigma_{\W}$ of $\W$ is defined by $\Sigma_{\W}\cap V_{i}=\mr{Sing}(\omega_{i})$ and it is contained in $\Delta(\W)$.
\end{defin}

Notice that for $k=1$ we recover the usual definition of (singular) codimension one foliation $\F$ (see \cite{Brunella,CCD}). In that case $\Delta(\F)=\Sigma_{\F}$ is just the singular set of $\F$.
For arbitrary $k\ge 2$, a $k$-web always looks like locally  as the superposition of $k$ foliations, but not necessarily globally. If this is the case we say that the web is \emph{totally decomposable}. In fact, there is a \emph{monodromy representation} $\mu_{\W}:\pi_{1}(Y\setminus\Delta(\W))\to\mf S_{k}$ of $\W$ which determines the irreducible subwebs of $\W$ and whose triviality is equivalent to the total decomposability of $\W$ (see \cite[\S1.3.3 and \S1.3.4]{PP}).
Condition~(c) allows us to define the \emph{tangent set} $T_{y}\W$ of $\W$ at a  point $y\in U_{i}\setminus\Delta(\W)$ as the union of the $k$ different kernels at $y$ of the linear factors of $\omega_{i}(y)$. 

\begin{obs}\label{primeW}
Let $L\to Y$ be the line bundle associated to the cocycle $\{g_{ij}\}$.
The collection $\{\omega_i\}$ defines an element in  $H^0(Y,\mr{Sym}^k\Omega_Y^1\otimes L)$ which can be interpreted as a meromorphic $k$-symmetric form $\omega$ on $Y$.
Condition~(c) implies that the prime decomposition of $\omega=\prod_{\alpha}\omega_{\alpha}$ is reduced. Each prime factor $\omega_{\alpha}$ defines an irreducible web $\W_{\alpha}$ on $Y$ such that $\W$ is the superposition of the webs $\W_{\alpha}$.
\end{obs}

\subsection{Developing a web}\label{S4.1}

If $f:X\dasharrow Y$ is a dominant rational map and $\W$ is a $k$-web on $Y$
then the \emph{inverse image}  (or \emph{pull-back}) $f^{*}\W$ of $\W$ by $f$ is a well-defined $k$-web on $X$. Outside the indeterminacy locus $\Sigma_{f}$, $f^{*}\W$ is determined by the pull-back of the symmetric forms $\eta_{i}$ defining $\W$. It extends to $\Sigma_{f}$ by means of Levi's extension theorem.

The \emph{direct image} (or \emph{push-forward}) of a web $\W$ by a rational map $f$ is not defined in general. It is only defined for dominant rational maps and webs fulfilling some generic conditions. Let us consider first the case $\W$  is just a foliation. To this end we introduce the following notion.

\begin{defin}\label{gp}
We say that a holomorphic foliation $\F$ on $X$ is in \emph{general position} with respect to a dominant rational map $f:X\dasharrow Y$, or that $\F$ is \emph{$f$-general}, if for generic $y\in Y$ the set of tangent spaces $\{df_x(T_x\F)\,|\,x\in f^{-1}(y)\}$ has exactly $\deg f$ elements.
\end{defin}

Clearly, the set of $f$-general foliations is open. The following result shows that it is non-empty.

\begin{prop}\label{existgen}
For every dominant rational map  $f:X\dasharrow Y$  between projective manifolds of the same dimension $n\ge 2$ there exists a codimension one $f$-general foliation on $X$.
\end{prop}

\begin{dem} 
Fix $y_{0}\in Y\setminus\Delta_{f}$ and consider the fibre \mbox{$f^{-1}(y_{0})=\{x_{1},\ldots,x_{d}\}\subset X$.} Let us fix an embedding $X\subset\p^{N}$ and let us consider 
 an affine chart $\mb A^{N}\subset\p^{N}$ containing the points $x_{i}$, $i=1,\ldots,d$.
There exists a linear projection $\bar g:\mb A^{N}\to\mb A^{2}$ such that $p_{i}=\bar g(x_{i})$ are pairwise different points and $\ker  d\bar g_{x_{i}}+T_{x_{i}}X=T_{x_{i}}\mb A^{N}$. Let $g:X\dasharrow\p^{2}$ be the restriction of $\bar g$ to~$X$.
Consider the codimension two subspaces 
$\ell_{j}:=df_{x_{j}}(\ker dg_{x_{j}})\subset T_{y_{0}}Y$. For each $j=1,\ldots,d$, we fix pairwise different codimension one subspaces $h_{j}$ of $T_{y_{0}}Y$ containing $\ell_{j}$. Consider the one-dimensional subspaces $r_{j}=dg_{x_{j}}(df_{x_{j}}^{-1}(h_{j}))\subset T_{p_{j}}\p^{2}$. 
We fix affine coordinates $(u,v)$ on $\mb A^{2}\subset\p^{2}$ such that $p_{j}=(u_{j},v_{j})$ with $u_{i}\neq u_{j}$ if $i\neq j$ and $r_{j}$ has equation $v=a_{j}u+b_{j}$ with $a_{j},b_{j}\in\C$. Let $p(u)$ be a polynomial such that $p(u_{j})=a_{j}$.
The vector field $\partial_{u}+p(u)\partial_{v}$ defines a foliation $\F_{0}$ on $\p^{2}$  such that $T_{p_{j}}\F_{0}=r_{j}$.
Then $\F=g^{*}\F_{0}$ is a $f$-general foliation on $X$ because $df_{x_{j}}(T_{x_{j}}\F)=h_{j}\subset T_{y_{0}}Y$ are pairwise different subspaces.
\end{dem}

\begin{prop}\label{direct}
Let $f:X\dasharrow Y$ be a dominant rational map of degree $d$ between projective manifolds of the same dimension and let $\F$ be a codimension one holomorphic $f$-general foliation on $X$. There is a unique $d$-web 
$f_{*}\F$ on $Y$, called the \emph{direct image} of $\F$ by $f$,
such that $T_{y}f_{*}\F=\bigcup\limits_{x\in f^{-1}(y)}d f_{x}(T_{x}\F)\subset T_{y}Y$ for generic $y\in Y$. 
\end{prop}

\begin{dem}
We follow the ideas sketched in \cite[\S1.3.2]{PP}. 
Let $\wt f:\wt X\to Y$ be a desingularization of $f$. Using the notations introduced in Subsection~\ref{3.2}, we fix an open set $V\subset Y\setminus \wt f(\nabla_{\wt f})$ such that $f^{-1}(V)=\bigsqcup_{m=1}^{d} U_m$, $f_{|U_m}$ is bijective onto $V$ and there are holomorphic $1$-forms $\omega_m$ on $U_m$ defining $\F$. Then $\omega_V:=\prod_{m=1}^d(f|_{U_m}^{-1})^*\omega_m$ is an element of $\mr{Sym}^d\Omega^1_V$. These $d$-symmetric forms differ by a non-vanishing multiplicative function in each non-empty intersection. Since $\F$ is $f$-general they define a $d$-web $\W_0$ on $Y\setminus \wt f(\nabla_{\wt f})$. 

In order to  extend $\W_0$ to the generic point of $\Delta_{f}\subset\wt f(\nabla_{\wt f})$  we will use the local normal form (\ref{NF}) of the branched covering $\rho:N\to Y$ given by the Stein factorization  of  $\wt{f}$. According to formula~(\ref{delta2}), 
let $y\in\Delta_{\rho}=\Delta_{f}$ be a generic point and let $V\subset Y$ be a neighborhood of $y$ such that $f^{-1}(V)= \bigsqcup_{j=1}^{k} U_j$, $f|_{U_j}(z,w)=(z^{\varrho_j},w)$ and $\F|_{U_j}$ is defined by the holomorphic $1$-form $\omega_j=a_j(z,w)dz+ b_j(z,w)dw$. Then
$$\omega_V':=\prod_{j=1}^{k}\prod_{i=1}^{\varrho_j}\Big(a_j(z^\frac{1}{\varrho_j}\zeta^i_j,w)z^{\frac{1}{\varrho_j}-1}\frac{\zeta_j^i}{\varrho_j} dz+b_j(z^\frac{1}{\varrho_j}\zeta^i_j,w)dw\Big), $$
is a univalued meromorphic $d$-symmetric form on $V$, where $\zeta_j$ is a primitive $\varrho_j$-root of unity.
Multiplying $\omega_V'$ by a suitable meromophic function on $V$ we obtain a holomorphic $d$-symmetric form $\omega_V$ on $V$ with $\mr{codim}(\mr{Sing}(\omega_V))\ge 2$. These symmetric forms define an extension of $\W_0$ to $Y\setminus\wt f(\mc C_{\wt f})$. 

Finally as $\wt f(\mc C_{\wt f})$ has codimension $\ge 2$, we can extend $\W_0$ to the whole $Y$ by using the standard argument based on Levi's extension theorem for meromorphic functions (see for instance \cite[Remarque 2.17]{CCD}). 
\end{dem}

\begin{defin}\label{dev-triple}
Let $\W$ be a web on a complex projective manifold $Y$. Assume that there is a 
dominant rational map $f:X\dasharrow Y$ and a $f$-general foliation $\F$ on $X$ such that $\W=f_{*}\F$. We then say that $(X,\F,f)$ is a developing triple of the web $\W$. Two developing triples $(X,\F,f)$ and $(\wh X,\wh\F,\wh f)$ of $\W$ are said to be birationally equivalent if there exists a birational map $g\colon X\dasharrow\wh X$ such that $g^{*}\wh\F=\F$ and $\wh f\circ g=f$.
\end{defin}

The following theorem reformulates
several results stated in \cite{Pan,CL1,CL2,PP}.
For the sake of completeness we give a sketch of proof.

\begin{teo}\label{univ}
For every web $\W$ on a complex projective manifold $Y$ there is a 
developing triple $(Z_{\W},\mc C_{\W},\pi_{\W})$ of $\W$ unique up to birational equivalence. Moreover, there is a natural bijection between the irreducible factors of $\W$ and the connected components of $Z_{\W}$.
\end{teo}

\begin{dem} 
The closure $\Gamma_{\W}$ of 
$$\{(y,[\eta])\in\p T^{*}Y\,|\, y\notin\Delta(\W),\ \ker\eta\subset T_{y}\W\}$$
in $\p T^{*}Y$ 
is a projective (possibly singular) subvariety of dimension $n=\dim Y$. Fix a desingularization $\delta_{\W}:Z_{\W}\to\Gamma_{\W}\subset\p T^{*}Y$ of $\Gamma_{\W}$ which is biholomorphism outside $\mr{Sing}(\Gamma_{\W})$ and set $\pi_{\W}=\pi\circ\delta_{\W}$.
The pull-back  by $\delta_{\W}$ of the contact distribution $\mc C$ on $\p T^{*}Y$ is an integrable distribution on $Z_{\W}$ giving rise to a codimension one holomorphic foliation $\mc C_{\W}$ on $Z_{\W}$ that is in general position with respect to $\pi_{\W}$ and that satisfies $(\pi_{W})_{*}\mc C_{\W}=\W$, showing the existence of a developing triple of $\W$.

In order to show the uniqueness let us consider a dominant rational map $f:X\dasharrow Y$ with $\dim X=\dim Y$ and $\F$ a $f$-general foliation on $X$ such that $f_{*}\F=\W$. The rational map $\p T^{*}f:\p  T^{*}X\dasharrow\p T^{*}Y$ defined by $(x,[\eta])\mapsto (f(x),[\eta\circ df_{x}^{-1}])$  is dominant and preserves the contact distributions. Since $\F$ is $f$-general and $f_{*}\F=\W$ the restriction $g_{\Gamma}$ of $\p T^{*}f$ to $\Gamma_{\F}\subset\p T^{*}X$ has image contained in $\Gamma_{\W}$ and it is generically injective.
Then the composition $g:=\delta_{\W}^{-1}\circ g_{\Gamma}\circ\pi_{\F}^{-1}$
is birational and fulfills $g_{*}\F=\mc C_{\W}$.

The discriminant $\Delta(\W)$ of $\W$ contains  the set of critical values $\Delta_{\pi_{\W}}$  of~$\pi_{\W}$ and the monodromy representation 
$\mu_{\W}$ of $\W$ is the composition of the epimorphism induced by the inclusion $Y\setminus\Delta(\W)\subset Y\setminus\Delta_{\pi_{\W}}$ and the monodromy representation $\mu_{\pi_{\W}}$ of $\pi_{\W}$, cf. for instance \cite[\S1.3.3 and \S1.3.4]{PP}. In particular, the monodromy groups of $\W$ and $\pi_{\W}$ coincide and  the irreducible components of a web $\W$ on $Y$ considered in Remark~\ref{primeW} are in one to one correspondence with the connected (necessarily irreducible) components of the manifold $Z_{\W}$ of its developing triple.
Since the connected components of the total spaces of any two developing triples are in one to one correspondence, the above considerations  complete the proof.
\end{dem}

Let $\W$ be a $k$-web on $X$ and let $f\colon X\dasharrow Y$ be dominant rational map. We say that $\W$ is in general position with respect to $f$ if there is a developing triple $(Z_{\W},\mc C_{\W},\pi_{\W})$ of $\W$ such that the foliation $\mc C_{\W}$ is in general position with respect to the composition $f\circ\pi_{\W}$. In that case we define the direct image of $\W$ by $f$ as the $kd$-web $(f\circ\pi_{\W})_{*}\mc C_{\W}$. The above theorem guaranties that this definition does not depends on the choice of the developing triple of $\W$.

\begin{lema}
Let $f:X\dasharrow Y$ be a dominant rational map with $\dim X=\dim Y$ and let $\F$ be a  $f$-general foliation on $X$. Let
$\delta:Z\to X\times_{Y}X$ be a desingularization of the fibered product $X\times_{Y}X\subset X\times X$ and let  $p$ and $q$ denote the compositions of $\delta$ with the canonical projections onto the two factors $X$.
Then $(Z,q^{*}\F,p)$ is a developing triple of the web $f^{*}f_{*}\F$.
\end{lema}

\begin{dem}
Since the projections $f$ and $p$ are locally equivalent on suitable Zariski open subsets we have that $\F$ is $f$-general if and only if $q^{*}\F$ is $p$-general.
The commutativity of the diagram
$$\xymatrix{Z\ar@/^/[rrd]^{q}\ar@/_/[rdd]_{p} \ar@{->}[dr]^{\delta}& &\\ &X\times_{Y}X\ar[r]\ar[d] & X\ar@{.>}[d]^{f}\\
& X\ar@{.>}[r]^{f} & Y\,}$$
implies that $p_{*}q^{*}\F=f^{*}f_{*}\F$ provided that $\F$ is $f$-general.
\end{dem}

From Theorem~\ref{univ} and the above Lemma we obtain the following characterization.

\begin{teo}\label{webdecomp}
Let $f:X\dasharrow Y$ be a dominant rational map between complex projective manifolds of the same dimension and let $\F$ be a $f$-general foliation on $X$.
Then $f$ is Galois if and only if the web $f^{*}f_{*}\F$ is totally decomposable. In that case $f^{*}f_{*}\F$ is the superposition of the foliations $\phi^{*}\F$ varying $\phi\in\mr{Deck}(f)\subset\mr{Bir}(X)$.
\end{teo}

\begin{obs}
This result says that the decomposability of the subvariety $X\times_{Y}X\subset X\times X$ of codimension $n$ into $d=\deg f$ irreducible components is equivalent to the total decomposability of a rational \mbox{$d$-symmetric} form
defining the web $f^{*}f_{*}\F$ according to Remark~\ref{primeW}, or equivalently to the total decomposability of a  single degree $d$ polynomial in $n-1$ variables over $\C(X)$. From the computational point of view, this simplifies the problem of deciding if the rational map $f$ is Galois. From this interpretation it is clear that the case $n=2$ is very special as the question is reduced to the decomposibility of a single polynomial in one variable (cf. Proposition~\ref{CDpol}).
\end{obs}

\begin{defin}\label{def-web-gal}
A web $\W$ on $Y$ with developing triple $(Z_{\W},\mc C_{\W},\pi_{\W})$ is called Galois if the rational map $\pi_{\W}:Z_{\W}\dasharrow Y$ is Galois.  
\end{defin}

From Theorem~\ref{Namba-existencia} and Proposition~\ref{existgen} we obtain the following result.

\begin{teo}
For every finite group $G$ and every connected complex projective manifold $Y$ there is a Galois $|G|$-web on $Y$ with monodromy group isomorphic to $G$.
\end{teo}

\subsection{Foliations and webs on the projective space}

The rest of this section is devoted to treat the case $X=\p^{n}$.
In that case $\p T^{*}X$ can be canonically identified with the incidence variety 
$$\V=\{(p,h)\in\p^{n}\times\pd^{n}\,:\,p\in\check{h}\}\subset \p^{n}\times\pd^{n},$$
where $\check{h}$ (resp. $\check{p}$) is the hyperplane in $\p^{n}$ (resp. $\pd^{n}$) corresponding to the point $h\in\pd^{n}$ (resp. $p\in\p^{n}$).
By symmetry, $\mc V$ is also canonically isomorphic to $\p T^{*}\pd^{n}$.
Moreover, the contact distributions $\mc C$ of $\p T^{*}\p^{n}$ and $\p T^{*}\pd^{n}$ coincide under the identification with $\mc V$ and
\begin{equation}\label{contactpn}
\mc C_{(p,h)}=d\pi^{-1}(T_{p}\,\check{h})=d\pid^{-1}(T_{h}\,\check{p})\subset T_{(p,h)}\mc V,
\end{equation}
where  $\pi$ and $\pid$ are the restrictions to $\mc V$ of the natural projections onto $\p^{n}$ and $\pd^{n}$.

For each web $\W$ on $\p^{n}$ we consider the developing triple $(Z_{\W},\mc C_{\W},\pi_{\W})$ of $\W$ given by the composition
 $\pi_{\W}:Z_{\W}\to\p^{n}$ of a desingularization $\delta_{\W}:Z_{\W}\to\Gamma_{\W}$ of the possibly singular subvariety $\Gamma_{\W}\subset \p T^{*}\p^{n}\simeq\mc V$ considered in the proof of Theorem~\ref{univ}. 
 Let $\pid_{\W}$ be the composition of $\delta_{\W}:Z_{\W}\to\Gamma_{\W}$ and the restriction of $\pid$ to $\Gamma_{\W}\subset\mc V$. Thanks to formula~(\ref{contactpn}) we see that $\mc C_{\W}$ is in general position with respect to the projections $\pi_{\W}$ and $\pid_{\W}$, whenever they are dominant maps.

\begin{defin}
We say that a  web $\W$ on $\p^{n}$ is \emph{non-degenerate} if the map $\pid_{\W}:Z_{\W}\to\pd^{n}$ is dominant. 
In that case we can consider the web
$\leg\W:=(\pid_{\W})_{*}\mc C_{\W}$ on $\pd^{n}$ which is called the
 \emph{Legendre transform} of~$\W$.
\end{defin}

To every  web $\W$ on~$\p^{n}$ we can associate its \emph{characteristic numbers} $d_{i}(\W)$, $i=0,\ldots,n-1$, which can be defined as the number of tangency points between the leaves of $\W$ and a generic linear $i$-plane $\ell_{i}\subset\p^{n}$. More precisely (see \cite[\S1.4.1]{PP}) $d_{i}(\W)$ is the number
of pairs $(p,h)\in\p^{n}\times\pd^{n}$ such that $p\in\ell_{i}\subset h\subset\p^{n}$ and $T_{p}h\subset T_{p}\W$, for a given generic linear $i$-plane $\p^{i}\simeq\ell_{i}\subset\p^{n}$. 
Notice that $d_{0}(\W)$ counts the number of leaves of $\W$ through a generic point of $\p^{n}$, that is $\W$ is a $d_{0}(\W)$-web. 

From now on we focus on the case of foliations. 
\begin{defin}
Let $\F$ be a codimension one foliation on $\p^{n}$. The \emph{Gauss map} of $\F$ is the rational map $\G_{\F}:\p^{n}\dasharrow\pd^{n}$ defined by  $\G_{\F}(p)=T_{p}\F$, where the tangent space $T_{p}\F$ of $\F$ at a regular point $p$ of $\F$ is thought as a hyperplane of $\p^{n}$.
\end{defin}

Notice that $\G_{\F}=\pid_{\F}\circ\pi_{\F}^{-1}$. This implies that
the topological degree of $\G_{\F}$ is just $d_{n-1}(\F)$ and $\leg\F=(\G_{\F})_{*}\F$. In the case $n=2$ the topological degree of $\G_{\F}$ coincides with its usual degree $d_{1}(\F)$, i.e. the number of tangency points of the leaves of $\F$ with a generic line.

\begin{obs}\label{degfol}
The classification of \emph{degenerate} foliations, i.e. foliations whose Gauss map is not dominant, is known in dimension $n\le 4$: for $n=2$ they are of degree zero, i.e. pencils of lines, for $n=3$ see \cite{CLN} and \cite{Fas} for $n=4$.
\end{obs}

Although by Definition~\ref{def-web-gal} every foliation is a Galois $1$-web, in the sequel we will understand this notion, when applied to foliations on $\p^{n}$, according the following definition.

\begin{defin}
A non-degenerate codimension one foliation $\F$ on $\p^{n}$ is said to be Galois if the web $\leg\F$ is Galois or equivalently if the Gauss map $\mc G_{\F}$ is Galois. 
\end{defin}

From Theorem~\ref{webdecomp} we obtain:
\begin{cor}\label{5}
A non-degenerate codimension one foliation $\F$ on $\p^{n}$ is Galois if and only if the web $\mc G_{\F}^{*}\mr{Leg}\F$ is totally decomposable. In that case, $\mc G_{\F}^{*}\mr{Leg}\F$  is the superposition of the foliations $\phi^{*}\F$ with $\phi\in\mr{Deck}(\mc G_{\F})\subset\mr{Bir}(\p^{n})$.
\end{cor}

\begin{ex}
Every foliation $\F$ on $\p^{n}$ with $d_{n-1}(\F)\in\{1,2\}$ is Galois because its Gauss map $\G_{\F}$ induces a covering of degree $d_{n-1}(\F)\le 2$. Notice that, if $n\ge 3$, there are examples of such foliations with $d_{1}(\F)>2$. For instance, for each $\nu\ge 2$ consider the exceptional foliation $\mc E_{\nu}$ on $\p^{3}$ (cf. \cite{Omegar}) given in the affine chart $(x,y,z)$ by the integrable $1$-form
$\imath_{S_{\nu}}\imath_{X_{\nu}}(dx\wedge dy\wedge dz)$, where
\begin{alignat}{3}
S_{\nu}&=x\partial_{x}+
\nu y\partial_{y}+
(1-\nu+\nu^{2})z\partial_{z}\nonumber\\
X_{\nu}&=
\partial_{x}+\nu x^{\nu-1}\partial_{y}+(1-\nu+\nu^{2})y^{\nu-1}\partial_{z}.\nonumber
\end{alignat}
We have that $d_{1}(\mc E_{\nu})=\nu$ and $d_{2}(\mc E_{\nu})=\nu-1$. Then, foliations $\mc E_{2}$ and $\mc E_{3}$ are Galois but  $\mc E_{4}$ is not Galois.
To see the last  assertion, take affine charts $(x,y,z)$ and $(p,q,r)$ of $\p^{3}$ and $\pd^{3}$ such that the incidence variety has equation $z=px+qy+r$.
 The foliation
 $\mc E_{4}$ is given by the $1$-form 
 $$\omega=52\left( -{x}^{3}z+{y}^{4} \right){\it dx} + 13\left( -x{y}
^{3}+z \right) {\it dy}+ 4\left( {x}^{4}-y \right) {\it dz}
$$
and the $3$-web
$\leg(\mc E_{4})$ is given by the symmetric ternary form
\begin{align*}
\Omega=&\left( 729\,p{q}^{3}+28561\,{r}^{3} \right) {{\it dp}}^{3}-2916\,{{
\it dp}}^{2}{\it dq}\,{p}^{2}{q}^{2}-79092\,{{\it dp}}^{2}{\it dr}\,p{
r}^{2}\\&+3888\,{\it dp}\,{{\it dq}}^{2}{p}^{3}q+73008\,{\it dp}\,{{\it 
dr}}^{2}{p}^{2}r+ \left( -1728\,{p}^{4}+8788\,q{r}^{3} \right) {{\it 
dq}}^{3}\\ &-18252\,{{\it dq}}^{2}{\it dr}\,{q}^{2}{r}^{2}+12636\,{\it dq}
\,{{\it dr}}^{2}{q}^{3}r+ \left( -2916\,{q}^{4}-22464\,{p}^{3}
 \right) {{\it dr}}^{3}.
 \end{align*}
With the help of an algebraic manipulator we can check that
$\G_{\mc E_{4}}^{*}\Omega=\omega\cdot\eta$ with $\eta$ a  quadratic form that do not factorize over $\C(x,y,z)$ because the discriminant $27z^4+1024$ of the restriction 
$$\left( 4563\,{z}^{4}+302848 \right) {{\it dx}}^{2}-26\, \left( 27\,{z
}^{4}-7424 \right) {\it dz}\,{\it dx}+ \left( 27\,{z}^{4}+93952
 \right) {{\it dz}}^{2}
$$
of $\eta$ to $(x,y,dy)=(0,1,0)$ is not a square in $\C(z)$.
\end{ex}

The following result provides a new dimensional reduction that allow us to exhibit examples of Galois foliations in any dimension. It will be also used in the last section.

\begin{prop}\label{AutF}
Let $\F$ be a non-degenerate codimension one foliation on $\p^{n}$ admitting a
transverse infinitesimal symmetry $R\in\mf X(\p^{n})$ with a dominant rational first integral $f:\p^{n}\dasharrow\p^{n-1}$ whose generic fibre is irreducible. Then there is a dominant rational map $\wh\G_{\F}:\p^{n-1}\dasharrow\p^{n-1}$ such that $\mr{Deck}(\G_{\F})$ and $\mr{Deck}(\wh\G_{\F})$ are canonically isomorphic.
\end{prop}
\begin{dem}
Let $\phi_t$ the flow of homographic transformations of $\p^{n}$ associated to the vector field  $R$ and let $\check{\phi}_t$ be the dual flow on $\pd^{n}$ associated to the dual vector field $\check{R}$. Let $\phi_{t}^{\mc V}:\mc V\to\mc V$ be the flow induced by $\p T^{*}\phi:\p T^{*}\p^{n}\to\p T^{*}\p^{n}$ via the identification $\mc V=\{(p,h)\in\p^{n}\times\pd^{n}\,|\, p\in h\}\simeq\p T^{*}\p^{n}$. The fact that $\phi_{t}\in\mr{Aut}(\F)$ implies that $\phi_{t}^{\mc V}$ preserves the graph $\Gamma_{\F}\subset\mc V\subset\p^{n}\times\pd^{n}$ of the Gauss map $\G_{\F}$ of the foliation $\F$. The commutativity of the lateral faces of the diagram
$$
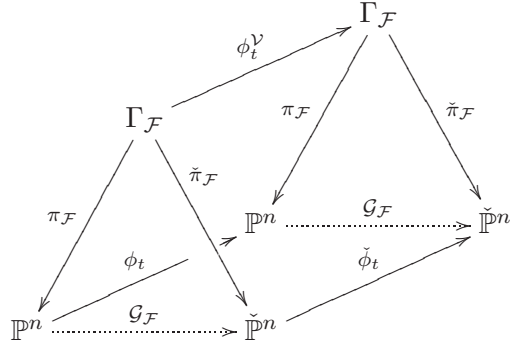
\xymatrix{
& & & & \Gamma_{\F}\ar[rdd]^{\pid_{\F}}\ar[ldd]_{\pi_{\F}} & &\\
& & \Gamma_{\F}\ar[rru]^{\phi_{t}^{\mc V}}\ar[rdd]^<<<<<<<{\pid_{\F}}\ar[ldd]_{\pi_{\F}} & & & &\\
& & & \p^{n}\ar@{.>}[rr]^{\mc G_{\F}} & & \pd^{n} &\\
& \p^{n}\ar[rru]|-->>>{\color{white}AAA}
^{\phi_{t}}\ar@{.>}[rr]^{\mc G_{\F}} & & \pd^{n}\ar[rru]^{\check{\phi}_{t}} & & &
}$$
implies that
\begin{equation}\label{equivar}
\check{\phi}_t \circ \G_{\F}=\G_{\F} \circ \phi_t.
\end{equation}
Since the one-dimensional foliation defined by the vector field $R$ admits a  dominant  rational first integral $f:\p^{n}\dasharrow\p^{n-1}$ whose generic fibre is irreducible, we deduce the existence of  $\check{f}:\pd^{n}\dasharrow\p^{n-1}$ fulfilling the same properties for the dual vector field $\check{R}$. Relation~(\ref{equivar}) implies the existence of a rational map $\wh\G_{\F}:\p^{n-1}\dasharrow\p^{n-1}$ such that the following diagram commutes:
\begin{equation}\label{sqw}
\xymatrix{\p^{n}\ar@{.>}[r]^{\G_{\F}}\ar@{.>}[d]_{f}&\pd^{n}\ar@{.>}[d]^{\check{f}}\\ \p^{n-1}\ar@{.>}[r]^{\wh\G_{\F}} & \p^{n-1}}
\end{equation}
We will finish by applying Proposition~\ref{quo-red} once we check that $\deg\G_{\F}=\deg\wh\G_{\F}$, or equivalently, that the restriction of $\G_{\F}$ to a generic fibre of $f$ is injective. To see that, fix a generic point $p\in\p^{n}$ and assume that $\G_{\F}(\phi_{t}(p))=\G_{\F}(p)$. By (\ref{equivar}) we deduce that $\check{\phi}_{t}(\G_{\F}(p))=\G_{\F}(p)$. Since $\G_{\F}(p)$ is generic then $\check{\phi}_{t}=\mr{Id}_{\pd^{n}}$ and consequently $\phi_{t}=\mr{Id}_{\p^{n}}$.
\end{dem}
\begin{cor}\label{n-dim-ex}
For each $n\ge 2$ and $k\ge 1$,  the codimension one foliation $\F$  on $\p^{n}$ given in an affine chart $(x_{1},\ldots,x_{n})$ by the polynomial first integral 
$$F(x_{1},\ldots,x_{n})=\sum\limits_{i=1}^{n}x_{i}^{k+1}$$ 
is Galois with 
$\mr{Deck}(\G_{\F})\simeq\Z_{k}^{n-1}$.
\end{cor}

\begin{dem}
The foliation $\F$ admits $R=\sum\limits_{i=1}^{n}x_{i}\partial_{x_{i}}\in\mf X(\p^{n})$ as transverse infinitesimal symmetry with rational first integral $f:\p^{n}\dasharrow\p^{n-1}$ given by $f(x_{1},\ldots,x_{n})=[x_{1},\ldots,x_{n}]$ that satisfies the hypothesis of Proposition~\ref{AutF}. Moreover, taking $(y_{1},\ldots,y_{n})$ the affine chart of $\pd^{n}$ such that $\sum\limits_{i=1}^{n}x_{i}y_{i}=1$ is an affine equation of the incidence variety $\mc V\subset\p^{n}\times\pd^{n}$ we have that $\G_{\F}(x_{1},\ldots,x_{n})=\left(\frac{x_{1}^{k}}{F},\ldots,\frac{x_{n}^{k}}{F}\right)$, $\check{f}(y_{1},\ldots,y_{n})=[y_{1},\ldots,y_{n}]$ is a rational first integral of the dual vector field $\check{R}=\sum\limits_{i=1}^{n}y_{i}\partial_{y_{i}}$ and the rational map $\wh\G_{\F}:\p^{n-1}\dasharrow\p^{n-1}$ given by
$\wh\G_{\F}([x_{1},\ldots,x_{n}])=[x_{1}^{k},\ldots,x_{n}^{k}]$, which is clearly Galois with  
$\mr{Deck}(\wh\G_{\F})\simeq\Z_{k}^{n-1}$, makes commutative the diagram~(\ref{sqw}).
\end{dem}

\begin{obs}
We can write down all the elements of the group $\mr{Deck}(\G_{\F})\subset\mr{Bir}(\p^{n})$ as follows:
$$\phi_{j}(x_{1},\ldots,x_{n})=\left(\frac{x_{1}\zeta^{j_{1}}\sum\limits_{\ell=1}^{n}x_{\ell}^{k+1}}{\sum\limits_{\ell=1}^{n}x_{\ell}^{k+1}\zeta^{j_{\ell}}},\ldots,\frac{x_{n}\zeta^{j_{n}}\sum\limits_{\ell=1}^{n}x_{\ell}^{k+1}}{\sum\limits_{\ell=1}^{n}x_{\ell}^{k+1}\zeta^{j_{\ell}}}\right),\quad j\in\Z_{k}^{n},\ j_{n}=0.$$
\end{obs}

\section{Galois foliations on the projective plane}\label{S6}

The aim of this section is to study the space of Galois foliations of degree $d$ on the complex projective plane. We begin by giving an algebraic characterization of Galois foliations based on the total decomposability of their dual webs. We use this criterion to exhibit some explicit examples in any degree. Using results of Section~\ref{fam} we show that the space of degree $3$ Galois foliations has at least two irreducible components. 
We also provide a  characterization of Galois foliations in terms of geometric elements naturally associated to them by using the main result of Subsection~\ref{RT}. We give one  necessary and one sufficient local  conditions for the Galois character of a foliation that become equivalent in the prime degree case.
Finally we obtain a full characterization of homogeneous Galois foliations which
implies in particular that the space of Galois foliations of even degree has at least two irreducible components. 
More generally we characterize Galois foliations with all possible continuous symmetries and we exhibit some examples.

\subsection{The space of Galois degree $d$ foliations on $\p^{2}$}\label{7.1}
Recall that a degree~$d$ foliation $\F$ on $\p^{2}$ is given by a $1$-form on $\C^{3}$,
$$\omega=a(x,y,z)dx+b(x,y,z)dy+c(x,y,z)dz,$$  with $a,b,c$ homogeneous polynomials of degree $d+1$ without common factors and
fulfilling $\omega(R)=ax+by+cz=0$, where $R=x\partial_{x}+y\partial_{y}+z\partial_{z}$ is the radial vector field (see for instance~\cite[\S 9.1]{CCD}).
Thus, the space $\mb F_{d}$ of degree $d$ foliations on $\p^{2}$ is a Zariski open subset of 
the projective space $\overline{\mb F}_{d}:=\mb P(U_{d})$, where
\begin{equation}\label{ud}U_{d}:=\{(a,b,c)\in\C_{d+1}[x,y,z]^{\oplus 3}\,|\, ax+by+cz=0\}
 \end{equation}
and $\C_{d}[x,y,z]$ is the vector space of degree $d$ homogeneous polynomials in $x,y,z$.
For practical purposes it will be convenient to define foliation~$\F$ in an affine chart $(x,y)$ of $\p^{2}$ by a polynomial vector field
\begin{equation}\label{X}
X=A(x,y)\partial_{x}+B(x,y)\partial_{y}=\bar a(x,y)\partial_{x}+\bar b(x,y)\partial_{y}+\bar c(x,y)(x\partial_{x}+y\partial_{y})
\end{equation}
with $\bar a,\bar b,\bar c\in\C[x,y]$, $\deg \bar a,\deg\bar b\le d$ and $\bar c$ homogeneous of degree $d$.
The vector field $X$ is said \emph{saturated} if $\gcd(A,B)=1$. This condition jointly with $\max(\deg A,\deg B)\ge d$ is equivalent to the condition $\gcd(a,b,c)=1$ defining~$\mb F_{d}$.
We consider the Gauss map $\mc G_{\F}:\p^{2}\dasharrow\pd^{2}$ of $\F$ which is written as
$$\G_{\F}([x,y,z])=[a(x,y,z),b(x,y,z),c(x,y,z)]$$ 
in homogeneous coordinates.

We can consider the family of dominant rational maps of constant topological degree $d$ 
$$\G:\p^{2}\times\mb F_{d}\dasharrow\p^{2}\times\mb F_{d},\qquad\G(p,\F):=(\G_{\F}(p),\F),$$ over the Zariski open subset $\mb F_{d}$ of the projective space $\ov{\mb F}_{d}$.
By applying Theorem~\ref{TC} we deduce that the subset
$\mb G_{d}:=\mr{Gal}(\mb F_{d})$ of $\mb F_{d}$ consisting of degree $d$ Galois foliations on $\p^{2}$ is Zariski closed.

\begin{question}\label{Q}
What are the number and type of the irreducible components of the Zariski  closed subset  $\mb G_{d}\subset\mb F_{d}$  of degree $d$ Galois foliations on $\p^{2}$?
\end{question}

Notice that according to  Theorems~\ref{TB} and~\ref{TC}
the weighted branching type $\mf b_{\F}^{w}$  and the genus $\mf g_{\F}$ of $\G_{\F}$ are generically constant and the Galois group $\mr{Deck}(\G_{\F})$ is constant along each irreducible component of~$\mb G_{d}$.

\subsection{Examples of Galois foliations on $\p^{2}$}

We begin this section by noticing that Corollary~\ref{5} on $\p^{2}$ implies the following computational criterion already considered 
in \cite{CD}  for the degree $3$ case (cf. Proposition~5.2 loc. cit.):
\begin{prop}\label{CDpol}
A foliation $\F$ on $\p^{2}$ given by the polynomial vector field $X=A(x,y)\partial_{x}+B(x,y)\partial_{y}$ is Galois if and only if the polynomial 
\begin{equation}\label{polCD}
P(x,y,t)=\det\left(\begin{array}{cc}A(x,y) &A(x+tA(x,y),y+tB(x,y))\\
B(x,y) &B(x+tA(x,y),y+tB(x,y))
\end{array}\right)\in\C[x,y,t]\end{equation} decomposes totally over the field $\C(x,y)$. In fact, each one of its rational roots $t=t(x,y)\in\C(x,y)$ determines a birational deck transformation of~$\G_{\F}$:
$$(x,y)\mapsto(x+t(x,y)A(x,y),y+t(x,y)B(x,y)).$$
In particular, if $\deg\F=3$ then $\F$ is Galois if and only if the $t$-discriminant of the  polynomial $P(x,y,t)/t$ of degree $2$ in $t$ is a square in $\C[x,y]$.
\end{prop}

Before going further with Question~\ref{Q} let us present some explicit examples. 
The following result provides continuous families $E_{d}$ of Galois foliations in each degree $d$, and all of them have cyclic monodromy group after Corollary~\ref{corext}. It would be interesting to decide if  $E_{d}$ forms an irreducible component of $\mb G_{d}$. It is worth to notice that $E_{3}$ contains as particular cases all the examples considered in~\cite{CD}.

\begin{prop}\label{2.4}
For all linearly independent vectors $(\alpha,\gamma,\lambda)$, $(\beta,\delta,\mu)\in\C^{3}$ and every $\C$-linearly independent $u,v\in\C[x,y]$  with $\deg u,\deg v\le 1$, the degree $d$ foliation $\F$ defined by the saturated vector field
$$(\alpha u^{d}+\beta v^{d})\partial_{x}+(\gamma u^{d}+\delta v^{d})\partial_{y}
+(\lambda u^{d}+\mu v^{d})(x\partial_{x}+y\partial_{y})
$$
is Galois  with extremal weighted branching type $\mf b_{\F}^{w}=2(d)_{1}$ and genus~\mbox{$\mf g_{\F}=0$.}
\end{prop}
\begin{dem}
The slope of $\F$ takes the form
$p(x,y)=
\frac{\gamma+\delta w^{d}+y(\lambda+\mu w^{d})}{\alpha+\beta w^{d}+x(\lambda+\mu w^{d})}$,  with $w=\frac{v}{u}$. 
 The roots of polynomial~(\ref{polCD}) for the vector field $X=\partial_{x}+p(x,y)\partial_{y}$ are the solutions of the equation
$p(x+t,y+tp(x,y))=p(x,y)$, which reduces to $w(x+t,y+tp(x,y))^{d}=w(x,y)^{d}$. Using that $\deg u,\deg v\le 1$, the last equation factorizes as the following $d$ linear equations in the variable $t$:
\begin{equation}\label{zeta}
w(x+t,y+tp(x,y))=\zeta^{k}w(x,y)\hspace{2mm} \text{with}\hspace{2mm} \zeta=e^{\frac{2i\pi}{d}}\hspace{2mm} \text{and}\hspace{2mm}  k\in\{1,\ldots,d\}.
\end{equation}

Let $p=(a,b)\in\C^{2}\subset\p^{2}$ be a generic point so that its dual line $\ell=\check{p}\subset\pd^{2}$ belongs to the Zariski open subset $V$ considered in Theorem~\ref{dim-red}. Then the curve $P=\G_{\F}^{-1}(\ell)\subset\p^{2}$ has  affine equation
$$F(x,y):=\left|\begin{array}{cc}x-a & (\alpha A+\beta B)(u,v)+x(\lambda A+\mu B)(u,v)\\ y-b & (\gamma A+\delta B)(u,v)+y(\lambda A+\mu B)(u,v)\end{array}\right|=0,$$
where $A(x,y)=x^{d}$ and $B(x,y)=y^{d}$.
Since $u$ and $v$ are $\C$-linearly independent polynomials of degree $\le 1$, from the equation $\frac{v}{u}=w\in\p^{1}$ we can express either $y=y_{0}(w)+y_{1}(w)x$  or $x=x_{0}(w)+x_{1}(w)y$, with $x_{i}(w),y_{i}(w)\in\C(w)$.
Without loss of generality we can assume that we are in the first situation. From equation
$F(x,y_{0}(w)+y_{1}(w)x)=0$ we obtain an explicit rational parametrization $\pi:\p^{1}\to P$ given by
\begin{align*}
x(w)&=\left.\frac{ \left(  \left( \lambda A+\mu B \right) a+\alpha\,A+\beta B
 \right) y_{{0}}+ \left( \gamma A+\delta B \right) a+ \left( -\alpha
 A-\beta B \right) b}{ -\left(  \left( \lambda A+\mu B \right) a+\alpha A+\beta B
 \right) y_{{1}}+ \left( \lambda A+\mu B \right) b+(\gamma A+\delta
 B)}\right|_{(1,w)}\\
y(w)&=y_{0}(w)+y_{1}(w)x(w).
\end{align*}
On the other hand, 
the pencil $\check{p}=\ell$ of lines through $p$ can be parametrized by $t\in\p^{1}$ by means of $\frac{y-b}{x-a}=t$. By composing $\G_{|P}:P\to\ell$ to the left by $\pi:\p^{1}\to P$ and to the right with the inverse of $\p^{1}\stackrel{\sim}{\to}\ell$ we obtain the rational map $\p^{1}\to\p^{1}$ given by 
$w\mapsto{\frac { \left( \gamma+b\lambda \right) A(1,w)+ \left( \delta+b\mu \right) 
B(1,w)}{ \left(\alpha + a\lambda\right) A(1,w)+ \left(\beta+ a\mu \right) B(1,w)}}
$, which is right equivalent to the Galois rational map $w\mapsto\frac{B(1,w)}{A(1,w)}=w^{d}$ because 
$$\left( \gamma+b\lambda \right)\left(\beta +a\mu \right)-\left( \delta+b\mu \right)\left(\alpha + a\lambda\right)=\left|\begin{array}{ccc}\alpha & \gamma & \lambda\\ \beta & \delta & \mu\\ a & b & -1\end{array}\right|\neq 0$$
if $(a,b)\in\C^{2}$ is generic. Hence $\mf b_{\F}^{w}=2(d)_{1}$ and $\mf g_{\F}=g(P)=0$.
\end{dem}

\begin{ex}\label{exg1}
The degree $3$ foliation $\F$ given by the vector field 
$$xy\partial_{x}+(\zeta y^{2}+x^{3})\partial_{y},\quad\text{with}\quad \zeta=\frac{1\pm i\sqrt{3}}{2},$$ is Galois with extremal weighted branching type $\mf b_{\F}^{w}=3(3)_{1}$ and genus $\mf g_{\F}=1$. Indeed, the $t$-discriminant  $-\zeta x^{2}y^{2}(y^{2}-x^{3})^{2}((\zeta-1)y^{2}+x^{3})^{2}$ of the polynomial $P(x,y,t)/t$ considered in Proposition~\ref{CDpol} is a square in $\C[x,y]$. Hence $\F$ is Galois of extremal type $\mf b^{w}_{\F}=c(3)_{1}$   because $\deg\F=3$ is prime. 
 On the other hand, if $p=(a,b)\in\C^{2}\subset\p^{2}$ and $\ell=\check{p}\subset\pd^{2}$ are generic then the curve $\G_{\F}^{-1}(\ell)$ is irreducible, has affine equation
$(x-a)(\zeta y^{2}+x^{3})-(y-b)xy=0$ and its geometric genus is $\mf g_{\F}=1$.
We conclude that $c=3$ by using Example~\ref{ebt}.
\end{ex}

\begin{prop}\label{G3red}
The Zariski closed set $\mb G_{3}$ of degree $3$ Galois foliations is reducible. More preci\-sely, let $C_{0}$ be an irreducible component of $\mb G_{3}$ containing  the family  $E_{3}\subset\mb G_{3}$ of genus zero Galois foliations given in Proposition~\ref{2.4} for  $d=3$ and let $C_{1}$ be an irreducible component of $\mb G_{3}$ containing the degree $3$ and genus one Galois foliation $\F_{1}$ considered in Example~\ref{exg1}. Then $C_{0}\neq C_{1}$.
\end{prop}

Before proving it let us make some previous considerations.
Recall that the vector space $U_{3}$ defined in~(\ref{ud}) is isomorphic to the space of vector fields~(\ref{X}).
 If $X$ is  such a vector field we will denote $[X]\in\overline{\mb F}_{3}$ the foliation defined by $X$. 
In order to estimate the dimension of $\mb G_{3}$
we can compute an upper bound of the dimension of the tangent space of $\mb G_{3}$ at a point $[X]\in\mb G_{3}$. To do that, we note that $\mb G_{3}$ coincides with the set of foliations $[X]\in\mb F_{3}$ such that the $t$-discriminant $\Delta_{X}=a_{2}^{2}-4a_{1}a_{3}\in\C[x,y]$ of the polynomial $P_{X}(x,y,t)/t\in\C[x,y,t]$ considered in~(\ref{polCD}), is a square, i.e. 
$\Delta_{X}=\delta_{X}^{2}$ with $\delta_{X}\in\C[x,y]$.

\begin{lema}\label{lema-dim}
If $[X]\in\mb G_{3}$ then $T_{[X]}\mb F_{3}=U_{3}/\langle X\rangle$ and
$$
T_{[X]}\mb G_{3}\subset\Big\{Y\in U_{3}\,|\,\delta_{X} \text{ divides }\frac{d}{d\varepsilon}\Big|_{\varepsilon=0}(\Delta_{X+\varepsilon Y})\in \C[x,y]\Big\}\Big/\langle X\rangle.
$$
\end{lema}

\begin{dem}
Let $V_{m}$ denote the space of polynomials in $\C[x,y]$ of degree $\le m$. Writing $P_{X}=a_{1}t+a_{2}t^{2}+a_{3}t^{3}$, it is easy to check that $a_{1}\in V_{9}$, $a_{2}\in V_{12}$ and $a_{3}\in V_{15}$, so that $\Delta_{X}\in V_{24}$. The map $\mf s:V_{12}\to V_{24}$ given by $\delta\mapsto \delta^{2}$ induces a morphism $\overline{\mf s}:\p(V_{12})\to\p(V_{24})$ whose image $\overline{S}$ is Zariski closed. Then the preimage $S$ of $\overline{S}$ in $V_{24}$ is also Zariski closed.
Let $f_{1},\ldots,f_{k}$ be generators of the ideal $I(S)$. 
Then $f_{1}\circ \Delta,\ldots, f_{k}\circ\Delta$ 
is a system of equations defining the preimage $G_{3}$ of $\mb G_{3}$ in $U_{3}$. Although we do not know whether 
$f_{i}\circ \Delta$ 
generate the ideal $I(G_{3})$, 
we have
\begin{align*}T_{X}G_{3}\subset\bigcap\limits_{i=1}^{k}\ker d(f_{i}\circ\Delta)_{X}&=\Big\{Y\in U_{3}\ \Big|\ \frac{d}{d\varepsilon}\Big|_{\varepsilon=0}\Delta_{X+\varepsilon Y}\in\bigcap_{i=1}^{k}\ker (d f_{i})_{\Delta_{X}}\Big\}\\
&= \Big\{Y\in U_{3}\ \Big|\ \frac{d}{d\varepsilon}\Big|_{\varepsilon=0}\Delta_{X+\varepsilon Y}\in T_{\Delta_{X}}S\Big\}.
\end{align*}
Consider $\Delta=\delta^{2}\in S\setminus\{0\}\subset V_{24}$ with $\delta\in V_{12}\setminus\{0\}$ and $\Gamma\in T_{\Delta}V_{24}=V_{24}$. 
Since, for $\gamma\in T_{\delta}V_{12}=V_{12}$, $d\mf s_{\delta}(\gamma)=2\delta\gamma\neq 0$ if $\gamma\neq 0$, it follows that $S\setminus\{0\}$ is smooth and consequently $T_{\Delta}S=\mr{Im}\,d\mf s_{\delta}$. 
Hence $\Gamma\in T_{\Delta}S$ if and only if  $\delta$ divides $\Gamma$. We conclude by taking the quotient by the $1$-dimensional subspace $\langle X\rangle$ of $T_{X}G_{3}$.
\end{dem}

\begin{proof}[Proof of Proposition~\ref{G3red}]
By means of Lemma~\ref{lema-dim} and an explicit computation carried out with \texttt{maple} we deduce that \mbox{$\dim T_{\F_{1}}\mb G_{3}\le 9$} and consequently $\dim C_{1}\le 9$. On the other hand, the family $E\subset\mb F_{3}$ given in Proposition~\ref{2.4} for $d=3$ is the image of an explicit morphism $\varphi:W\subset\p^{11}\to\mb F_{3}$.
It can be checked that the rank of $d\varphi$ at the point $[\alpha,\beta,\gamma,\delta,\lambda,\mu,u,v]=[1,0,0,1,0,0,x,y]$ is $9$ and consequently $\dim E\ge 9$. 
Theorem~\ref{TC} implies that $E_{1}:=\{\F\in C_{1}\,|\, \mf g_{\F}=0\}$ is a proper Zariski closed subset of $C_{1}$. If $C_{0}=C_{1}$ then $9\le \dim E\le\dim C_{0}=\dim C_{1}\le 9$ contradicting that $E$ is contained in the \emph{proper} Zariski closed set $E_{1}$. 
\end{proof}

\subsection{Geometric characterization}\label{7.3}

We address now the question of characterizing Galois foliations on $\p^2$ in terms of geometric elements naturally associated to the foliation.
Thanks to Theorem~\ref{GG} we know that a foliation $\F$ on $\p^2$ is Galois if and only if its associated Gauss map $\G\colon\p^{2}\dasharrow\pd^{2}$ is of regular type. We proceed as in subsection~\ref{3.2} and we consider a
 desingularization  $\wt\G\colon\wt\p^{2}\to\pd^{2}$ of $\G$ by blowing up $\beta:\wt\p^{2}\to\p^{2}$ the singular locus $\Sigma_{\F}$ of the foliation which coincides with the indeterminacy locus of $\G$. 
According to formulae~\eqref{delta} and~\eqref{delta2}, the birational type of the ramification locus $\mc R_{\wt\G}\subset\wt\p^{2}$ and the curve 
$$\Delta_{\F}:=\wt\G(\mc R_{\wt\G})\subset\pd^{2}$$
do not depend on the choice of the desingularization $\wt\G$. 
For each $\varrho>1$ we consider the union $\mc R_{\wt \G}^{\varrho}$ of the components of $\mc R_{\wt\G}$ having ramification index~$\varrho$.

In order to describe geometrically the components of $\mc R_{\wt\G}$ that are not included in the exceptional divisor $\mc E=\nabla_{\beta}$ of $\beta$
we proceed as follows. Let $\I$ be  the inflection locus  of the foliation $\F$ introduced in \cite{P}. It is the closure of the set of points in $\p^{2}\setminus\Sigma_{\F}$ where the leaves of $\F$ have a contact of order greater than one with its tangent line. If $\F$ is defined by a vector field $X=A(x,y)\partial_{x}+B(x,y)\partial_{y}$ in an affine chart $(x,y)$ then 
\begin{equation}\label{infafin}
f(x,y)=
\left|\begin{array}{cc}A(x,y) & B(x,y)\\ X(A(x,y)) & X(B(x,y))\end{array}\right|=0.
\end{equation}
is an affine equation for $\I$. This local description gives $\I$ a natural structure of divisor (cf. \cite{P}).
We can decompose $\I=\Ii+\It$, where the support of $\Ii$ consists in  the union of the invariant lines of $\F$ (which are collapsed by $\G$) and the support of $\It$ is the closure of the inflection points that are isolated along the leaves of $\F$. For each $\varrho>1$ we  consider the reduced (maybe empty) curves $\mc I_{\F}^{\varrho}\subset\p^{2}$ defined by the equality of divisors
$$\mc I_{\F}^{\mr{tr}}=\sum_{\varrho>1}(\varrho-1)\mc I_{\F}^{\varrho}.$$
The number $\varrho$ in $\mc I_{\F}^{\varrho}$ is the tangency order between the leaf of $\F$ through a generic point $p$ of $\mc I_{\F}^{\varrho}$ and its tangent line $\ell=T_{p}\F$, that is the number of simple tangencies bifurcating from $p$ when one perturbs $\ell$.

\begin{lema}\label{tangencia}
For each $\varrho>1$ we have $\beta(\ov{\mc R_{\wt\G}^{\varrho}\setminus\mc E})=\mc I_{\F}^{\varrho}$.
\end{lema}

\begin{dem}
Since $\beta$ is an isomorphism outside $\mc E$,
the ramification index  $\varrho$ of $\wt \G$ at a generic point $p$ of $\mc R_{\wt\G}^{\varrho}$ is just the number of local  preimages $\G^{-1}(q')$ by $\G$ collapsing to $\beta(p)\in \G^{-1}(q)$ as $q'\to q:=\wt\G(p)$, that is, the number of tangency points of order one collapsing to $\beta(p)$.
\end{dem}

Now we deal with the ramification components contained in the exceptional divisor $\mc E$. For each $s\in\Sigma_{\F}$ we set  $\mc E_{s}=\beta^{-1}(s)$ and we note that $\wt\G(\mc E_{s})=\check{s}\subset\pd^{2}$ is the dual line of $s\in\p^{2}$.  We denote $\mc E_{s}^{\mr{dom}}$ (resp. $\mc E_{s}^{\mr{ram}}$) the union of irreducible components $D$ of $\mc E_{s}$ such that $\delta_{D}:=\deg(\wt\G_{|D})>0$ (resp. with ramification index $\varrho_{D}>1$). We also set $\Sigma_{\F}^{\mr{ram}}:=\{s\in\Sigma_{\F}\,|\,\mc E_{s}^{\mr{ram}}\neq\emptyset\}$.  
We notice that $\mc E_{s}^{\mr{dom}}\neq\emptyset$ and that 
$\mc E_{s}^{\mr{ram}}\subset\mc E_{s}^{\mr{dom}}$.
For each $\varrho>1$ let us consider the subset $\Sigma_{\F}^{\varrho}\subset\Sigma_{\F}^{\mr{ram}}$ of those singularities $s$ of $\F$ such that each irreducible component of $\mc E_{s}^{\mr{dom}}$ have the same ramification index~$\varrho$.

The geometric characterization of Galois foliations is given by the following statement.

\begin{teo}\label{FolGal}
A degree $d>0$ foliation $\F$  on $\p^{2}$ is Galois if for each $p\in\Delta_{\F}
\setminus\mr{Sing}(\Delta_{\F})$ there is $\varrho|d$, $\varrho>1$ such that $\mr{Tang}(\F,\check{p})\subset(\mc I_{\F}^{\varrho}\setminus\Sigma_{\F})\cup\Sigma_{\F}^{\varrho}$.
\end{teo}
The proof will show that it suffices to test the above condition for one generic point $p$ of each irreducible component of $\Delta_{\F}\subset\pd^{2}$.

\begin{dem}
Let $\rho:N\to \pd^{2}$ be the branched covering associated to $\G$ by Proposition~\ref{rho} and recall  that $\Delta_{\F}=\wt\G(\mc R_{\wt\G})=\Delta_{\rho}\subset\pd^{2}$, see formula~\eqref{delta}. 
Thanks to Theorem~\ref{dim-red} we can choose a $\Delta_{\F}$-admisible line
$\ell\subset\pd^{2}$ and consider the one-dimensional branched covering $\rho_{\ell}:N_{\ell}=\rho^{-1}(\ell)\to\ell$.
From Corollary~\ref{grtl} we deduce that 
$\F$ is Galois if and only if $\rho_{\ell}$ is of regular type.
Moreover, the restriction $\wt\G_{\ell}$ of $\wt\G$ to $\wt\G^{-1}(\ell)\subset\wt\p^{2}$ is a branched covering isomorphic to $\rho_{\ell}$. Since $\Delta_{\wt\G_{\ell}}=\Delta_{\F}\cap\ell$, by Proposition~\ref{Z} we deduce that  $\wt\G_{\ell}$ is of regular type if and only if for each $p\in\Delta_{\F}\cap\ell$ the ramification indices of all the preimages of $p$ by $\wt\G$ are equal, say to $\varrho>1$, but this is equivalent to the fact that $\mr{Tang}(\F,\check{p})=\beta(\wt\G^{-1}(p))\subset(\mc I_{\F}^{\varrho}\setminus\Sigma_{\F})\cup\Sigma_{\F}^{\varrho}$ thanks to Lemma~\ref{tangencia}.
\end{dem}

Finally we want to give a geometric characterization of the sets $\Sigma_{\F}^{\varrho}$. 
To this purpose, we introduce a last geometric ingredient: the \emph{polar curve} of $\F$ with respect to a point $p\in\p^{2}$, which is defined as
$\G^{-1}(\check{p})=\mr{Tang}(\F,R_{p})$, where $R_{p}$ is the radial vector field centered at $p$. We consider the following definition.
\begin{defin}
Let $\F$ be a foliation on $\p^{2}$ and let $s\in\Sigma_{\F}$ be a singular point.
We define
\begin{enumerate}[(a)]
\item 
the \emph{vanishing order} of $\F$ at $s$ as 
$$\nu_{s}:=\min\{k\ge 1\,:\, J_{s}^{k}X\neq 0\} $$ and 
the \emph{tangency order} of $\F$ at $s$ as
$$\tau_{s}:=\min\{k\ge \nu_{s}:\, \det(J_{s}^{k}X, R_{s})\neq 0\},$$
where $X$ is a saturated vector field defining $\F$, $J^{k}_{s}X$ is its $k$-jet at $s$ and $R_{s}$ is the radial vector field centered at $s$;
\item the \emph{characteristic order} of $\F$ at $s$ as
$$\chi_{s}:=\tau_{s}/\beta_{s}\in\mb Q_{>0},$$
where $\beta_{s}$ is the number of branches at $s$ of a generic polar curve of $\F$.
\end{enumerate}
\end{defin}
Notice that for each $s\in\Sigma_{\F}$ we have $1\le\beta_{s}\le\nu_{s}\le\tau_{s}\le d=\deg\F$ so that $\chi_{s}\ge 1$. This arithmetical invariant of the singularities is related with the subsets $\Sigma_{\F}^{\varrho}$ by the following result.

\begin{lema}\label{sigmad}
Let $\F$ be a degree $d$ foliation on $\p^{2}$ and fix $s\in\Sigma_{\F}$. Then
\begin{enumerate}[(i)]
\item $\chi_{s}>1\Leftrightarrow s\in\Sigma_{\F}^{\mr{ram}}\qquad$ and  $\qquad\sum\limits_{D\subset\mc E_{s}^{\mr{dom}}}\delta_{D}\rho_{D}=\tau_{s}$,
\item $s\in\Sigma_{\F}^{\varrho}\Rightarrow\chi_{s}=\varrho\qquad$ and $\qquad s\in\Sigma_{\F}^{d}\Leftrightarrow\chi_{s}=d$,
\item $\Sigma_{\F}^{\mr{ram}}=\Sigma_{\F}^{d}\Leftrightarrow\chi_{s}\in\{1,d\}$ for all $s\in\Sigma_{\F}$,
\item $\Sigma_{\F}^{\mr{ram}}=\bigcup\limits_{1<\varrho|d}\Sigma_{\F}^{\varrho}\Rightarrow\chi_{s}\in\N$ and $\chi_{s}|d$ for all $s\in\Sigma_{\F}$.
\end{enumerate}
\end{lema}

\begin{dem}
(i) If $q\in\check{s}\subset\pd^{2}$ is generic then $\mc E_{s}\cap\wt\G^{-1}(q)=\{p_{1},\ldots,p_{\beta_{s}}\}$ and each point $p_{i}$ has a ramification index $\varrho_{i}\ge 1$ satisfying the relation $\sum_{i=1}^{\beta_{s}}\varrho_{i}=\tau_{s}$. Hence $\beta_{s}=\tau_{s}$ if and only if $s\notin\Sigma_{\F}^{\mr{ram}}$.
Moreover, for each irreducible component $D$ of $\mc E_{s}^{\mr{dom}}$ all the points of $\wt\G_{|D}^{-1}(q)=\{p_{i_{1}},\ldots,p_{i_{\delta_{D}}}\}$ share the same ramification index $\varrho_{D}$.
(ii) If $s\in\Sigma_{\F}^{\varrho}$ then $\varrho_{i}=\varrho$ for $i=1,\ldots,\beta_{s}$ and consequently $\varrho\beta_{s}=\tau_{s}$. The converse is also true for $\varrho=d$ because necessarily $\beta_{s}=1$ in that case.
Assertions (iii) and (iv) follow by (i) and~(ii). 
\end{dem}

We can not expect to obtain a fully characterization of Galois foliations only in local terms as it is explained in  Remark~\ref{global-local}. Nevertheless 
from Theorem~\ref{FolGal} and Lemma~\ref{sigmad}  we can deduce different conditions that are necessary or sufficient, using only purely local arithmetic invariants, and which become equivalent in the case of prime degree:

\begin{teo}\label{necsufloc}
Let $\F$ be a degree $d>0$ foliation on $\p^{2}$ and consider the following assertions:
\begin{enumerate}[(1)]
\item 
$\G_{\F}$ has extremal type, or equivalently, $\mc I_{\F}^{\mr{tr}}=(d-1)\mc I_{\F}^{d}$ and $\Sigma_{\F}^{\mr{ram}}=\Sigma_{\F}^{d}$ 
\item $\F$ is Galois;
\item $\mc I_{\F}^{\mr{tr}}=\sum\limits_{\varrho|d}(\varrho-1)\mc I_{\F}^{\varrho}$ and $\Sigma_{\F}^{\mr{ram}}=\bigcup\limits_{1<\varrho|d}\Sigma_{\F}^{\varrho}$;
\item $\mc I_{\F}^{\mr{tr}}=\sum\limits_{\varrho|d}(\varrho-1)\mc I_{\F}^{\varrho}$ and $\chi_{s}\in\N$ divides $d$ for each $s\in\Sigma_{\F}$.
\end{enumerate}
Then $(1)\Rightarrow(2)\Rightarrow(3)\Rightarrow(4)$. Moreover, $(4)\Rightarrow(1)$ when $d$ is prime.
\end{teo}
The following example is an application of the above theorem.

\begin{ex}\label{exg1b}
Let $\F$ be the degree $3$ foliation given by the vector field
$$(y+x^{2})\partial_{x}-\frac{x^{3}}{3}\partial_{y}.$$   
Using formula~(\ref{infafin}), it can be easily checked that $\It=\{x^{2}(3y+2x^{2})^{2}=0\}$.
On the other hand, we have that $\Sigma_{\F}=\{s_{1}=[0,0,1],s_{2}=[0,1,0]\}$
and it can be checked that $\chi_{s_{i}}=1$ for $i=1,2$.
By applying Theorem~\ref{necsufloc} we deduce that $\F$ is Galois of extremal type.
Since $\G$ maps $x=0$ into $p=0$ and $3y+2x^{2}=0$ into $3q-p^{2}=0$, its weighted branching type is $\mf b_{\F}^{w}=3(3)_{1}$, so that the genus of its generic polar is $\mf g_{\F}=1$. 
\end{ex}

\begin{obs}
Let $\F$ be a Galois foliation of degree $3$ and genus $\mf g_{\F}=1$. 
For each generic $\ell\in\pd^{2}$ the dimensional reduction branched covering $\G_{\ell}:X_{\ell}:=\G^{-1}(\ell)^{\nu}\to\G^{-1}(\ell)\to\ell\simeq\p^{1}$ is Galois with source an elliptic curve. Then $\mr{Deck}(\G_{\ell})$ does not contain any element of $\mr{Aut}_{0}(X_{\ell})\simeq X_{\ell}$ acting on $X_{\ell}$ by translations because the ramification locus must remain fixed. Hence for all $\ell$ the elliptic curve $X_{\ell}$ is hexagonal and its $j$-invariant is constant equal to zero. In particular, we obtain the isotriviality of the polars in Examples~\ref{exg1} and~\ref{exg1b}.
\end{obs}

\begin{obs}\label{convexnogalois}
If $\F$ is a degree $d\ge 3$ Galois foliation on~$\p^{2}$ whose Gauss map is  of extremal type  then  $\mc I_{\F}^{\mr{inv}}\neq\emptyset$ when $d\neq 4$. 
Indeed, if $\mc I_{\F}^{\mr{inv}}=\emptyset$ then $\mc I_{\F}=(d-1)\mc I_{\F}^{d}$, which implies that $3d=(d-1)k$ and $(k,3)=\ell(d,d-1)$ for some $\ell\in\Z$, because $\gcd(d,d-1)=1$. Hence $\ell=1$ and $d=4$.
\end{obs}

The Galois character of a foliation is encoded in the sets $\mc I_{\F}^{\mr{tr}}$ and $\Sigma_{\F}^{\mr{ram}}$. The following example shows that the two elements are relevant.
\begin{ex}
The degree $4$ foliation~$\F$ given by the vector field
$$(y^{2}+x^{3})x\partial_{x}+(\frac{\zeta}{6}y^{2}+4x^{3})\zeta y\partial_{y},\qquad\zeta=2\pm i\sqrt{2},$$
is not Galois because it has $\mc I_{\F}^{\mr{tr}}=3\mc I_{\F}^{4}$, $s=(0,0)\in\Sigma_{\F}$ and $2=\beta_{s}<\nu_{s}=\tau_{s}=3$ so that $\chi_{s}=\frac{3}{2}\notin\N$.
\end{ex}

A natural class of foliations to study is that of \emph{convex foliations}, that is those for which $\mc I_{\F}^{\mr{tr}}=\emptyset$. In that case we have:

\begin{lema}
If $\F$ is a degree $d$ convex Galois foliation then $\tau_{s}=d$ for each $s\in\Sigma_{\F}^{\mr{ram}}$. 
\end{lema}
\begin{dem}
We can choose a line $\ell\subset\p^{2}$ such that $\ell\cap\Sigma_{\F}^{\mr{ram}}=\{s\}$. By Theorem~\ref{FolGal},  $\mr{Tang}(\F,\ell)\subset\Sigma_{\F}^{\mr{ram}}\cap\ell=\{s\}$ and consequently $\tau_{s}=d$.
\end{dem}

\begin{ex}\label{convex}
A class of convex foliations are those for which $\mc I_{\F}=\mc I_{\F}^{\mr{inv}}$ is reduced.
In \cite{MP} the authors study some of them that we list here:
\begin{enumerate}[(i)]
\item The infinite family of Fermat foliations defined by the vector fields \mbox{$(x^{d}-\varepsilon x)\partial_{x}+(y^{d}-\varepsilon y)\partial_{y}$} with $\varepsilon\neq 0$ and $d\ge 3$. 
\item The degree $4$ Hessian pencil of cubics $\mc H_{4}$ given by the rational first integral $\frac{x^{3}+y^{3}+z^{3}}{xyz}$.
\item The degree $5$ Hilbert modular foliation $\mc H_{5}$ given by the vector field
\[
  (x^2-1)(x^2- (\sqrt{5}-2)^2)(x+\sqrt{5}y)\partial_{x} + (y^2-1)(y^2- (\sqrt{5}-2)^2)(y+\sqrt{5}x)\partial_{y} \, .
\]
\item The degree $7$ foliation $\mc H_{7}$  invariant by the Hessian group given by the vector field
\[
            (x^3-1)(x^3+7 y^3+1) x \partial_x + (y^3-1)(y^3+7x^3+1) y \partial_y\,.
\]
\end{enumerate}
All these foliations have (radial) singularities $s\in\Sigma_{\F}$ with $\nu_{s}<\tau_{s}<\deg\F$. Hence $\chi_{s}>1$ and $s\in\Sigma_{\F}^{\mr{ram}}$. By Lemma~\ref{convex}, none of these foliations is Galois. However, the degenerations $x^{d}\partial_{x}+y^{d}\partial_{y}$ of Fermat foliations, obtained by taking $\varepsilon=0$, are convex and Galois, as we have seen in Proposition~\ref{2.4}.
\end{ex}

\subsection{Homogeneous Galois  foliations and their deformations}\label{homog}
In \cite{CD} the authors are interested in describing the algebraic set $\mb G_{3}$ of degree three Galois foliations. Due to the difficulty of problem in its full generality, they focus on the homogeneous case, for which they dispose of a particularly simple generic normal form depending only on $4$ complex parameters:
\begin{equation}\label{nf}\F_{\alpha;\lambda,\mu,\nu}\,:\,\frac{dx}{x}+\lambda\frac{dy}{y}+\mu\frac{dy-dx}{y-x}+\nu\frac{dy-\alpha dx}{y-\alpha x}=0,
\end{equation}
with $\lambda\mu\nu(1+\lambda+\mu+\nu)\alpha(\alpha-1)\neq 0$.
They prove some partial results about the subset of  $(\alpha;\lambda,\mu,\nu)\in \C^{4}$ such that the foliation $\F_{\alpha;\lambda,\mu,\nu}$ is Galois.
This subsection is devoted to describe completely the set  of homogeneous Galois foliations of arbitrary degree as well as its geometry.\\

Let $\mb H_{d}$ be the set of degree $d$  homogeneous foliations given by saturated vector fields
 $A(x,y)\partial_{x}+B(x,y)\partial_{y}$. It is the Zariski open subset of $\p(\C_{d}[x,y]^{\oplus 2})\simeq\p^{2d+1}$ considered in Example~\ref{Rd}.
The left-right actions of $\mr{PSL}_{2}(\C)$ on the set of rational functions induce a natural action $\varphi$ of $\mr{PSL}_{2}(\C)\times\mr{PSL}_{2}(\C)$ on $\mb H_{d}$ by means~of
$$\varphi([\alpha_{ij}],[\beta_{ij}],[A_{1},A_{2}])=[\beta_{11} A_{1}^{\alpha}+\beta_{12} A_{2}^{\alpha},\beta_{21} A_{1}^{\alpha}+\beta_{22} A_{2}^{\alpha}],$$
where
$$ A_{i}^{\alpha}(x,y)=A_{i}(\alpha_{11}x+\alpha_{12}y,\alpha_{21}x+\alpha_{22}y).$$

\begin{teo}\label{clas-homog}
The irreducible components of the Zariski closed subset $\mr{Gal}(\mb H_{d})=\mb H_{d}\cap\mb G_{d}$ of $\mb H_{d}$ are smooth unirational varieties which consist of the orbits by $\varphi:\mr{PSL}_{2}(\C)\times\mr{PSL}_{2}(\C)\times\mb H_{d}\to\mb H_{d}$
 of the foliations
\begin{enumerate}[(1)]
\item $x^{d}\partial_{x}+y^{d}\partial_{y}$ for every $d$,
\item $(x^{n}+y^{n})^{2}\partial_{x}+(x^{n}-y^{n})^{2}\partial_{y}$ if $d=2n$ is even,
\item $(x^{4}+2i\sqrt{3}x^{2}y^{2}+y^{4})^{3}\partial_{x}+(x^{4}-2i\sqrt{3}x^{2}y^{2}+y^{4})^{3}\partial_{y}$ if $d=12$, 
\item $(x^{8}+14x^{4}y^{4}+y^{8})^{3}\partial_{x}+(xy(x^{4}-y^{4}))^{4}\partial_{y}$ if $d=24$,
\item $\scriptstyle{(x^{20}-228x^{15}y^{5}+494x^{10}y^{10}+228x^{5}y^{15}+y^{20})^{3}\partial_{x}+(xy(x^{10}+11x^{5}y^{5}-y^{10}))^{5}\partial_{y}}$ if $d=60$.
\end{enumerate} 
The first one is $5$-dimensional and the rest are $6$-dimensional. Each irreducible component corresponds to a different Galois group: 
cyclic, dihedral, tetrahedral, octahedral and icosahedral.
\end{teo}

\begin{dem}
Since every homogeneous foliation $\F\leftrightarrow[A,B]\in\mb H_{d}$ on $\p^{2}$ is invariant by homotheties we can apply Proposition~\ref{AutF} to deduce that $\F$ is Galois if and only if $\wh{\G}_{\F}=[A,B]:\p^{1}\to\p^{1}$ is Galois. By Theorem~\ref{klein} $[A,B]$ is Galois if and only if it is LR-equivalent to one of  the five stated models. Finally, the LR-equivalence in $\wh{\G}_{\F}$ translates to the action $\varphi$ on $\mb H_{d}$. 
Since the group $G=\mr{PSL}_{2}(\C)\times\mr{PSL}_{2}(\C)$ is an irreducible rational quasi-projective variety,  we deduce that its orbits on $\mb H_{d}$, which are isomorphic to $G/H$, where $H$ is the corresponding isotropy subgroup of $G$, are smooth unirational varieties.
The  assertion about the dimension in (1)  follows from an explicit computation 
of the differential of the map $\varphi([\alpha_{ij},\beta_{ij}],[x^{d},y^{d}])$ at the identity, which is
$$[(\beta_{{11}}+\alpha_{{11}}d)x^{d}+{x}^{d-1}y\alpha_{{12}}d
+\beta_{{12}}{y}^{d},\beta_{{21}}{x}^{d}+{y}^{d
-1}x\alpha_{{21}}d+(\beta_{{22}}+\alpha_{{22}}d)y^{d}].
$$
Its kernel is a $1$-dimensional subspace of $\mf{sl}_{2}(\C)\times\mf{sl}_{2}(\C)$.  Analogous computations can be made in the cases (2)-(5).
\end{dem}

From Theorems~\ref{clas-homog} and~\ref{TC} we deduce the following result:

\begin{cor} The Zariski closed subset $\mb G_{2n}$ of $\mb F_{2n}$ has at least two disjoint irreducible components and $\mb G_{12}$, $\mb G_{24}$ and $\mb G_{60}$ have at least three  disjoint irreducible components.
\end{cor}

Notice that, for each degree $d$, the first component of $\mb H_{d}\cap\mb G_{d}$ considered in the above Theorem consists of the homogeneous foliations appearing in Proposition~\ref{2.4}.
In addition, we can write the deck transformations of $\mc G$ in terms of 
 $\wh\tau(z)\in\mr{Deck}(B(1,z)/A(1,z))\subset\mr{PSL}_{2}(\C)$ in the following way
\begin{equation}\label{tauhat}
\tau(x,y)=\frac{A(x,y)y-B(x,y)x}{A(x,y)\wh\tau(y/x)-B(x,y)}(1,\wh\tau(y/x)).
\end{equation}

The classification of homogeneous Galois foliations given by Theorem~\ref{clas-homog} can be used to obtain a negative test for proving that a given foliation on~$\p^{2}$ is not Galois. It also provides (see Proposition~\ref{singgalois} below)  restrictions to either the type of the singularities of Galois foliations or the finite subgroups of $\mr{Bir}(\p^{2})$ that can occur as Galois groups of foliations on $\p^{2}$. For a general account on the finite subgroups of $\mr{Bir}(\p^{2})$ we refer to \cite{Dol}. 
Notice that Theorem~\ref{Namba-existencia} asserts that every finite group $G$ occurs as the Galois group  of a Galois branched covering $f:X\to Y$ but it does not give any indication about those that can be realized with rational total space $X$.

Let $\F$ be a foliation on~$\p^{2}$, for each singularity $s\in\Sigma_{\F}$ and each $\F$-invariant line $\ell\subset\mc I_{\F}^{\mr{inv}}$ we consider the homogeneous foliations $\F_{s}$ and $\F_{\ell}$ defined respectively by:
\begin{enumerate}[$\bullet$] 
\item $\F_{s}$ is the saturation of the first non-zero jet  of a vector field defining $\F$ at~$s$,
\item $\F_{\ell}$ is the saturation of the top degree homogeneous part of a vector field defining~$\F$ in the affine chart $\p^{2}\setminus\ell$.
\end{enumerate}
Notice $\F_{s}$ and $\F_{\ell}$ are homogeneous foliations on $\p^{2}$. Therefore, if they are Galois their deck transformation group are of Klein type, that is, appearing in the list given in Theorem~\ref{klein}. The relation between the foliations $\F$, $\F_{s}$ and $\F_{\ell}$ is given by the following statement.

\begin{prop}\label{singgalois}
Let $\F$ be a Galois foliation on $\p^{2}$. For each $s\in\Sigma_{\F}$ and $\ell\subset\mc I_{\F}^{\mr{inv}}$ we have that
\begin{enumerate}[(1)]
\item if $\deg\F_{s}>0$ the homogeneous foliation $\F_{s}$ is Galois; moreover if the exceptional divisor $D_{s}$ obtained blowing up once the point $s$ is not a ramification component, then the Klein type deck transformation group of~$\F_{s}$ injects into the deck transformation group of $\F$;
\item if $\deg\F_{\ell}=\deg\F$ then the homogeneous foliation $\F_{\ell}$ is also Galois.
\end{enumerate}
\end{prop}
\begin{dem}
We obtain assertion~(1) by applying Proposition~\ref{1} to the irreducible component $D_{s}\subset\wt\G^{-1}(\check{s})\subset\wt\p^{2}$ and noting that
the restrictions of $\wt\G=\wt\G_{\F}$ and $\wt\G_{\F_{s}}$  to $D_{s}\subset\wt\p^{2}$ coincide.
Assertion~(2) follows from the fact that $\mb G_{d}$ is closed by noting that 
$\F_{\ell}=\lim\limits_{\varepsilon\to\infty}h_{\varepsilon}^{*}\F$, where $h_{\varepsilon}\in\mr{PSL}_{3}(\C)$ is given by $h_{\varepsilon}(x,y)=(\varepsilon x,\varepsilon y)$ in the affine chart $\p^{2}\setminus\ell$.
\end{dem}

Motivated by Theorem~\ref{clas-homog} and Proposition~\ref{2.4} we consider the following family of deformations of a homogeneous foliation.

\begin{defin}\label{ELR}
Let $\F\in\mb H_{d}$ be a homogeneous foliation given by a saturated homogeneous vector field $X=A(x,y)\partial_{x}+B(x,y)\partial_{y}$. For every $\C$-linearly independent polynomials $u,v\in\C[x,y]$ of degrees $\le 1$, and every linearly independent vectors $(\alpha,\gamma,\lambda),(\beta,\delta,\mu)\in\C^{3}$
 we consider the \emph{extended left-right deformation} (ELR in short) of $\F$ as the family of foliations given by the vector fields
$$(\alpha A+\beta B)(u,v)\partial_{x}+(\gamma A+\delta B)(u,v)\partial_{y}+(\lambda A+\mu B)(u,v)(x\partial_{x}+y\partial_{y}).$$
\end{defin}

The proof of Proposition~\ref{2.4} shows that if $\F$ is a Galois homogeneous foliation then every element of its ELR-deformation is Galois with the same weighted branching type as~$\F$.
The family of vector fields considered in Proposition~\ref{2.4} consists in the ELR-deformation of the homogeneous Galois foliation defined by $x^{d}\partial_{x}+y^{d}\partial_{y}$. One can also made explicit the ELR-deformation of each homogeneous Galois foliation given in Theorem~\ref{clas-homog}, obtaining, by using formula~(\ref{tauhat}), explicit  continuous deformations of faithful representations of the triangular groups $C_{n}$, $D_{n}$, $A_{4}$, $S_{4}$ and $A_{5}$ into the Cremona group $\mr{Bir}(\p^{2})$, whose images are not contained in $\mr{PSL}_{3}(\C)$.

\begin{ex}\label{cremona}
The above considerations applied to the homogeneous foliation given by the vector field
$A(x,y)\partial_{x}+B(x,y)\partial_{y}$ with  
$$A(x,y)=(x^{4}+2i\sqrt{3}x^{2}y^{2}+y^{4})^{3}\quad\text{and}\quad B(x,y)=(x^{4}-2i\sqrt{3}x^{2}y^{2}+y^{4})^{3}$$
allows to embed 
its Galois group $A_{4}=\langle\sigma,\tau\,|\,\sigma^{2}=\tau^{3}=(\sigma\tau)^{3}=1\rangle$ 
 into the Cremona group $\mr{Bir}(\p^{2})$ by means of
$\sigma(x,y)=(-x,y)$ and $$\tau(x,y)=\frac{(\alpha A+\beta B)y-(\gamma A+\delta B)x}{(\alpha A+\beta B)(y+ix)-(\gamma A+\delta B)(y-ix)}\left(y-ix,y+ix\right),$$
where $\left[\begin{array}{cc}\alpha & \beta\\ \gamma & \delta\end{array}\right]\in\mr{PSL}_{2}(\C)$.
\end{ex}

\begin{obs}\label{degen}
Every homogeneous foliations admits the infinitesimal symmetry $R=x\partial_{x}+y\partial_{y}$ but the general element of its ELR-deformation does not admit $R$ as infinitesimal symmetry any more. However, it can be checked that the set of all ELR-deformations of every homogeneous foliation contains the special subsets:
\begin{enumerate}[(a)]
\item $\{P(y)\partial_{y}+Q(y)(x\partial_{x}+y\partial_{y})\,|\, P,Q\in\C[y]\}$, obtained by taking $u,v\in \C[y]$ and $\alpha=\beta=0$ and admitting the infinitesimal symmetry $x\partial_{x}$,  
\item $\{P(y)\partial_{x}+Q(y)(x\partial_{x}+y\partial_{y})\,|\,P,Q\in\C[y]$, obtained by taking $u,v\in \C[y]$ and $\gamma=\delta=0$ and admitting the infinitesimal symmetry $y\partial_{x}$.
\end{enumerate}
\end{obs}

\subsection{Foliations with continuous automorphism group}\label{cont}
A natural class of foliations on $\p^2$ including homogeneous foliations is that of foliations $\F$ with a continuous group of automorphisms $\mr{Aut}(\F)\subset\mr{PSL}_{3}(\C)$.
After giving a classification of foliations in that class we establish a general criterion to decide whether they are Galois in terms of a suitable rational map $\p^1\to\p^1$.

Taking into account that every foliation of degree $1$ or $2$ is Galois,
we can assume that $\F$ has degree $\ge 3$.
Let $R\in\mr{Lie}(\mr{Aut}(\F))\subset\mf X(\p^{2})\simeq\mf{sl}_{3}(\C)$ be a non-trivial infinitesimal automorphism of $\F$.
There are four possible Jordan form types for the  traceless matrix associated to $R$:
\begin{align*}
&\text{(a)}\quad \left(\begin{array}{ccc} \alpha& 0 & 0 \\ 0 &\beta & 0\\ 0& 0 &-(\alpha+\beta)\end{array}\right)\qquad
&\text{(b)}&\quad \left(\begin{array}{ccc} 0 &1 & 0 \\ 0 & 0 & 0\\ 0 & 0 & 0\end{array}\right)
\\
&\text{(c)}\quad
\left(\begin{array}{ccc} 0 & 1 & 0 \\ 0 & 0 & 1 \\ 0 &0 &0\end{array}\right)\qquad
&\text{(d)}&\quad  \left(\begin{array}{ccc} 1 & 1 & 0\\ 0 &1  &0 \\ 0 & 0 & -2\end{array}\right)
\end{align*}
In a suitable affine chart $(x,y)$ the corresponding vector field  $R$ takes one of the following normal forms:
\begin{enumerate}[(a)]
\item $R=\alpha x\partial_{x}+\beta y\partial_{y}$ with $\alpha\in\C^{*}$ and $\beta\in\C$ and $\mr{Re}(\beta/\alpha)\ge 0$,
\item $R=y\partial_{x}$,
\item $R=y\partial_{x}+\partial_{y}$,
\item $R=(x+y)\partial_{x}+y\partial_{y}$.
\end{enumerate}

Let $X=A(x,y)\partial_{x}+B(x,y)\partial_{y}$ be a \emph{saturated} polynomial vector field defining $\F$. The fact that $R\in\mr{Lie}(\mr{Aut}(\F))$ translates into the relation
\begin{equation}\label{Lie}
L_{R}X=\varepsilon\,X,
\end{equation}
for some rational function $\varepsilon\in\C(x,y)$.  Since 
the poles of $\varepsilon$ are contained in the zeroes of the coefficients of
$X$ and that vector field is saturated we see that $\varepsilon\in\C[x,y]$. Finally, using that $\deg R=1$ we deduce that $\varepsilon$ must be constant. 
The following result describes the foliations of degree $\ge 2$ having a continuous automorphism group.
\begin{prop}\label{descrip}
Let $X=A(x,y)\partial_{x}+B(x,y)\partial_{y}$ be a saturated polynomial vector field of degree $\ge 2$ satisfying 
$L_{R}X=\varepsilon X$ for some $R\in\mf X(\p^{2})$ in the precedent list  (a)-(d) of normal forms and for $\varepsilon\in\C$. 
\begin{enumerate}[(a)]
\item If $R=\alpha x\partial_{x}+\beta y\partial_{y}$ then $\beta/\alpha\in\mb Q$, so that we can assume that $\alpha,\beta\in\mb Z^{+}$ are coprime, $\varepsilon\in\Z\alpha+\Z\beta$ and 
$$A(x,y)=\sum\limits_{\alpha i+\beta j=\varepsilon+\alpha}a_{ij}x^{i}y^{j}\quad\text{and}\quad B(x,y)=\sum\limits_{\alpha i+\beta j=\varepsilon+\beta}b_{ij}x^{i}y^{j}$$
are quasi-homogenous polynomials of weights $(\alpha,\beta)$. 
\item If $R=y\partial_{x}$ then $\varepsilon=0$ and 
$X=P(y)\partial_{x}+Q(y)(x\partial_{x}+y\partial_{y})$
for some coprime polynomials $P,Q\in\C[y]$.
\item If $R=y\partial_{x}+\partial_{y}$ then $\varepsilon=0$ and
$X=P(y^{2}-2x)(y\partial_{x}+\partial_{y})+Q(y^{2}-2x)\partial_{x}$
for some coprime polynomials $P,Q\in\C[z]$.
\end{enumerate}
In addition,
\begin{enumerate}[(d)]
\item  if $R=(x+y)\partial_{x}+y\partial_{y}$, relation $L_{R}X=\varepsilon X$ does not hold for any saturated polynomial vector field $X$ of degree $\ge 2$.
\end{enumerate}
\end{prop}

\begin{dem}
\begin{enumerate}[(a)]
\item Writing $A=\sum a_{ij} x^{i}y^{j}$ and $B=\sum b_{ij}x^{i}y^{j}$, if $L_{R}X=\varepsilon X$ then $$\left(\sum a_{ij}(\alpha i+\beta j-\alpha-\varepsilon)x^{i}y^{j}\right)\partial_{x}+\left(\sum b_{ij}(\alpha i+\beta j-\beta-\varepsilon)x^{i}y^{j}\right)\partial_{y}=0,$$ leading to the claimed form of $A$ and $B$. 
It is not difficult to see that if $\beta/\alpha\not\in\mb Q$ and $\deg X\ge 2$ then $X$ cannot be saturated. 
\item If $R=y\partial_{x}$ then 
$L_{R}X-\varepsilon X=(y\partial_{x}A-B-\varepsilon A)\partial_{x}+(y\partial_{x}B-\varepsilon B)\partial_{y}=0$ implies that
$B=e^{\frac{\varepsilon x}{y}}\bar Q(y)\in \C[x,y]$. Hence $\varepsilon=0$ and $B=\bar Q\in\C[y]$. From the $\partial_{x}$-component of $L_{R}X-\varepsilon X=0$ we obtain that
$A=\frac{\bar Q(y)x}{y}+P(y)\in\C[x,y]$. Thus, $\bar Q(y)=yQ(y)$ for some $Q\in\C[y]$.
\item If $R=y\partial_{x}+\partial_{y}$ then 
$$
L_{R}X-\varepsilon X=(y\partial_{x}A+\partial_{y}A-B-\varepsilon A)\partial_{y}+(y\partial_{x}B+\partial_{y}B-\varepsilon B)\partial_{y}=0
$$ 
implies that $B=e^{\varepsilon y}P(y^{2}-2x)$ and necessarily $\varepsilon=0$.  From the $\partial_{x}$-component of $L_{R}X-\varepsilon X=0$ we obtain that $A(x,y)=yP(y-x^{2})+Q(y^{2}-x)$ for some polynomials $P,Q\in\C[z]$.
\item If $R=(x+y)\partial_{x}+y\partial_{y}$ and $X=\sum_{n\ge 0} X_{n}$ with $X_{n}=A_{n}\partial_{x}+B_{n}\partial_{y}$ homogeneous of degree $n$, then the degree $n$ homogeneous part of $L_{R}X-\varepsilon X$ is 
\begin{eqnarray*}
0&=&L_{R}X_{n}-\varepsilon X_{n}=((x+y)\partial_{x}A_{n}+y\partial_{y}A_{n}-(\varepsilon+1)A_{n}-B_{n})\partial_{x}+\\ & &
((x+y)\partial_{x}B_{n}+y\partial_{y}B_{n}-(\varepsilon+1)B_{n})\partial_{y}\\
&=&(y\partial_{x}A_{n}-(\varepsilon+1-n)A_{n}-B_{n})\partial_{x}+(y\partial_{x}B_{n}-(\varepsilon+1-n)B_{n})\partial_{y}.
\end{eqnarray*}
As before, looking at   the $\partial_{y}$-component we deduce  that if $B_{n}\neq 0$ then $B_{n}=e^{\frac{(\varepsilon+1-n)x}{y}}Q(y)\in\C[x,y]$. Hence $\varepsilon=n-1$ and $B(y)=Q(y)=q y^{n}$ for some $q\in\C$. Substituting $B$ in the $\partial_{y}$-component of $L_{R}X_{n}-\varepsilon X_{n}$ we easily deduce that $A(x,y)=q x y^{n-1}+p y^{n}$ for some $p\in\C$. Since there is at most one $n\in\Z_{+}$ such that $\varepsilon=n-1$, we deduce that $X=X_{n}=y^{n-1}((py+qx)\partial_{x}+qy\partial_{y})$ is not saturated because $\deg X=n\ge 2$.
\end{enumerate}
\end{dem}

\begin{prop}\label{ghat}
To every  foliation $\F$ on $\p^{2}$ admitting a continuous group of automorphisms we can associate a non-constant morphism $\wh{\mc G}:\p^{1}\to\p^{1}$ so that
$\mr{Deck}(\mc G)\simeq\mr{Deck}(\wh{\mc G})$. In particular,
 $\F$ is Galois $\Longleftrightarrow$  $\wh{\mc G}$ is Galois.
\end{prop} 
\begin{dem}
In cases (a), (b) and (c) the foliations defined by the vector fields $R$ and 
its dual $\check{R}$ admit explicit primitive rational first integrals $$\rho:\p^{2}\dashrightarrow\p^{1}\qquad\text{and}\qquad\check{\rho}:\pd^{2}\dashrightarrow\p^{1}$$ respectively, and rational sections 
$$\sigma:\p^{1}\dashrightarrow\p^{2}\qquad\text{and}\qquad\check{\sigma}:\p^{1}\dashrightarrow\pd^{2}$$
such that $\rho\circ\sigma=\check{\rho}\circ\check{\sigma}=\mr{Id}_{\p^{1}}$.
It can be easily checked that, in the affine charts 
considered above, these maps are given by
\begin{enumerate}[(a)]
\item $\rho(x,y)=y^{\alpha}/x^{\beta}$, $\sigma(z)=(z^{\gamma},z^{\delta})$,  $\check{\rho}(a,b)=b^{\alpha}/a^{\beta}$ and $\check{\sigma}(z)=(z^{\gamma},z^{\delta})$, where $\gamma,\delta\in\mb Z$ satisfy B\'ezout's relation $\alpha\delta-\beta\gamma=1$,
\item $\rho(x,y)=y$, $\sigma(z)=(0,z)$, $\check{\rho}(a,b)=a$ and $\check{\sigma}(z)=(z,0)$,
\item $\rho(x,y)=y^{2}-2x$, $\sigma(z)=(-z/2,0)$, $\check{\rho}(a,b)=\frac{b^{2}+2a}{a^{2}}$ and $\check{\sigma}(z)=(2/z,0)$.
\end{enumerate}
Moreover, the Gauss map of the foliation  given by the vector field $A(x,y)\partial_{x}+B(x,y)\partial_{y}$ is written as 
$$\mc G(x,y)=\left(\frac{-B(x,y)}{C(x,y)},\frac{A(x,y)}{C(x,y)}\right),\quad\text{with}\quad C(x,y)=yA(x,y)-xB(x,y).$$ 
Thus we obtain explicit expressions for the map  $\wh{\mc G}=\check{\rho}\circ\mc G\circ\sigma:\p^{1}\to\p^{1}$:
\begin{equation}
\label{Ghat}
\left\{\begin{array}{rcl}
\text{(a)} &\quad &\wh{\mc G}(z)=
A(z^{\gamma},z^{\delta})^{\alpha}(-B(z^{\gamma},z^{\delta}))^{-\beta}C(z^{\gamma},z^{\delta})^{\beta-\alpha},\\[1mm]
\text{(b)} & \quad & \wh{\mc G}(z)=-\frac{B(0,z)}{C(0,z)}=-\frac{Q(z)}{P(z)},\\[1mm]
\text{(c)} &\quad &\wh{\mc G}(z)=\frac{Q(z)^{2}-zP(z)^{2}}{P(z)^{2}}=\left(\frac{Q(z)}{P(z)}\right)^{2}-z,
\end{array}\right.
\end{equation}
where $A,B$ take the form given by Proposition~\ref{descrip} in each case.
Consequently, we can apply  Proposition~\ref{AutF} in order to conclude.
\end{dem}

Notice that all Galois foliations of this type have Galois group appearing in Klein's classification given by Theorem~\ref{klein}. 
This fact and Proposition~\ref{singgalois} motivate the following natural question:

\begin{question}
Are there Galois foliations on $\p^{2}$ whose Galois group is not of Klein type?
\end{question}

\begin{obs}\label{obs-homog}
If we set $\alpha=\beta=1$ in case (a), we obtain the class of homogeneous foliations studied in Subsection~\ref{homog}.
For every coprime  homogeneous polynomials $A,B$ in two variables of the same degree, the  homogeneous and type (b) foliations on $\p^{2}$ given respectively by the vector fields
$$A(x,y)\partial_{x}+B(x,y)\partial_{y}\quad\text{and}\quad A(1,y)\partial_{x}-B(1,y)(x\partial_{x}+y\partial_{y})$$
satisfy that the map $\wh{\G}$ induced by their Gauss map is $\wh{\G}=[A,B]:\p^{1}\to\p^{1}$. 
Moreover, in the homogeneous case we have
$\wh{\mc G}=[A,B]=\wt{\mc G}_{|D_{O}}$, where~$D_O$ is the exceptional divisor obtained after blowing up once the origin, and we recover Theorem~\ref{clas-homog} in an alternative way. As we have already pointed out in Remark~\ref{degen}, cases (a) with $\beta=0$ and~(b) can be thought as degenerations of homogeneous foliations.
\end{obs}

Despite the criterion provided by Proposition~\ref{ghat} for deciding whether a foliation with an infinitesimal symmetry is Galois and the explicit form of the rational map $\wh{\G}$ given in~(\ref{Ghat}), it is not easy to find new examples of Galois foliations admitting such a symmetry. This is due to the difficulty of recovering the coefficients $A$ and $B$ based only on the map $\wh{\G}$.
Explicit Galois examples of the quasi-homogeneous case (a) with $0<\alpha<\beta$ are the following:
\begin{enumerate}[$\bullet$]
\item The degree $d$ foliation $\F$ given by the vector field
$$x^{d+1}\partial_{x}+(y^{d}+x^{d}y)\partial_{y}$$ belongs to the Galois family of Proposition~\ref{2.4} and that it is quasi-homogeneous with weights $\alpha=d-1$ and $\beta=d$, $\mf b_{\F}^{w}=2(d)_{1}$ and $\mf g_{\F}=0$. Moreover, it can be checked that the foliation $\F$ is convex, i.e.  $\It=\emptyset$, and  $\Sigma_{\F}^{\mr{ram}}=\Sigma_{\F}=\{[0,0,1],[0,1,0]\}$.
\item The degree $3$ foliation $\F$ considered in Example~\ref{exg1} is quasi-homogeneous with weights $\alpha=2$ and $\beta=3$, $\mf b_{\F}^{w}=3(3)_{1}$ and $\mf g_{\F}=1$.
\item The degree $3$ foliation $\F$ considered in Example~\ref{exg1b} is quasi-homogeneous with weights $\alpha=1$ and $\beta=2$, $\mf b_{\F}^{w}=3(3)_{1}$ and $\mf g_{\F}=1$.
\end{enumerate}

\bibliographystyle{plain}

\end{document}